\documentclass[11pt]{amsart}

\usepackage[utf8]{inputenc} 	
\usepackage[T1]{fontenc}
\usepackage{lmodern}
\usepackage{yhmath}
\usepackage{cancel}

\usepackage{amsthm,amsmath,amssymb,amsbsy,bbm,mathrsfs,supertabular,
eurosym,graphicx,enumitem,xcolor}
\usepackage{mathtools}

\newtheoremstyle{BBstyle0}  {}{}{\itshape}{}{\bfseries}{}{6pt}{}
\newtheoremstyle{BBstyle1}  {3pt}{3pt}{\rmfamily}{}{\itshape}{: }{3pt}{}
\newtheoremstyle{BBstyle2}  {3pt}{3pt}{\itshape}{}{\bfseries\large}{}{0pt}{}
\newtheoremstyle{BBstyle3}  {}{}{\itshape}{}{\bfseries}{: }{3pt}{}
\newtheoremstyle{BBstyle4}  {}{}{\rmfamily}{}{\bfseries}{}{6pt}{}
  
\usepackage[normalem]{ulem} 
\usepackage[authoryear]{natbib}

\newtheorem{thm}{Theorem}
\newtheorem{lem}{Lemma}
\newtheorem{prop}{Proposition}

\newtheorem{cor}{Corollary}
\newtheorem{ass}{Assumption}

\theoremstyle{definition}
\newtheorem{exa}{Example}

\usepackage[english]{babel}

\usepackage{enumitem}
\usepackage{hyperref}  

\newcommand{\pa}[1]{\left({#1}\right)}
\newcommand{\norm}[1]{\left\|{#1}\right\|}
\newcommand{\cro}[1]{\left[{#1}\right]}
\newcommand{\ab}[1]{\left|{#1}\right|}
\newcommand{\ac}[1]{\left\{{#1}\right\}}

\newcommand{\Var}{\mathop{\rm Var}\nolimits}

\newcommand{\E}{{\mathbb{E}}}
 
\renewcommand{\L}{{\mathbb{L}}}
\newcommand{\N}{{\mathbb{N}}}
\renewcommand{\P}{{\mathbb{P}}}
\newcommand{\Q}{{\mathbb{Q}}} 
\newcommand{\R}{{\mathbb{R}}}

\newcommand{\Z}{{\mathbb{Z}}}

\newcommand{\sB}{{\mathscr{B}}}

\newcommand{\sL}{{\mathscr{L}}} 
\newcommand{\sM}{{\mathscr{M}}}

\newcommand{\sP}{{\mathscr{P}}}
\newcommand{\sQ}{{\mathscr{Q}}}

\newcommand{\sT}{{\mathscr{T}}}

\DeclareMathAlphabet{\mathscrbf}{OMS}{mdugm}{b}{n}

\newcommand{\sbB}{{\mathscrbf{B}}}

\newcommand{\sbE}{{\mathscrbf{E}}}

\newcommand{\sbM}{{\mathscrbf{M}}}

\newcommand{\sbP}{{\mathscrbf{P}}}

\newcommand{\cC}{{\mathcal{C}}}

\newcommand{\cE}{{\mathcal{E}}}
\newcommand{\cF}{{\mathcal{F}}}
\newcommand{\cG}{{\mathcal{G}}} 

\newcommand{\cI}{{\mathcal{I}}}

\newcommand{\cM}{{\mathcal{M}}}
\newcommand{\cN}{{\mathcal{N}}}

\newcommand{\cV}{{\mathcal{V}}}

\newcommand{\cX}{{\mathcal{X}}}

\newcommand{\gi}{{\mathbf{i}}}

\newcommand{\gk}{{\mathbf{k}}}

\newcommand{\gu}{{\mathbf{u}}}

\newcommand{\gw}{{\mathbf{w}}}
\newcommand{\gx}{{\mathbf{x}}}

\newcommand{\gE}{{\mathbf{E}}}

\newcommand{\gP}{{\mathbf{P}}}
\newcommand{\gQ}{{\mathbf{Q}}}

\newcommand{\gT}{{\mathbf{T}}}

\newcommand{\gZ}{{\mathbf{Z}}}

\newcommand{\bs}[1]{\boldsymbol{#1}}

\newcommand{\bsX}{{\bs{X}}}
 
\newcommand{\bsZ}{{\bs{Z}}}
\newcommand{\geps}{\bs{\varepsilon}}

\newcommand{\gtheta}{{\bs{\theta}}}

\newlist{lista}{enumerate}{1}
\setlist[lista,1]{label=\alph*),ref=\alph*)}

\newlist{listi}{enumerate}{1}
\setlist[listi,1]{label=(\roman*),ref=(\roman*),align=left}

\newcommand{\eref}[1]{(\ref{#1})}
\renewcommand{\ge}{\geqslant}
\renewcommand{\le}{\leqslant}
\newcommand{\1}{1\hskip-2.6pt{\rm l}}

\renewcommand{\>}{{\rangle}}
\newcommand{\scal}[2]{\langle #1,#2\rangle}

\newcommand{\etc}[1]{#1_1,\ldots,#1_n}

\newcommand{\st}{\strut}

\newcommand{\on}{^{\otimes n}}

\newcommand{\et}{^{\star}}
\newcommand{\eps}{{\varepsilon}}

\def\bst{\bs{\theta}}

\def\gell{{\bs{\ell}}}

\def\dis{{\rm d}}

\begin{document}
\title[]{Tests and estimation strategies associated to some loss functions}
\author{Yannick BARAUD}
\address{\parbox{\linewidth}{Department of Mathematics,\\
University of Luxembourg\\
Maison du nombre\\
6 avenue de la Fonte\\
L-4364 Esch-sur-Alzette\\
Grand Duchy of Luxembourg}}
\email{yannick.baraud@uni.lu}
\thanks{This project has received funding from the European Union's Horizon 2020 research and innovation programme under grant agreement N\textsuperscript{o} 811017}

\date{\today}

\begin{abstract}
We consider the problem of estimating the joint distribution of $n$ independent random variables. Given a loss function and a family of candidate probabilities, that we shall call {\em a model}, we aim at designing an estimator with values in our model that possesses good estimation {properties} not only when the distribution of the data belongs to the model but also when it lies close enough to it. The losses we have in mind are the total variation, Hellinger, Wasserstein and $\L_{p}$-distances to name a few. We show that the risk of our estimator can be bounded by the sum of an approximation term that accounts for the loss between the true distribution and the model and a complexity term that corresponds to the bound we would get if this distribution did belong to the model. Our results hold under mild assumptions on the true distribution of the data and are based on exponential deviation inequalities that are non-asymptotic and involve explicit constants. 
Interestingly, when the model reduces to two distinct probabilities, our procedure results in a robust test whose errors of first and second kinds only depend on the losses between the true distribution and the two tested probabilities. 
 \end{abstract}

\maketitle

\section{Introduction}
Observe $n$ independent random variables $X_{1},\ldots,X_{n}$ with values in a measured space $(E,\cE,\mu)$ and assume they are i.i.d.\ with common distribution $P\et$.  Consider now a loss function $\ell$ for evaluating the performance of an estimator $\widehat P$ of $P\et$. The loss $\ell$ is a nonnegative function defined on $\sP\times\sM$ for some suitable set $\sP$ containing the true distribution $P\et$ and a model $\sM$, i.e.\ a family of candidate probabilities for $P\et$, that should either contain $P\et$ or at least provide a suitable approximation of it. The purpose of the present paper is to design a generic method for estimating $P\et$ that takes into account our choice of $\ell$ and $\sM$ in order to {build} an estimator $\widehat P=\widehat P(X_{1},\ldots,X_{n})$ with values in $\sM$ that enjoys good estimation properties. Even though $\ell$ may not be a genuine distance (it may neither be symmetrical nor satisfy the triangle inequality), we shall interpret it as if it were: small values of $\ell(P\et,\widehat P)$ would mean $\widehat P$ is ``close'' to $P\et$ while large {values} of $\ell(P\et,\widehat P)$ would in contrast be understood as $\widehat P$ is ``far'' from it. Our aim is therefore to define $\widehat P$ {in} such a way that $\ell(P\et,\widehat P)$ be as close as possible to $\inf_{P\in\sM}\ell(P\et,P)=\ell(P\et,\sM)$.

This problem was solved for the Hellinger loss in Baraud {\em et al.}~\citeyearpar{MR3595933} and Baraud and Birg\'e~\citeyearpar{BarBir2018}. In order to give an account of {their} results, let us first recall that the squared Hellinger distance $h^{2}(P,Q)$ between two probabilities $P$ and $Q$ on $E$ is given by  the formulas
\begin{equation}
h^{2}(P,Q)=\frac{1}{2}{\int_{E}\pa{\sqrt{\frac{dP}{d\nu}}-\sqrt{\frac{dQ}{d\nu}}}^{2}d\nu=1-\int_{E}\sqrt{\frac{dP}{d\nu}\frac{dQ}{d\nu}}\;d\nu}
\label{eq-Hell1}
\end{equation}
where $\nu$ denotes an arbitrary measure on $(E,\cE)$ that dominates both $P$ and $Q$, the result being independent of the choice of $\nu$. The estimator $\widehat P(\bsX)$ which results from their procedure (named $\rho$-estimation) typically satisfies an inequality of the form
\begin{equation}
\E\cro{h^{2}\left(P\et,\widehat P(\bsX)\right)}\le C\left[\inf_{P\in\sM}h^{2}(P\et,P)+\frac{D_{n}(\sM)}{n}\right],
\label{eq-0}
\end{equation}
where $C$ is a positive numerical constant and $D_{n}(\sM)$ a complexity term that may depend on the number $n$ of observations and the dimension (in some sense) of the statistical model $\sM$. This inequality essentially says that the loss between $P\et$ and $\widehat P$ is not larger $CD_{n}(\sM)/n$ when $P\et$ belongs to the model $\sM$ and that this bound does not deteriorate too much as long as $\inf_{P\in\sM}h^{2}(P\et,P)$ remains sufficiently small. An interesting feature of~\eref{eq-0} lies in the following facts: the inequality (\ref{eq-0}) is true under very weak assumptions on both the statistical model $\sM$ and the underlying distribution $P\et$ and, in all cases we know, the quantity $D_{n}(\sM)/n$ turns out to be the best possible bound that can be achieved uniformly over the model $\sM$ (up to a possible logarithmic factor). 

In the present paper, we wish to extend this result to other losses, typically the total variation distance, the Wasserstein distance, the Kullback-Leibler divergence and the $\L_{j}$-distances with $j\in (1,+\infty]$, among others. Unfortunately, for most of these losses there is no  hope to establish a risk bound akin to \eref{eq-0} under weak assumptions on both $\sM$ and $P\et$ as it was the case for the Hellinger loss. If, for instance, $\sM$ is the set of all uniform distributions on $[\theta,\theta+1]$ with $\theta\in\R$ and $\ell$ is the Kullback-Leibler divergence, $\sup_{P\et\in\sM}\E[\ell(P\et,\widehat P)]=+\infty$ whatever the estimator $\widehat P\in\sM$ and  there is consequently no way of controlling the risk of $\widehat P$ as in~\eref{eq-0}. The situation does not improve much with the $\L_{j}$-loss which requires that $P\et$ and the probabilities in $\sM$ admit densities in $\L_{j}(E,\cE,\mu)$ with respect some given reference measure $\mu$, etc. In view of these disappointing observations, we see that specific assumptions {need to be made necessarily}. Throughout this paper our point of view is to make (possibly strong) assumptions on the model $\sM$, since it is chosen by the statistician, but to assume as little as possible on $P\et$ since it is unknown. 

Despite some differences, our approach shares {many} similarities with that developed for the Hellinger loss in Baraud {\em et al.}~\citeyearpar{MR3595933} and Baraud and Birg\'e~\citeyearpar{BarBir2018}. In particular, it is also based on the existence of suitable tests between probability ``balls'' (with a suitable sense when $\ell$ is not a genuine distance).  We shall give some general recipe on how to build such tests for the various loss functions we consider. We shall see that our general construction enables us to recover some well-known tests for some losses while for other losses, these tests are to our knowledge new. For the total variation distance, our testing procedure bears some similarities, though not exactly the same, with that proposed by Devroye and Lugosi (see Devroye and Lugosi~\citeyearpar{MR1843146}[Chapter 6]). Their approach is based on the seminal paper by Yatracos~\citeyearpar{MR790571}. For the Kullback-Leibler divergence we obtain the classical likelihood ratio test while, for the $\L_{2}$-distance, our approach results in the test based on the comparison of $\L_{2}$-contrast functions between the candidate densities. For the Wasserstein distance and the $\L_{j}$-losses with $j\ne 2$, the tests we get seem to be new in the literature.

Our estimation procedure results in a new class of estimators {that} we shall call {\em $\ell$-estimators} {and} which generalize $\rho$-estimators. A nice feature of our approach is that the study of these estimators can be made within a unified framework even though, in order to keep the present paper to a reasonable size, we shall mainly discuss the cases of the total variation and $\L_{2}$-losses. We shall see that $\ell$-estimators based on the total variation loss are robust and can even provide more robustness than $\rho$-estimators do. However, in some parametric models, they fail to reach the optimal rate of convergence while $\rho$-estimators are optimal (or nearly optimal) in all cases we know.  

In the present paper, we shall not address the problem of model selection (or adaptation) nor shall we discuss the computational issues that may result from the calculation of these estimators, even though for some specific models and losses we shall occasionally provide an explicit form of the $\ell$-estimator. Model selection will be addressed in a companion paper. The implementation of $\ell$-estimator would in general require to see what could be done to calculate them in each particular situation. 

We shall rather provide several examples for the purpose of illustrating the {performance} of $\ell$-estimators and contextualizing them within the literature.  For example, we shall show that, for the total variation loss, $\ell$-estimators can achieve a convergence rate which is faster than the usual $1/\sqrt{n}$ rate. Such results contrast with those obtained previously by Birg\'e~\citeyearpar{MR2219712} with $T$-estimators (see his Corollary~6) and Devroye and Lugosi~\citeyearpar{MR1843146} with skeleton estimators. Closer to our approach (for this {particular} loss) is that of  Gao {\em et al.}~\citeyearpar{gao2018robust}. In their paper, these authors proposed a robust estimation {method} of the mean of a Gaussian vector based on the observation of an $n$-sample. The estimator proposed by Gao {\em et al.} in this specific framework shares some similarities with ours. It is also obtained as the minimizer of the supremum of a random functional {defined on} a suitable class of functions.  However, our construction differs from theirs by the choices of the classes over which the supremum and infimum are computed.

%

We also address the problem of estimating a density on $\R^{d}$ with respect to the Lebesgue measure when the risk is based on the $\L_{2}$-loss and the target density is not necessarily bounded in sup-norm. We are only aware of very few results in this direction. Birg\'e and Massart~\citeyearpar{MR1653272} studied the performances of minimum $\L_{2}$-contrast estimators on linear spaces $V$. Their results, however, suffer from two limitations: the functions in $V$ are supported on a known compact set, say $[0,1]^{d}$, and $V$ is finite dimensional. Our theory allows us to relax these two restrictions and generalize their results to an infinite dimensional linear space $V$ of functions with possibly unbounded support ($\R^{d}$ typically). For a suitable choice of $V$ we shall derive a uniform risk bound over the class of all squared integrable densities that lie in a Besov space $B^{\alpha}_{s,\infty}(\R^{d})$ with parameters $s\ge 2$ and $\alpha>0$. This result is to our knowledge new and generalizes that obtained by Reynaud-Bouret {\em et al.}~\citeyearpar{MR2719482} on the real line when $s>2$ (we also refer to Reynaud-Bouret and Rivoirard~\citeyearpar{MR2645482} for a lower bound on the minimax risk). 

Finally, let us mention that the basic ideas that underline the construction of $\ell$-estimators bear similarities with those used in {\em Generative Adversarial Nets} (GAN). As described in  Goodfellow {\em et al.}~\citeyearpar{goodfellow2014generative}, GAN can be viewed as a minimax two player game. Given an $n$-sample  $X_{1},\ldots,X_{n}$ with distribution $P\et$ and a model $\sM$ for $P\et$,  the first player aims at designing an estimator $\widetilde P$ of $P\et$ with values in $\sM$ for which the second player will hardly be able to discriminate between a (fake) $n$-sample with distribution $\widetilde P$ and a true one with distribution $P\et$. In our case, we aim at designing an estimator $\widehat P\in\sM$ such that $\ell(P\et,\widehat P)$ is so small that there is no way to test (hence to discriminate)  between $P\et$ and $\widehat P$ from an $n$-sample with distribution $P\et$ or $\widehat P$. When $\ell$ is the Hellinger distance, it is well-known that this goal is achieved as soon as $h(P\et,\widehat P)$ is small compared to $1/\sqrt{n}$. 

Our paper is organized as follows. 
We first present the statistical framework as well as our main assumptions in Section~\ref{sect-2}.  {We actually} consider a more general framework than the one described in this introduction since we assume the observations $\etc{X}$ to be independent but not necessarily i.i.d. We also allow our model $\sM$ to contain finite and possibly signed measures, hence not only probabilities. Such models are useful when dealing with $\L_{j}$-losses. The heuristics underlying our approach is also described in Section~\ref{sect-2} as well as our main assumptions on the loss functions we use. The estimation procedure and the general results on the performance of $\ell$-estimators are presented in Section~\ref{sect-RB} and some consequences of these results for the Wasserstein and the $\L_{2}$-losses in Section~\ref{sect-exlEst1}. A uniform risk bound of the $\ell$-estimators for the $\L_{2}$-loss over Besov balls on $\R^{d}$ can also be found there. We then put a special emphasis on the total variation loss in Section~\ref{sect-TV}. In particular, we provide an illustration to the problem of estimating a non-increasing density on a half line for the $\L_{1}$-distance. The Hellinger and {Kullback-Leibler losses} are considered in Section~\ref{sect-HKL} while Section~\ref{sect-TVversusH} {will} be devoted to the comparison between $\rho$- and $\ell$-estimators for the total variation loss. As already mentioned, our procedure is based on the existence of a family of robust tests between two distinct probabilities. The performance of such tests being interesting {\em per se}, it {will} be studied in Section~\ref{sect-RobustTest} with a special emphasis on the cases of the total variation and $\L_{j}$-losses. Finally Section~\ref{sect-pfsth} is devoted to the proofs of the main theorems. The other proofs may be found in Section~\ref{sect-opfs}.
%
\section{The statistical framework and our main assumptions}\label{sect-2}
Throughout the paper, we assume that the observations $X_{1},\ldots,X_{n}$ are independent but not necessarily i.i.d.\ and denote by $P_{1}\et,\ldots,P_{n}\et$ their {marginal distributions}. However, in many cases, our statistical model is based on the assumption that the data are i.i.d., although this might not be true, and we shall analyze the behaviour of our estimator with respect to a possible departure from this {assumption of equidistribution}. 

We denote by $\sP$ a set of probabilities on $(E,\cE)$ that contains the marginal distributions $P_{1}\et,\ldots, P_{n}\et$, so that $\gP\et=\bigotimes_{i=1}^{n}P_{i}\et$ is the distribution
{of $\bsX=(X_{1},\ldots, X_{n})$} and
$\sbP=\left\{\gP=\bigotimes_{i=1}^{n}P_{i},\; P_{i}\in\sP\right\}$ the set of all product probabilities with marginals in $\sP$. In particular, $\gP\et$ belongs to $\sbP$.
For convenience, we identify  an element $\gP=\bigotimes_{i=1}^{n}P_{i}$ of $\sbP$ with the $n$-tuple $(P_{1},\ldots,P_{n})$. Depending on the context, we either write $\gP$ as a product of probabilities or as an $n$-tuple. {We use the notation $Y\sim S$ to say that the random variable $Y$ is distributed according to $S$ and when we write $\E[g(\bsX)]$, we assume that $\bsX\sim\gP\et$ while $\E_{S}[f(Y)]$ represents the expectation of $f(Y)$ when $Y\sim S$}. We use the same conventions for $\Var\left(\st g(\bsX)\right)$ and {$\Var_{S}\left(\st f(Y)\right)$}.

Beside these conventions, we use the following notations. For $x\in\R$ and $k>0$, $x_{+}^{k}=\max\{0,x\}^{k}$ and $x_{-}^{k}=\max\{0,-x\}^{k}$, $\text{sign}(x)=\1_{x>0}-\1_{x<0}$; for $x\in\R^{d}$, $|x|$ denotes the Euclidean norm of $x$ and $B(x,r)$ the closed Euclidean ball centered at $x$ with radius $r\ge 0$. Given a $\sigma$-finite measure $\mu$ on $(E,\cE)$ and $j\in [1,+\infty]$, we denote by $\sL_{j}(E,\mu)$ the set of measurable functions $f$ on $(E,\cE,\mu)$ such that $\norm{f}_{\mu,j}<+\infty$ with
%
\begin{align}
\norm{f}_{\mu,j}&=\pa{\int_{E}|f|^{j}d\mu}^{1/j}\quad \text{when $j\in [1,+\infty)$}\label{defL-jnorm}\\
\norm{f}_{\mu,\infty}&=\inf\{C>0, \ab{f}\le C\, \text{ $\mu$-a.e.}\}\quad \text{when $j=+\infty$}.\label{defL-inftynorm}
\end{align}
We associate to $\sL_{j}(E,\mu)$, the set $\L_{j}(E,\mu)$ of equivalent classes on which two functions that coincide for $\mu$-almost all $x\in E$ are indistinguishable. For a positive integer $d$, we write $\sL_{j}(\R^{d})$ and  $\L_{j}(\R^{d})$ for $\sL_{j}(E,\mu)$ and $\L_{j}(E,\mu)$ respectively when $E=\R^{d}$, $\cE$ is the Borel $\sigma$-algebra and $\mu=\lambda$ is the Lebesgue measure on $\R^{d}$. Finally, we denote by $\norm{f}_{\infty}$ the quantity $\sup_{x\in E}|f(x)|\in [0,+\infty]$. In particular, one should not confuse $\norm{f}_{\infty}$ with $\norm{f}_{\mu,\infty}$.  

\subsection{Models and losses\label{sect-2.2}}
As already mentioned, our strategy for estimating $\gP\et$ is based on models. This means that we assume to have at disposal a family $\sbM$ of elements of the form $(P_{1},\ldots,P_{n})$ where the $P_{i}$ are finite measures on $(E,\cE)$, possibly signed, which belong to some set $\sM$. In most cases, the $P_{i}$ are probabilities but it is sometimes convenient to consider signed measures of the form $p\cdot \mu$ where  $p$ is not necessarily a {probability density} but an element of $\sL_{j}(E,\mu)\cap\sL_{1}(E,\mu)$ for some $j>1$. 

In the {\em density setting}, i.e.\ when we believe that the observations $X_{1}\ldots,X_{n}$ are i.i.d., we consider a model $\sbM$ which corresponds to this belief, {i.e.\ of} the form $\{(P,\ldots,P),\; P\in\sM\}$ and {specify} $\sM$ only. In other statistical frameworks such as the regression one, the data may no longer be i.i.d.\ and our model $\sbM$ for $\gP\et=(P_{1}\et,\ldots,P_{n}\et)$ consists of elements of the form $(P_{1},\ldots,P_{n})$ with possibly different entries in $\sM$.

Throughout this paper we assume that $\sM$ (and therefore $\sbM\subset\sM^{n}$) is {\em at most countable} (which means finite or countable) in order to avoid measurability issues. Since the model $\sbM$ is only assumed to provide an approximation of $\gP\et$ and may not contain it, this condition is not very restrictive: most of the models that statisticians use are separable and can therefore be well approximated by countable subsets.

Since $\sM$ is countable, it is dominated and there exists a $\sigma$-finite measure $\mu$ on $(E,\cE)$ for which we may write any element $P\in\sM$ as $P=p\cdot \mu$ for some integrable function $p$ on $(E,\cE)$. Throughout the paper, we assume the measure $\mu$ associated to $\sM$ to be fixed once and for all and that the statistician has chosen for each $P\in\sM$ a convenient version $p\in\sL_{1}(E,\mu)$ of $dP/d\mu$. We systematically use the corresponding lower case letter to denote this density ($P=p\cdot\mu$, $Q=q\cdot \mu$, {etc.}). This construction results in a family of densities $\cM$ associated to $\sM$. {Sometimes we shall actually rather start} from a countable family $\cM$ of densities in $\sL_{1}(E,\mu)$ (which may not be probability densities) and {then define $\sM$} as the family of (possibly signed) finite measures with densities in $\cM$, i.e.\ $\sM=\{P=p\cdot \mu,\; p\in\cM\}$. 

Given the previous framework, the observation $\bsX$ and the model $\sbM$, we build an estimator $\widehat \gP=\widehat \gP(\bsX)$ of $\gP\et$ with values in $\sbM$ {and, to} evaluate its performance, we introduce a loss function $\ell$ defined on $\sP\times \sM$ with values in $\R_{+}$. In order to avoid trivialities, we assume that $\ell$ is not constant equal to 0. For $\gP=(P_{1},\ldots,P_{n})\in\sbP$ and $\gQ=(Q_{1},\ldots,Q_{n})\in\sbM$, we set
\begin{equation}
\gell(\gP,\gQ)=\sum_{i=1}^{n}\ell(P_{i},Q_{i})
\label{eq-Loss}
\end{equation}
and measure the quality of $\widehat \gP$ by the quantity $\gell(\gP\et,\widehat \gP)$. The smaller this quantity, the better the estimator. Since, by construction, $\widehat \gP\in\sbM$, $\gell(\gP\et,\widehat \gP)$ cannot be smaller than $\inf_{\gQ\in\sbM}\gell(\gP\et,\gQ)=\gell(\gP\et,\sbM)$ and the best we can expect is that $\gell(\gP\et,\widehat \gP)$ be close to $\gell(\gP\et,\sbM)$. Of special interest is the situation where $\gell(\gP\et,\sbM)=0$, which generalizes the case of $\gP\et\in\sbM$ ($\gP\et$ belongs to the model) and suggests to introduce the notations
{
\begin{equation}
\overline{\sM}=\{P\in\sP\,|\,\ell(P,\sM)=0\}\quad\text{and}\quad\overline{\sbM}=\{\gP\in\sbP\,|\,\gell(\gP,\sbM)=0\}.
\label{eq-overM}
\end{equation}

\subsection{Some heuristics\label{sect-2.0}}
To simplify the presentation of our heuristics, let us assume that the $X_{i}$ are truly i.i.d.\ with distribution $P\et$ and that the elements of $\sbM$ take the form $P\on$ with $P\in\sM$ so that $\gell(\gP,\gQ)=n\ell(P,Q)$ by (\ref{eq-Loss}). If $P\et$ were known, the loss function $\ell$ would provide an ordering between the elements of $\sM$ by saying that $P$ is better than $Q$ if $\ell(P\et,P)\le\ell(P\et,Q)$ and an ideal point in $\sM$ for estimating $P\et$ would be $\overline{P}\in\sM$ that satisfies $\ell(P\et,\overline{P})=\inf_{P\in\sM}\ell(P\et,P)$, whenever this point exists.  
Since $P\et$ is unknown, one cannot find $\overline{P}$. 

Assume nevertheless that we are able to approximate $\ell(P\et,P)-\ell(P\et,Q)$ by some statistic $T(\bsX,P,Q)$ with an error bounded by $\Delta$. We can use $T(\bsX,P,Q)$ for testing between $P$ and $Q$, deciding $P$ when $T(\bsX,P,Q)$ is negative and $Q$ otherwise. It results in a robust test since we do not assume that $P\et$ is either $P$ or $Q$ (and not even very close to any of them) and it decides correctly whenever $|\ell(P\et,P)-\ell(P\et,Q)|>\Delta$. {Varying $(P,Q)$ among all possible pairs of probabilities in $\sM^{2}$}, we obtain a family $\{T(\bsX,P,Q), (P,Q)\in\sM^{2}\}$ of robust tests which we can use to build an estimator of $P\et$, or rather of $\overline{P}$, as defined above.

Deriving an estimator from a family of robust tests is not a new problem and methods for that have been developed a long time ago by Le Cam~\citeyearpar{MR0334381} and then Birg\'e~\citeyearpar{MR722129}, more recently by Baraud~\citeyearpar{MR2834722} and then Baraud {\em et al.}~\citeyearpar{MR3595933}, and it is actually this last recipe that we shall use here. In Baraud {\em et al.}~\citeyearpar{MR3595933} it was used to handle the loss $\ell=h^{2}$ based on the Hellinger distance $h$ to build $\rho$-estimators. It worked because we could (approximately) express $h^{2}(P\et,P)-h^{2}(P\et,Q)$ as the expectation of $T(\bsX,P,Q)$ or, more precisely, view $T(\bsX,P,Q)$ as an empirical version of an approximation of $h^{2}(P\et,P)-h^{2}(P\et,Q)$, then use the properties of the corresponding empirical process indexed by $(P,Q)$ to build a suitable estimator. To mimic this construction, we need that similar arguments apply to our choice of $\ell$. We shall explain more precisely in Sections~\ref{sect-ExaTest} and \ref{sect-ExaTest3} what properties of the loss imply the assumptions that are needed for our proofs. As to the performance of the robust tests based on the sign of $T(\bsX,P,Q)$ that we mentioned above, it will be studied in Section~\ref{sect-RobustTest}.

\subsection{Loss functions\label{sect-2.3}}
Let us now provide the definitions of the various {loss functions} we consider in this paper.\vspace{2mm}\\
{\bf Total variation loss (TV-loss).} The {\em total variation distance} $\norm{P-Q}$ between two probabilities $P,Q$ on $(E,\cE)$ is defined as follows: 
%
\begin{equation}\label{eq-def-TV0}
\norm{P-Q}=\sup_{A\in\cE}[P(A)-Q(A)]=\frac{1}{2}\int_{E}\ab{\frac{dP}{d\nu}-\frac{dQ}{d\nu}}d\nu,
\end{equation}
where $\nu$ denotes an arbitrary measure that dominates both $P$ and $Q$.
The total variation loss is $\ell(P,Q)=\norm{P-Q}$. We {shall hereafter write TV for {\em total variation} and} TV-estimator for the $\ell$-estimator associated to this loss. \vspace{2mm}\\
{\bf Hellinger loss.} The Hellinger loss is related to the {\em Hellinger distance} $h$, which is defined by (\ref{eq-Hell1}), by $\ell(P,Q)=h^2(P,Q)$. We recall that the quantity $\rho(P,Q)=1-h^{2}(P,Q)$ is the Hellinger affinity between $P$ and $Q$. 
\vspace{2mm}\\
{\bf Kullback-Leibler loss (KL-loss).} The {\em Kullback-Leibler divergence} $K(P,Q)$ between two probabilities {$P=p\cdot\mu$} and $Q=q\cdot \mu$ on $(E,\cE)$ is defined as 
\begin{equation}
K(P,Q)=
\begin{cases}
&\int_{E}\log(p/q)\,p\,d\mu\quad \text{when $P\ll Q$}\\
&+\infty\quad \text{otherwise,}
\end{cases}
\label{eq-KLa}
\end{equation}
with the following conventions:
\[
\text{For }x\in E,\qquad\log\pa{\frac{p}{q}}(x)=
\begin{cases}
0 \quad \text{if $p(x)=q(x)=0$}\\
+\infty\quad \text{if $p(x)>0$ and $q(x)=0$}\\
-\infty\quad \text{if $p(x)=0$ and $q(x)>0$.}
\end{cases}
\]
In particular, $\exp\cro{\log(p(x)/q(x))}=p(x)/q(x)$ for all $x\in E$ with the conventions $0/0=1$ and $a/0=+\infty$ for all $a>0$. {The KL-loss is defined as} $\ell(P,Q)=K(P,Q)$.\vspace{2mm}\\
{\bf Wasserstein loss.} The {\em (first) Wasserstein distance} between two probabilities $P$ and $Q$ on $E=[0,1]$ (with $\cE$ its Borel $\sigma$-algebra) associated to the Euclidean metric  is
\begin{equation}
W(P,Q)=\inf_{X\sim P,Y\sim Q}\E\cro{\ab{X-Y}}=\sup_{f}\cro{\E\pa{f(X)}-\E\pa{f(Y)}},
\label{eq-WPQ}
\end{equation}
where the infimum runs among all pairs $(X,Y)$ with marginal distributions $P$ and $Q$ and the supremum among all functions $f$ on $[0,1]$ which are Lipschitz with Lipschitz constant not larger than 1. We refer to Villani~\citeyearpar{MR2459454} [pages 77 and 78] {for the second equality in (\ref{eq-WPQ})}. The {estimator corresponding to the Wasserstein loss $\ell(P,Q)=W(P,Q)$} is the W-estimator.\vspace{2mm}\\
{\bf $\L_{j}$-loss}. {Given the reference measure $\mu$ on $E$ and} $j\in [1,+\infty]$, we consider the set $\overline \sP_{j}$ of finite and signed measures $P$ on $(E,\cE)$ of the form $P=p\cdot \mu$ with $p\in \sL_{j}(E,\mu)\cap \sL_{1}(E,\mu)$. It is a normed linear space with {$\L_{j}$-norm 
$\|P\|_{j}=\norm{p}_{\mu,j}$, whith $\norm{\cdot}_{\mu,j}$ defined in~\eref{defL-jnorm} and~\eref{defL-inftynorm}}.
Given two elements $P=p\cdot \mu$ and $Q=q\cdot \mu$ in $\overline\sP_{j}$, we define the {\em $\L_{j}$-loss} $\ell_{j}$ on $\overline{\sP_{j}}$ by
%
\begin{equation}
\ell_{j}(P,Q)=\norm{p-q}_{\mu,j}.
\label{eq-Ljloss}
\end{equation}

Unlike the losses we have seen so far, the $\L_{j}$-loss between $P$ and $Q$ depends on the choice of the reference measure $\mu$. Changing $\mu$ would automatically change the value of $\ell_{j}(P,Q)$. An $\ell$-estimator for the $\ell_{j}$-loss is called an $\ell_{j}$-estimator.

\subsection{Assumptions\label{sect-2.4}}
As already mentioned in Section~\ref{sect-2.0}, the construction we use here only applies to some specific loss functions $\ell$ and countable models $\sM$. They are characterized by the fact that one can find a family 
\[
\sT(\ell,\sM)=\ac{t_{(P,Q)},\; (P,Q)\in\sM^{2}}
\]
of measurable functions on $(E,\cE)$ with the following properties.
%
\begin{ass}\label{Hypo-1}
The elements $t_{(P,Q)}$ of $\sT(\ell,\sM)$ satisfy: 
\begin{listi}
\item\label{cond-1} for all $P,Q\in\sM$, $t_{(P,Q)}=-t_{(Q,P)}$;
\item\label{cond-2}  there exist positive numbers $a_{0},a_{1}$ such that, for all $S\in\sP$ and $P,Q\in\sM$,
%
\begin{equation}
\E_{S}\cro{t_{(P,Q)}(X)}\le a_{0}\ell(S,P)-a_{1}\ell(S,Q);
\label{eq-SPQ}
\end{equation}
%
\item\label{cond-4} whatever $P$ and $Q$ in $\sM$, 
\[
\sup_{x\in E}t_{(P,Q)}(x)-\inf_{x\in E}t_{(P,Q)}(x)\le 1.
\]
\end{listi}
\end{ass}
Note that~\ref{cond-1} implies that $t_{(P,P)}=0${, hence by (\ref{eq-SPQ})} that $0\le (a_{0}-a_{1})\ell(S,P)$ for all $S\in \sP$ and $P\in\sM$. Consequently $a_{1}\le a_{0}$ (since $\ell$ is not constant equal to 0). We may therefore assume hereafter with no loss of generality that $a_{1}\le a_{0}$. 
%

We shall see {later} that some losses actually satisfy a stronger assumption, namely:

\begin{ass}\label{Hypo-2}
Additionally to Assumption~\ref{Hypo-1}, there exists $a_{2}>0$ such that 
\begin{enumerate}[label=(\roman*)]
\addtocounter{enumi}{3}
\item \label{cond-3} for all {$S\in\sP$ and $P,Q\in\sM$},
\[
\Var_{S}\cro{t_{(P,Q)}(X)}\le a_{2}\cro{\ell(S,P)+\ell(S,Q)}.
\]
\end{enumerate}
\end{ass}
It is clear that if a function $t_{(P,Q)}$ satisfies \ref{cond-2} and \ref{cond-3} for some positive numbers $a_{0},a_{1}$ and $a_{2}$, so does $Ct_{(P,Q)}$ for $C>0$, with the constants $Ca_{0},Ca_{1}$ and $C^{2}a_{2}$  in place of $a_{0},a_{1}$ and $a_{2}$ respectively. Condition \ref{cond-4} may therefore be interpreted as a normalizing condition which can be applied to any family $\sT(\ell,\sM)$ which is bounded in supremum norm. 

We shall see in Section~\ref{sect-ExaTest} that all the loss functions we have introduced  {previously} can be associated to families $\sT(\ell,\sM)$ that do satisfy Assumption~\ref{Hypo-1} (and sometimes  Assumption~\ref{Hypo-2}). 

\section{The $\ell$-estimator and its risk bound on a model\label{sect-RB}}

\subsection{Our estimation procedure\label{sect-RB2}}
Given a family $\sT(\ell,\sM)$ of functions {satisfying Assumption~\ref{Hypo-1}}, we {introduce,} for {each $\gP=\otimes_{i=1}^{n}P_{i}\in\sbP$ and $\gQ=\otimes_{i=1}^{n}Q_{i}\in\sbM$}, the test statistic
\begin{equation}\label{def-teststatis}
\gT(\bsX,\gP,\gQ)=\sum_{i=1}^{n}t_{(P_{i},Q_{i})}(X_{i}).
\end{equation}
Applying Assumption~\ref{Hypo-1}-\ref{cond-2} {to $S=P_{i}\et$, $P=P_{i}$ and $Q=Q_{i}$ for all $i\in\{1,\ldots,n\}$ successively and then summing the resulting inequalities over $i$, we derive} that
\begin{align}
\E\cro{\gT(\bsX,\gP,\gQ)}&\le a_{0}\gell(\gP\et,\gP)-a_{1}\gell(\gP\et,\gQ).
\label{eq-Estat1}
\end{align}
{Exchanging} the roles of $\gP$ and $\gQ$ we deduce from Assumption~\ref{Hypo-1}-\ref{cond-1} that 
\begin{align}
\E\cro{\gT(\bsX,\gP,\gQ)}&\ge a_{1}\gell(\gP\et,\gP)-a_{0}\gell(\gP\et,\gQ).
\label{eq-Estat2}
\end{align}

The basic idea underlying our estimation procedure is based on the following heuristics. Assume for the sake of simplicity that $a_{0}=a_{1}>0$ so that \eref{eq-Estat1} and \eref{eq-Estat2} imply that
\[
\E\cro{\gT(\bsX,\gP,\gQ)}=a_{0}\cro{\gell(\gP\et,\gP)-\gell(\gP\et,\gQ)}\quad \text{for all $\gP\et\in\sbP$ and $\gP,\gQ\in\sbM$.}
\]
This means that $a_{0}^{-1}\gT(\bsX,\gP,\gQ)$ is an {unbiased} estimator of the difference $\gell(\gP\et,\gP)-\gell(\gP\et,\gQ)$ for $\gP$ and $\gQ$ in $\sbM$. If we believe {that, for each fixed $\gP\in\sbM$, this estimator} is uniformly good over all $\gQ\in\sbM$, the quantity $a_{0}^{-1}\sup_{\gQ\in\sbM}\gT(\bsX,\gP,\gQ)$ should be close to 
\[
\sup_{\gQ\in\sbM}\cro{\gell(\gP\et,\gP)-\gell(\gP\et,\gQ)}=\gell(\gP\et,\gP)-\inf_{\gQ\in\sbM}\gell(\gP\et,\gQ).
\]
Since this latter quantity is {minimal when $\gP$ is the best approximation point of $\gP\et$ in $\sbM$ (provided that it exists)}, it is natural to define our estimator as a minimizer over $\sbM$ of the map 
\[
\gP\mapsto \gT(\bsX,\gP)=\sup_{\gQ\in\sbM}\gT(\bsX,\gP,\gQ).
\]
This minimizer may not exist but {only an {\em $\epsilon$-minimizer} is actually} necessary. More precisely, given $\epsilon>0$, we define an {\em $\ell$-estimator} of $\gP\et$ in $\sbM$ as any element $\widehat \gP$ of the set 
\begin{equation}
\sbE(\bsX)=\ac{\gP\in\sbM,\; \gT(\bsX,\gP)\le \inf_{\gP'\in\sbM}\gT(\bsX,\gP')+\epsilon}.
\label{def-est1}
\end{equation}
{Note that, since $\sbM$ is countable, $\widehat \gP$ can always be chosen in a measurable way.}
As we shall see below, it is {wiser} to choose $\epsilon$ small ({i.e.\ }not much larger than 1)  in order to improve the risk bound of an $\ell$-estimator. In particular, if there exists an element $\gP\in\sbM$ (not necessarily unique) such that $\gT(\bsX,\gP)=\inf_{\gP'\in\sbM}\gT(\bsX,\gP')$,
we should choose it as our estimator $\widehat \gP$. 

It follows from Assumption~\ref{Hypo-1}-\ref{cond-1} that $\gT(\bsX,\gP)\ge \gT(\bsX,\gP,\gP)=0$ for all $\gP\in\sbM$ {so that} any element $\widehat \gP$ that satisfies $0\le \gT(\bsX,\widehat \gP)\le \epsilon$ is an $\ell$-estimator.

\subsection{Risk bounds of an $\ell$-estimator on a model\label{sect-RB3}}
As suggested by the previous heuristics, the performance of our {estimator depends} on how close $\gT(\bsX,\gP,\gQ)$ is to its expectation, hence on the behaviour of the process $\overline \gZ$ defined on $\sbM^{2}$ by
\begin{align}
(\overline \gP,\gQ)\mapsto \overline \gZ(\bsX,\overline \gP,\gQ)&=\gT(\bsX,\overline \gP,\gQ)-\E\cro{\gT(\bsX,\overline \gP,\gQ)}\nonumber \\
&=\sum_{i=1}^{n}\cro{t_{(\overline P_{i},Q_{i})}(X_{i})-\E\cro{t_{(\overline P_{i},Q_{i})}(X_{i})}}.
\label{def-barZ}
\end{align}
%
{More precisely, the performance of an $\ell$-estimator is controled by a combination of the approximation function $\overline{\gP}\mapsto\gell(\gP\et,\overline\gP)$ and the function $v$ from $\sbM$ to $\R_{+}$ given by
\begin{equation}\label{def-wpbar}
v(\overline \gP)=\frac{1}{\sqrt{n}}\gw(\overline \gP)\quad\text{with}\quad\gw(\overline \gP)=\E\cro{\sup_{\gQ\in\sbM}\ab{\overline \gZ(\bsX,\overline \gP,\gQ)}}
\end{equation}
as shown by the following theorem to be proven in Section~\ref{sect-pfsth}. It appears that in many applications, $\gw(\overline \gP)$ is of order $\sqrt{n}$, which motivates our introduction of $v(\overline \gP)$.} 
%
\begin{thm}\label{thm-main01}
Let Assumption~\ref{Hypo-1} be {satisfied, $\xi>0$ and $\overline\gP$ be an arbitrary element of $\sbM$.} Any $\ell$-estimator $\widehat \gP$, i.e.\ any element of the random set $\sbE(\bsX)$ defined by \eref{def-est1}, satisfies, {whatever $\gP\et\in\sbP$,}
\begin{align}
\gell(\gP\et,\widehat \gP)&\le \frac{2a_{0}}{a_{1}}\gell(\gP\et,\overline\gP)-\gell(\gP\et,\sbM)+{\frac{\sqrt{n}}{a_{1}}\cro{2v(\overline \gP)+\sqrt{2\xi}}+\frac{\epsilon}{a_{1}}}
\label{eq-thm01}
\end{align}
with probability at least $1-e^{-\xi}$. Consequently,
\begin{equation}\label{eq-BornRisk}
\E\cro{\frac{\gell(\gP\et,\widehat \gP)}{n}}\le C\inf_{\overline \gP\in\sbM}\cro{\frac{\gell(\gP\et,\overline\gP)}{n}+\frac{v(\overline \gP)+1}{\sqrt{n}}}
\end{equation}
where $C$ {only} depends on the constants $a_{0},a_{1}$ and $\epsilon$. 
\end{thm}
%
We shall see in our examples that $v(\overline \gP)$ is related to some global complexity of the model $\sbM$ with {respect to the approximation point $\overline \gP$, typically to its ``dimension'' in a suitable sense (linear, VC or metric)}. Note that the minimum in~\eref{eq-BornRisk} might not be achieved for the best approximation point of $\gP\et$ in $\sbM$ but rather by some element $\overline \gP\in \sbM$ that provides the best tradeoff between approximation and complexity at that point. However, in many situations, the quantity  $v(\overline \gP)+1$ can be bounded uniformly over $\sbM$ by some quantity $v(\sbM)$ that only depends on the model. In this case \eref{eq-BornRisk} leads to 
\begin{equation}\label{eq-BornRisk1}
{C^{-1}\E\cro{\frac{\gell(\gP\et,\widehat \gP)}{n}}\le \inf_{\overline \gP\in\sbM}\frac{\gell(\gP\et,\overline\gP)}{n}+\frac{v(\sbM)}{\sqrt{n}}=\frac{\gell(\gP\et,\sbM)}{n}+\frac{v(\sbM)}{\sqrt{n}}.}
\end{equation}
The quantity $v(\sbM)/\sqrt{n}$ corresponds to the bound we would get if $\gP\et$ did belong to $\sbM$ while $\gell(\gP\et,\sbM)/n$ corresponds to an approximation term due to a possible misspecification of the model.  

In density estimation where {$\sbM=\{P\on,\;P\in\sM\}$}, \eref{eq-BornRisk1} becomes
\begin{equation}\label{eq-BornRisk2}
\E\cro{\frac{\gell(\gP\et,\widehat \gP)}{n}}\le C\cro{\inf_{P\in\sM}\cro{\frac{1}{n}\sum_{i=1}^{n}\ell(P_{i}\et,P)}+\frac{v(\sbM)}{\sqrt{n}}}.
\end{equation}
Note that the approximation term can be small even in the unfavourable situation where none of the true marginals $P_{i}\et$ {belongs to $\sM$.
When the data} are truly i.i.d.\ with distribution $P\et\in\overline{\sM}$, then
\[
\sup_{P\et\in\overline{\sM}}\E\cro{\ell(P\et,\widehat P)}\le \frac{Cv(\sbM)}{\sqrt{n}},
\]
{which implies} that the minimax rate over $\overline{\sM}$ is at most of order $1/\sqrt{n}$. 

{Let us now see how} this bound can be improved under {the} additional property that the family $\sT(\ell,\sM)$ {satisfies Assumption~\ref{thm-main02}. In} order to obtain such an improvement we need to analyze the {behaviour of the process} $\gQ\mapsto \overline \gZ(\bsX,\overline \gP,\gQ)$ in some neighbourhood (with respect to {$\gell$}) of $\overline \gP$. For this purpose, we introduce the following sets, to be called {\em balls} hereafter, even though $\ell$ is not a distance in general:
\begin{equation}
\sbB(\gP\et,y)=\ac{{\gQ}\in\sbM,\; \gell(\gP\et,\gQ)\le y}\quad \text{for $y\ge 0$}.
\label{def-sball}
\end{equation}
We then define the {associated quantity $\gw(\overline\gP,y)$ which is a local analogue of $\gw(\overline\gP)$:}
\begin{equation}\label{eq-wpbary}
\gw(\overline\gP,y)= \E\cro{\sup_{\gQ\in\sbB(\gP\et,y)}\ab{\overline \gZ(\bsX,\overline \gP,\gQ)}}.
\end{equation}
%
We set
\begin{align}
c_{1}&=\frac{a_{1}}{2}\cro{2(1+\log 4)+\frac{4a_{1}}{a_{2}}+\frac{16a_{2}\log 2}{a_{1}}}^{-1}\label{def-c1p}
\end{align}
and
\begin{equation}\label{def-d}
D(\overline \gP)=\sup\ac{y>0\left|\,\gw(\overline \gP,y)> c_{1}y\right.}\vee c_{1}^{-1}.
\end{equation}
The quantity $D(\overline \gP)$ is related to the {\em local} complexity of the model $\sbM$ in a neighbourhood of $\overline \gP$. It shares some similarities {with the notion} of complexity introduced by V. Koltchinskii~\citeyearpar{MR2329442} in statistical {learning. 
It is clear that $\gw(\overline \gP,y)\le \gw(\overline \gP)$ for all values of $y>0$ and consequently that $D(\overline \gP)\le c_{1}^{-1}\cro{\gw(\overline \gP)\vee 1}$. However, this bound may be very crude since $\gw(\overline \gP)$ is of order $\sqrt{n}$ while $D(\overline \gP)$ is, in many cases, of order a constant or a power of $\log n$. 

\begin{thm}\label{thm-main02}
Let Assumption~\ref{Hypo-2} be {satisfied, $\xi>0$ and $\overline\gP$ be an arbitrary element of $\sbM$.} Any $\ell$-estimator $\widehat \gP$ satisfies, whatever $\gP\et\in\sbP$ 
\begin{align}
\gell(\gP\et,\widehat \gP)\le\cro{\frac{4a_{0}}{a_{1}}+1}\!\gell(\gP\et,\overline\gP)-\gell(\gP\et,\sbM)+2D(\overline\gP)+\cro{\frac{4a_{2}}{a_{1}}+1}\!\frac{8\xi}{a_{1}}+\frac{2\epsilon}{a_{1}}\label{thm2-b1}
\end{align}
with probability at least $1-e^{-\xi}$.
\end{thm}
The proof of this theorem is also postponed to Section~\ref{sect-pfsth}. 

In the common situation where $D(\overline \gP)$ can be bounded by some quantity $D_{n}$, independently of $\overline \gP$, we may derive from Theorem~\ref{thm-main02} an upper bound for the risk of the form
\[
\E\cro{\frac{\gell(\gP\et,\widehat \gP)}{n}}\le C\cro{\frac{\gell(\gP\et,\sbM)}{n}+\frac{D_{n}}{n}}
\]
for some positive constant $C$ depending on $a_{0},a_{1},a_{2}$ and the choice of $\epsilon$. In density estimation, if $\gP\et=(P\et)\on$ with $P\et\in\overline{\sM}$ and $D_{n}\le D$ for all $n$, we {conclude} that the minimax rate over $\overline{\sM}$ with respect to the loss $\ell$ is not larger than $1/n$ (up to a numerical constant).  This is {a substancial}  improvement over inequality~\eref{eq-BornRisk2} which is solely based on Assumption~\ref{Hypo-1}. 

\section{Examples of $\ell$-estimators and their performances}\label{sect-exlEst1}

\subsection{Building suitable families $\sT(\ell,\sM)$\label{sect-ExaTest}}
In order to apply Theorems~\ref{thm-main01} or \ref{thm-main02}, we have to find families $\sT(\ell,\sM)$ which satisfy Assumptions~\ref{Hypo-1} or \ref{Hypo-2}.
Let us first explain how to build such families for three of our loss functions, {namely} Wasserstein, $\L_{j}$  and TV. These losses share the property that they can be defined via a variational formula. {More generally}, let us assume {that the loss $\ell$ can be defined as follows}. There exists a subset $\overline \sP$ of the space of finite and possibly signed measures on $(E,\cE)$ that contains  $\sP\cup\sM$ and for which 
%
\begin{equation}\label{eq-lossG}
\ell(P,Q)=\sup_{f\in\cF}\cro{\int_{E}fdP-\int_{E}fdQ}\quad \text{for all $P,Q\in\overline \sP$},
\end{equation}
where $\cF$ is a symmetric class of measurable functions on $(E,\cE)$ (if $f\in\cF$ then $-f\in\cF$). It follows from~\eref{eq-lossG} and the symmetry property of the class that $\ell$ satisfies all the requirements for being a distance except from the fact that $\ell(P,Q)=0$ does not necessarily imply that $P=Q$. Adding, if ever necessary, the null function 0 to $\cF$, which does {not change} the definition of $\ell$, we {may} assume with no loss of generality that $\cF$ contains 0. {Let us} moreover require that the following assumption be satisfied.
%
%
\begin{ass}\label{Hypo-classF}
There exists a subset $\cF_{0}=\{f_{(P,Q)},\; (P,Q)\in \sM^{2}\}$ of $\cF$ with the following properties: 
\begin{listi}
\item\label{Hypo-classF1} for all $P,Q\in\sM$, $f_{(P,Q)}=-f_{(Q,P)}$; 
\item\label{Hypo-classF2} {there} exists a number $b>0$ such that, for all $P,Q\in\sM$, 
\[
\sup_{x\in E}f_{(P,Q)}(x)-\inf_{x\in E}f_{(P,Q)}(x)\le b
\]
\item\label{Hypo-classF3} for all $P,Q\in\sM$,
\begin{equation}\label{eq-Hypo-classF}
\ell(P,Q)=\int_{E}f_{(P,Q)}dP-\int_{E}f_{(P,Q)}dQ.
\end{equation}
\end{listi}
\end{ass}
Assumption~\ref{Hypo-classF} essentially means that, for each pair $(P,Q)$ of elements of $\sM^{2}$, we know where the supremum in~\eref{eq-lossG} is reached. We then associate to $\cF_{0}$ the family $\sT(\ell,\sM)$ of functions 
%
\begin{equation}\label{eq-def-phiPQ}
t_{(P,Q)}=\frac{1}{b}\ac{{\left[\int_{E}f_{(P,Q)}\frac{dP+dQ}{2}\right]}-f_{(P,Q)}}\quad \text{for {all} $(P,Q)\in\sM^{2}$.}
\end{equation}
The following result {which is} proven in Section~\ref{sect-pfs1}
shows that this family {fulfills} Assumption~\ref{Hypo-1}.
%
\begin{prop}\label{prop-lossvar}
If the loss {function} $\ell$ satisfies~\eref{eq-lossG} for some symmetric class $\cF$ containing 0, {then} it is nonnegative, symmetric and satisfies the triangle inequality on $\overline \sP$. Under Assumption~\ref{Hypo-classF}, for $(P,Q)\in\sM^{2}$ the function $t_{(P,Q)}$ defined by~\eref{eq-def-phiPQ} satisfies 
\begin{equation}\label{eq-loss-prty}
\E_{S}\cro{t_{(P,Q)}(X)}\le \frac{3}{2b}\ell(S,P)-\frac{1}{2b}\ell(S,Q)\quad \text{for all $S\in\overline \sP$.}
\end{equation}
In particular, the family $\sT(\ell,\sM)$ of such functions satisfies Assumption~\ref{Hypo-1} with $a_{0}=3/(2b)$ and $a_{1}=1/(2b)$.
\end{prop}
{With this proposition at hand, we are now able} to deal successively with the Wasserstein, $\L_{j}$ and TV-losses which do satisfy~\eref{eq-lossG}.

\subsection{The Wasserstein loss\label{sect-Wasset2}}
In this section, let $\sP=\overline \sP$ be the set of all probabilities on $([0,1],\sB([0,1]))$. As already seen in (\ref{eq-WPQ}), the Wasserstein distance between $P$ and $Q$ in $\overline \sP$ satisfies the variational formula
%
\begin{equation}\label{Wass-variat}
W(P,Q)=\sup_{f\in\cF}\cro{\E_{P}(f)-\E_{Q}(f)}
\end{equation}
where $\cF$ is the (symmetric) class of 1-Lipschitz functions on $[0,1]$. The following result, {which is} proven in Section~\ref{sect-pfs6}, 
provides a suitable family of functions $f_{(P,Q)}$.
%
\begin{prop}\label{lem-Wass}
For all $(P,Q)\in\sP^{2}$, the supremum in \eref{Wass-variat} is reached for the function $f_{(P,Q)}$ defined on $[0,1]$ by 
\begin{equation}\label{def-f-wass}
f_{(P,Q)}:x\mapsto \int_{0}^{x}\cro{\1_{F_{Q}(t)>F_{P}(t)}-\1_{F_{P}(t)>F_{Q}(t)}}dt,
\end{equation}
{where $F_{P}$ and $F_{Q}$ denote the cumulative distribution functions of $P$ and $Q$ respectively.} In particular, the family $\cF_{0}=\{f_{(P,Q)},\; (P,Q)\in\sM^{2}\}$ satisfies Assumption~\ref{Hypo-classF} with $b=1$.
\end{prop}
{As an immediate consequence of Proposition~\ref{prop-lossvar} we get}:
%
\begin{cor}\label{cor-wasser}
Let $\sM$ be a countable subset of $\sP$ and {$\ell(\cdot,\cdot)=W(\cdot,\cdot)$}. The family $\sT(\ell,\sM)$ of functions $t_{(P,Q)}$ given by~\eref{eq-def-phiPQ} {with} $f_{(P,Q)}$ defined by~\eref{def-f-wass}, satisfies Assumption~\ref{Hypo-1} with $a_{0}=3/2$ and $a_{1}=1/2$.
\end{cor}
The following proposition gives the expression of the statistic $\gT(\bsX,\cdot,\cdot)$ associated to the family $\sT(\ell,\sM)$.
\begin{prop}\label{prop-Wasser}
Let $P$ and $Q$ be two probabilities in $\sM$ with distribution functions $F_{P}$ and $F_{Q}$ respectively. For $\gP=P\on$ and $\gQ=Q\on$,
\begin{align*}
\gT(\bsX,\gP,\gQ)=n\int_{0}^{1}\left[\1_{F_{Q}>F_{P}}(t)-\1_{F_{P}>F_{Q}}(t)\right]\cro{\widehat F_{n}(t)-\frac{F_{P}(t)+F_{Q}(t)}{2}}dt
\end{align*}
where $\widehat F_{n}$ denotes the empirical distribution function.
\end{prop}
In particular, whenever the empirical measure $\widehat P_{n}=n^{-1}\sum_{i=1}^{n}\delta_{X_{i}}$ belongs {to $\sM$, for all $Q\in\sM$,}
\begin{align*}
\gT(\bsX,\widehat P_{n}\on,Q\on)&=n\int_{0}^{1}\left[\1_{F_{Q}>\widehat F_{n}}(t)-\1_{\widehat F_{n}>F_{Q}}(t)\right]\cro{\frac{\widehat F_{n}(t)-F_{Q}(t)}{2}}dt\\
&=-\frac{n}{2}\int_{0}^{1}\ab{\widehat F_{n}(t)-F_{Q}(t)}dt\le 0,
\end{align*}
which implies that $\gT(\bsX,\widehat P_{n}\on)=0$ and that $\widehat P_{n}$ is a W-estimator. 
\begin{proof}[Proof of Proposition~\ref{prop-Wasser}]
It follows from (\ref{def-f-wass}) that for all random variables $X$ with values in $[0,1]$, 
\begin{align*}
f_{(P,Q)}(X)&=\int_{0}^{1}\cro{\1_{F_{Q}(t)>F_{P}(t)}-\1_{F_{P}(t)>F_{Q}(t)}}\1_{t<X}dt\\
&=\int_{0}^{1}\cro{\1_{F_{Q}(t)>F_{P}(t)}-\1_{F_{P}(t)>F_{Q}(t)}}\cro{1-\1_{X\le t}}dt
\end{align*}
and for all probabilities $R$ on $[0,1]$ with distribution function $F_{R}$,
\begin{align*}
\E_{R}\cro{f_{(P,Q)}(X)}&=\int_{0}^{1}\cro{\1_{F_{Q}(t)>F_{P}(t)}-\1_{F_{P}(t)>F_{Q}(t)}}\cro{1-F_{R}(t)}dt.
\end{align*}
Hence, for all $i\in\{1,\ldots,n\}$ 
\begin{align*}
t_{(P,Q)}(X_{i})&=\frac{1}{2}\cro{\E_{P}\cro{f_{(P,Q)}(X_{i})}+\E_{Q}\cro{f_{(P,Q)}(X_{i})}}-f_{(P,Q)}(X_{i})\\
&=\int_{0}^{1}\cro{\1_{F_{Q}(t)>F_{P}(t)}-\1_{F_{P}(t)>F_{Q}(t)}}\cro{\1_{X_{i}\le t}-\frac{F_{P}(t)+F_{Q}(t)}{2}}dt
\end{align*}
and the result follows by averaging over $i\in\{1,\ldots,n\}$.
\end{proof}
%
%
%
%

\begin{exa}\label{exa-wasser}
The observations $X_{1},\ldots,X_{n}$ are independent with values in $[0,1]$ but presumed to be i.i.d.\  with a common distribution close to a model $\sM\subset\sP$. Our aim is to estimate $\gP\et$ using the Wasserstein loss. The following result, {which is} proven in Section~\ref{sect-pfscor2},
is a consequence of Theorem~\ref{thm-main01}.

\begin{cor}\label{cor-Wasser}
Whatever the model $\sM$ and $\xi>0$, any W-estimator $\widehat P\in\sM$ based on the family $\sT(\ell, \sM)$ provided by Corollary~\ref{cor-wasser} satisfies, with probability at least $1-e^{-\xi}$,
\begin{align}
\frac{1}{n}\sum_{i=1}^{n}W(P_{i}\et,\widehat P)&\le 5\inf_{P\in\sM}\cro{\frac{1}{n}\sum_{i=1}^{n}W(P_{i}\et,P)}+\frac{2}{\sqrt{n}}\cro{1+\sqrt{2\xi}+\frac{\epsilon}{\sqrt{n}}}.\label{eq-cor-Wasser}
\end{align}
If, in particular, the data are truly i.i.d.\ with distribution $P\et\in\overline{\sM}$, {it follows} that 
\[
\P\left[W(P\et,\widehat P)\le \frac{2}{\sqrt{n}}\pa{1+\sqrt{2\xi}+\frac{\epsilon}{\sqrt{n}}}\right]
\ge1-e^{-\xi}\quad\text{for all }\xi>0.
\]
\end{cor}
Note that the bound does not depend on the choice of the model $\sM\subset \sP$ which can therefore be as large as desired. {In particular a countable and dense subset of $\sP$ with respect to the Wasserstein distance would do.}
\end{exa}
%

\subsection{The $\L_{j}$-loss for $j\in (1,+\infty)$.}\label{sect-Lj}
Let us consider the $\L_{j}$-loss $\ell_{j}$ defined in Section~\ref{sect-2.3} and take $\overline \sP=\overline \sP_{j}$. {Let $\cF$ be the (symmetric) class of functions $f\in \sL_{j'}(E,\mu)$ satisfying $\norm{f}_{\mu,j'}\le 1$ where $j'$ denotes the conjugate exponent $j/(j-1)$ of $j$.}
It is well-known that 
%
\begin{equation}\label{Lj-variat}
\ell_{j}(P,Q)=\sup_{f\in \cF}\int_{E}(p-q)fd\mu,
\end{equation}
{which is \eref{eq-lossG}. It follows} from H\"older inequality (actually from the case of equality), that the supremum in \eref{Lj-variat} is reached for  
%
\begin{equation}\label{def-fPQ}
f_{(P,Q)}=\frac{\pa{p-q}_{+}^{j-1}-\pa{p-q}_{-}^{j-1}}{\norm{p-q}_{\mu,j}^{j-1}}\quad \text{when }P\ne Q\quad\text{and}\quad f_{(P,P)}= 0. 
\end{equation}
Note {that $f_{(P,Q)}=(p-q)/\norm{p-q}_{\mu,2}$ for $j=2$}. 
\begin{cor}\label{cor-Lj1}
Let $j\in (1,+\infty)$. Assume that the set of probabilities $\sP$ and the countable model $\sM$ are two subsets of $\overline \sP_{j}$ and that there exists a number $R>0$ such that 
%
\begin{equation}\label{eq-LinftyLj}
\norm{p-q}_{\infty}\le R\norm{p-q}_{\mu,j}\quad \text{for all $P,Q\in\sM$.}
\end{equation}
{For $(P,Q)\in\sM^{2}$, let}
%
\begin{equation}\label{T-Lj}
t_{(P,Q)}=\frac{1}{2R^{j-1}}\cro{\int_{E}f_{(P,Q)}\frac{dP+dQ}{2}-f_{(P,Q)}}
\end{equation}
{with $f_{(P,Q)}$ given by~\eref{def-fPQ}. The resulting family $\sT(\ell,\sM)=\{t_{(P,Q)},\;(P,Q)\in\sM^{2}\}$ satisfies} Assumption~\ref{Hypo-1} with $a_{0}=3/(4R^{j-1})$ and $a_{1}=1/(4R^{j-1})$ for the loss $\ell_{j}$. 
\end{cor}
%
\begin{proof}
The family of functions $\{f_{(P,Q)},\; (P,Q)\in\sM^{2}\}$ clearly satisfies {Assumption~\ref{Hypo-classF}-\ref{Hypo-classF1} and-\ref{Hypo-classF3}}. Besides, {when (\ref{eq-LinftyLj}) holds,} for all $x,x'\in\R$, 
\begin{align*}
f_{(P,Q)}(x)-f_{(P,Q)}(x')\le \frac{2\norm{p-q}_{\infty}^{j-1}}{\norm{p-q}_{\mu,j}^{j-1}}\le 2R^{j-1}
\end{align*}
{so that} {Assumption~\ref{Hypo-classF}-\ref{Hypo-classF2}} is satisfied with $b=2R^{j-1}$. The conclusion {then follows from} Proposition~\ref{prop-lossvar}.
\end{proof}
%
\subsubsection{The quadratic loss and linear models of densities\label{sect-exlEst2}}
Assume that the marginal distributions $P_{i}\et$ of the data $\bsX=(X_{1},\ldots,X_{n})$ admit densities $p_{i}\et$ with respect to some positive dominating measure $\mu$ and that $p_{1}\et,\ldots, p_{n}\et$ belong to $\sL_{2}(E,\mu)$. Our set $\sP$ is {therefore} the set of all probabilities $P=p\cdot \mu$ with $p\in \sL_{2}(E,\mu)$. {We pretend that the observations are i.i.d., although} this might not be true. {To estimate the presumed common density of the data we introduce a model of densities $\cM$ which may contain functions $p$ that {are not} probability densities and which is a subset of some linear subspace $V$ of $\sL_{1}(E,\mu)\cap \sL_{2}(E,\mu)$ with} the following property:
%
\begin{ass}\label{hypo-Linfty-L2}
The pair $(V,\norm{\cdot}_{\mu,2})$ is a Hilbert space of functions for which there exists a positive number $R$ such that 
\begin{equation}\label{eq-LinfL20}
\norm{t}_{\infty}\le R \norm{t}_{\mu,2}\quad \text{for all $t\in V$}.
\end{equation}
\end{ass}
When $E$ is a compact space, typically $[0,1]^{d}$, this assumption is met for many finite dimensional spaces with good approximation properties as shown in Birg\'e and Massart~\citeyearpar{MR1653272}[Section~3]. Nevertheless, our approach allows us to consider more general situations where the set $E$ is not compact and $V$ possibly infinite dimensional. Illustrations will be given in Section~\ref{sect-Besov-L2}.
In this framework, we  use the family $\sT(\ell,\sM)$ of functions given by (\ref{T-Lj}) with $j=2$ to build our $\ell_{2}$-estimator. Its performance is given by the following result {which is} proven in Section~\ref{sect-pfscor6}.
%
\begin{cor}\label{cor-L21}
Assume that $\cM$ is a subset of a linear space $V\subset \sL_{1}(E,\mu)\cap \sL_{2}(E,\mu)$ that  satisfies Assumption~\ref{hypo-Linfty-L2}. Any $\ell_{2}$-estimator $\widehat P=\widehat p\cdot \mu$ based on $\cM$ satisfies, for all $\xi>0$,
%
\begin{align}
\frac{\gell_{2}(\gP\et,\widehat P\on)}{n}&\le 5\inf_{P\in\sM}\frac{\gell_{2}(\gP\et,P\on)}{n}+\frac{4R}{\sqrt{n}}\cro{1+\sqrt{2\xi}+\frac{\epsilon}{\sqrt{n}}}\label{eq-cor-L1}
\end{align}
with probability at least $1-e^{-\xi}$. 
\end{cor}
The bound we get does not depend on the dimension of the linear space $V$ (which can therefore be infinite) but rather on the constant $R$ that controls the ratio between the sup-norm and the $\L_{2}$-norm on $V$. 

{When} $p_{i}\et=p\et\in\sL_{2}(\mu)$ for all $i$, there exists a large amount of literature on the problem of estimating {the density $p\et$ using the} $\L_{2}$-norm. A nice feature of~\eref{eq-cor-L1} lies in the fact that it does not involve the sup-norm of the density $p\et$ which may therefore be unbounded. Birg\'e and Massart~\citeyearpar{MR1653272}[Theorem~2 p.~343] studied the property of the projection estimator on finite dimensional linear spaces $V$ satisfying~\eref{eq-LinfL20}, typically linear spaces of functions on $[0,1]^{d}$. Since our result holds for possibly non-i.i.d.\ data and infinite dimensional linear spaces, it generalizes theirs.

\subsubsection{Risk bounds for the quadratic loss over Besov spaces\label{sect-Besov-L2}}
In this section we consider the problem of estimating a density $p\et$ with respect to the Lebesgue measure $\mu=\lambda$ on $E=\R^{d}$ {using $n$ i.i.d.\ oservations with density $p\et$}, under the assumption that $p\et$ is close to a given Besov space $B^{\alpha}_{s,\infty}(\R^{d})$ with $\alpha>0$ and $s\in [2,+\infty)$. We refer to Meyer~\citeyearpar{MR1228209} for a definition of these classes of functions and to Gin\'e and Nickl~\citeyearpar{MR3588285} Section~4.3.6 for their {characterization} in terms of coefficients in a suitable wavelet basis. Our loss function is based on the $\L_{2}$-norm.

\begin{prop}\label{prop-approx-L2}
Let $s\ge 2$, $d\ge 1$ and $\alpha>0$. There exist two constants $K,K'$ depending on $d,\alpha$ and $s$ with the following properties.  For all $J\ge 0$, there exists a linear subspace $V_{J}$ of $\sL_{1}(\R^{d})\cap \sL_{2}(\R^{d})$ such that $(V_{J},\norm{\cdot}_{\lambda,2})$ is a Hilbert space satisfying Assumption~\ref{hypo-Linfty-L2} with $R=K2^{Jd/2}$ and, for all $f\in B^{\alpha}_{s,\infty}(\R^{d})\cap \sL_{1}(\R^{d})\cap \sL_{2}(\R^{d})$,
%
\begin{equation}\label{eq-approx-Besov}
\inf_{t\in V_{J}}\norm{f-t}_{\lambda,2}^{2}\le K'\ab{f}_{\alpha,s,\infty}^{s/(s-1)}\norm{f}_{\lambda,1}^{(s-2)/(s-1)}2^{-Js\alpha/(s-1)}
\end{equation}
where $\ab{f}_{\alpha,s,\infty}$ is the Besov semi-norm of $f$ in $B^{\alpha}_{s,\infty}(\R^{d})$.
\end{prop}
The proof of this approximation result can be found in Section~\ref{sect-pfs-prop13}.
In the right-hand side of~\eref{eq-approx-Besov}, we use the convention $0^{0}=0$ when $s=2$ and $\norm{f}_{\lambda,1}=0$. Note that this approximation bound neither  depends on the $\L_{2}$-norm nor on the sup-norm of $f$ which may therefore be arbitrarily large. 
%
\begin{cor}\label{cor-L2besov}
Let $s\ge 2$, $\alpha>0$, $r>0$, $d\ge 1$ and $\cF_{\alpha,s,\infty}^{d}(r)$ be the class of all probability densities $p$ on $\R^{d}$ that belong to $B^{\alpha}_{s,\infty}(\R^{d})\cap \sL_{2}(\R^{d})$ and such that their Besov semi-norms are bounded by $r>0$. There exists an $\ell_{2}$-estimator $\widehat p$ (depending on $s,\alpha$ and $r$) that satisfies, whatever the density $p\et$ of the $X_{i}$, 
\[
\E\cro{\norm{p\et -\widehat p}_{\lambda,2}^{2}}\le C\left[\inf_{p\in \cF_{\alpha,s,\infty}^{d}(r)}
\norm{p\et -p}_{\lambda,2}^{2}+\frac{r^{ds/[d(s-1)+s\alpha]}}{n^{\alpha s/[d(s-1)+s\alpha]}}+\frac{1}{n}\right],
\]
where $C$ is a positive number that depends on $s,d,\alpha$ and $\epsilon$ only. 
\end{cor}
An interesting feature of this result lies in the fact that the class $\cF_{\alpha,s,\infty}^{d}(r)$ contains densities that are neither compactly supported nor bounded in supremum norm when $\alpha<1/s$. We are not aware of many results in this direction.
When $d=1$ and for $r,r'>0$, the bound we get is known to be optimal (up to a constant that depends on $r',\alpha$ and $s$) over the smaller set of densities $p\et$ which satisfy $\norm{p\et}_{\lambda,2}\vee \norm{p\et}_{\infty}\le r'$ and belong to $B^{\alpha}_{s,\infty}(\R)$ with Besov norms bounded by $r$. We refer the reader to Reynaud-Bouret {\em et al.}~\citeyearpar{MR2719482}[Theorem~4] and the references therein. The authors obtained there (see their Theorem~3) an upper bound which is similar to ours despite some differences. Our result does not require that the densities $p\et$ be uniformly bounded in $\L_{2}(\R)$ and it includes the case where $s=2$ while theirs is only true for $s>2$. Their estimator  is adaptive with respect to the parameters of the Besov space while ours is not. This could explain 
the extra-logarithmic factor that appears in their risk bound. Nevertheless, we believe that this extra-logarithmic factor is actually unnecessary for adaptation. 

\begin{proof}[Proof of Corollary~\ref{cor-L2besov}]
Throughout this proof, we fix some probability density $\overline{p}$ in $\cF_{\alpha,s,\infty}^{d}(r)$. Let $J$ be the nonnegative integer which satisfies
\[
2^{J}\le 1\vee\left(nr^{s/(s-1)}\right)^{(s-1)/[d(s-1)+s\alpha]}<2^{J+1}
\]
and $V_{J}$ be the Hilbert space provided by Proposition~\ref{prop-approx-L2} for this value of $J$. We consider the model of (signed) densities $\cM=V_{J}$ (or more precisely a countable dense subset of it with respect to the $\L_{2}$-norm). Since by Proposition~\ref{prop-approx-L2} the space $V_{J}$ satisfies Assumption~\ref{hypo-Linfty-L2} with $R=K 2^{Jd/2}$, Corollary~\ref{cor-L21} applies and{, by integrating~\eref{eq-cor-L1} with respect to $\xi>0$,} we obtain that an $\ell_{2}$-estimator $\widehat p$ of $p\et$ based on $\cM$ satisfies 
%
\begin{equation}\label{eq-proof-cor-besov}
\E\cro{\norm{p\et-\widehat p}_{\lambda,2}^{2}}\le C_{0}\cro{\norm{p\et-\overline p}_{\lambda,2}^{2}+\inf_{p\in \cM}\norm{\overline p-p}_{\lambda,2}^{2}+\frac{2^{Jd}}{n}},
\end{equation}
where $C_{0}$ is a positive constant that only depends on $d,s,\alpha$ and $\epsilon$. Since $\overline{p}$ belongs to $B^{\alpha}_{s,\infty}(\R^{d})\cap \sL_{2}(\R^{d})$ and satisfies $|\overline{p}|_{\alpha,s,\infty}\le r$, it follows from~\eref{eq-approx-Besov} that we may choose $p\in\cM$ such that
\[
\norm{\overline p-p}_{\lambda,2}^{2}\le K'r^{s/(s-1)}2^{-Js\alpha/(s-1)}
\]
with a possibly enlarged value of $K'$. Our choice of $J$ implies that 
\[
\norm{p-\overline p}_{\lambda,2}^{2}\le K'\pa{r^{d} n^{-\alpha}}^{s/[d(s-1)+s\alpha]}\quad \text{and}\quad \frac{2^{Jd}}{n}\le \pa{r^{d} n^{-\alpha}}^{s/[d(s-1)+s\alpha]}+\frac{1}{n}.
\]
The final bound on $\E\cro{\norm{p\et -\widehat p}_{\lambda,2}^{2}}$ follows from~\eref{eq-proof-cor-besov} and a minimization with respect to $\overline{p}\in\cF_{\alpha,s,\infty}^{d}(r)$.
\end{proof}
%

\subsubsection{The $\L_{j}$-loss for models of piecewise constant functions\label{sect-histo-Lj}}
{In this section, we assume that $\mu$ is a probability on $E$ and consider the $\L_{j}$-loss defined by~\eref{eq-Ljloss} with $j\in (1,+\infty)$. Let $\cI$ be a partition of $E$ into $D\ge 2$ pieces satisfying $\mu(I)=1/D$ for all $I\in\cI$. Our} density model $\cM=\cM_{D}$ is a countable and dense  subset {(with respect to} the $\L_{j}$-norm) of the set $\overline \cM_{D}$ which gathers the functions which are piecewise constant on the elements of $\cI$. As usual, $\overline \sM=\overline \sM_{D}=\{P=p\cdot\mu,\; p\in\overline \cM_{D}\}$. We emphasize the fact that an element $p$ of $\overline \cM_{D}$ may not be a density of probability, hence $P=p\cdot\mu\in\overline \sM_{D}$ may not be a probability but only a finite signed measure. Besides, if an element of $\overline \cM_{D}$ necessarily belongs to $\L_{\infty}(E,\mu)$, its supremum norm may be arbitrary large. The following result is proven in Section~\ref{Pr-coro6}.
%
\begin{cor}\label{cor-Lj-histo}
Let $n\ge 2$, $D\in\{2,\ldots,n\}$, $j\in (1,+\infty)$. Assume that the data are i.i.d.\ with distribution $P\et=p\et\cdot \mu$ with $p\et\in \sL_{j}(E,\mu)$ and set 
\[
\overline p_{D}=\sum_{I\in\cI}\cro{D\int_{I}p\et(x)\,d\mu(x)} \1_{I}.
\]
The $\ell_{j}$-estimator $\widehat P=\widehat p\cdot \mu$ of $P\et$ based on $\cM=\cM_{D}$ and the family $\sT(\ell,\sM)$ given in Corollary~\ref{cor-Lj1} with $R=D^{1/j}$ satisfies, whatever $\xi>0$  with probability at least $1-e^{-\xi}$,
\begin{align*}
\ell_{j}(P\et,\widehat P)\le 5\inf_{\overline P\in\overline \sM_{D}}\ell_{j}(P\et,\overline P)&+ C_{j}\sqrt{\frac{D}{n}\norm{\overline p_{D}}_{\mu,j/2}}+\frac{4D^{1-1/j}}{\sqrt{n}}\cro{\sqrt{2\xi}+\frac{\epsilon}{\sqrt{n}}}\label{eq-Lj-histo}
\end{align*}
where 
\[
C_{j}=
\begin{cases}
8\max\ac{2^{1-1/j}\pa{\frac{j\sqrt{e}}{\sqrt{e}-1}}^{1/2}+\pa{\frac{j}{e-\sqrt{e}}}^{1/2}, \frac{j\sqrt{e}}{2^{1/j}(\sqrt{e}-1)}}& \text{for $j>2$}\\
4 & \text{for $j\in (1,2]$.}
\end{cases}
\]
\end{cor}
It follows from convexity arguments that, whatever the density $p\et\in\sL_{j}(E,\mu)$, $1\le \norm{\overline p_{D}}_{\mu,j/2}\le D^{1-2/j}$ for $j>2$ while $D^{1-2/j}\le \norm{\overline p_{D}}_{\mu,j/2}\le 1$ for  $j\in (1,2]$. These inequalities together with Corollary~\ref{cor-Lj-histo} lead to the following uniform risk bound
\[
\sup_{p\et\in\overline \cM_{D}}\E\cro{\norm{p\et-\widehat p}_{\mu,j}}\le C'(j,\epsilon)\frac{D^{(1-1/j)\vee 1/2}}{\sqrt{n}}\quad \text{for all $j\in (1,+\infty)$.}
\]
In a private communication to the author, Lucien Birg\'e proved that this bound is minimax. This result shows in passing that~\eref{eq-thm01} cannot in general be improved for the $\ell_{j}$-loss for $j\in (1,+\infty)$. We also mention that the minimax rate would be different on the submodel $\{p\in\overline \cM_{D},\; \norm{p}_{\infty}\le R\}$ with $R>0$, that is, under a constraint on the supremum norm of the elements of $\overline \cM_{D}$.  

A well-known estimator of $p\et$ on $\overline \cM_{D}$ is the {histogram} $\widetilde p$ defined by 
\begin{equation}\label{def-histo}
\widetilde p=D\sum_{I\in\cI}\widehat \nu_{n}(I)\1_{I}
\end{equation}
where $\widehat \nu_{n}$ is the empirical measure $n^{-1}\sum_{i=1}\delta_{X_{i}}$. A natural question is how $\widetilde P=\widetilde p\cdot \mu$ compares to an $\ell_{j}$-estimator. Actually, when $\widetilde P$ belongs to $\sM_{D}$, it is an $\ell_{j}$-estimator: it follows from~\eref{T-Lj} and~\eref{def-fPQ} that for all $Q=q\cdot\mu\in\sM_{D}$, we may write $q=\sum_{I\in\cI}DQ(I)\1_{I}$ so that 
{
\begin{align*}
\lefteqn{2R^{j-1}\gT(\bsX,\widetilde P\on,Q\on)}\hspace{15mm}\\
&=n\cro{\int_{E}f_{(\widetilde P,Q)}\frac{\widetilde p+q}{2}d\mu-\int_{E}f_{(\widetilde P,Q)}d\widehat \nu_{n}}\\
&=n\sum_{I\in\cI}\cro{\int_{I}f_{(\widetilde P,Q)}\frac{\widetilde p+q}{2}d\mu-\int_{I}f_{(\widetilde P,Q)}d\widehat \nu_{n}}\\
&=\sum_{i\in I}\cro{\cro{(\widehat \nu_{n}(I)-Q(I))_{+}^{j-1}-(\widehat \nu_{n}(I)-Q(I))_{-}^{j-1}}\frac{Q(I)-\widehat \nu_{n}(I)}{2}}\\
&\quad\times \frac{nD^{j-1}}{\norm{\widetilde p-q}_{\mu,j}^{j-1}}\\
&=-\frac{nD^{j-1}}{2\norm{\widetilde p-q}_{\mu,j}^{j-1}}\sum_{i\in I}
\ab{\widehat \nu_{n}(I)-Q(I)}^{j}\le 0
\end{align*}
}
hence $\gT(\bsX,\widetilde P\on)=0$ and $\widetilde P$ is an $\ell_{j}$-estimator.

\subsubsection{{The $\L_{\infty}$-loss for models of piecewise constant functions}\label{sect-histo-Linfty}}
In this section, we consider the statistical framework and model  $\overline \cM_{D}$ introduced in Section~\ref{sect-histo-Lj}. Our aim is to estimate the density $p\et\in\sL_{\infty}(E,\mu)$ with respect to the $\ell_{\infty}$-loss given by~\eref{eq-Ljloss} with $j=\infty$. 
As for the other $\ell_{j}$-losses, the $\ell_{\infty}$-loss satisfies a variational formula of the form~\eref{eq-lossG} with $\overline \sP=\overline \sP_{\infty}$ and $\cF$ the set of functions $f$ on $(E,\cE,\mu)$ that satisfy $\norm{f}_{\mu,1}\le 1$. However, unlike the case $j\in (1,+\infty)$, the supremum is not reached in general. Fortunately, our Assumption~\ref{Hypo-classF} only requires that we know where the supremum is reached when the two measures $P,Q$ in~\eref{eq-lossG} belong to the model $\sM$. {Using this property, we can prove the following result.}
%
\begin{prop}\label{label-casLinfty}
For $P,Q$ in $\sM_{D}$, let $I\et=I\et(P,Q)$ be a maximizer on $\cI$ of the mapping $I\mapsto |P(I)-Q(I)|$ and  set $f_{(P,Q)}=D\text{sign}\pa{P(I\et)-Q(I\et)}\1_{I\et}$. The family $\cF_{0}$ that gathers the functions $f_{(P,Q)}$ for $(P,Q)$ varying among $\sM_{D}^{2}$ satisfies Assumption~\ref{Hypo-classF} with $b=D$. 
\end{prop}
\begin{proof}
Clearly, {Assumptions~\ref{Hypo-classF}-\ref{Hypo-classF1} and-\ref{Hypo-classF2}} are satisfied. It remains to prove~\eref{eq-Hypo-classF}. For $P,Q$ in $\sM_{D}$, 
\begin{align*}
\ell_{\infty}(P,Q)=\norm{\sum_{I\in\cI}D(P(I)-Q(I))\1_{I}}_{\mu,\infty}=D\ab{P(I\et)-Q(I\et)}
\end{align*}
and, by definition of $I\et=I\et(P,Q)$,  
{
\begin{align*}
\int_{E}f_{(P,Q)}(dP-dQ)&=D\text{sign}\pa{P(I\et)-Q(I\et)}\int_{I\et}\pa{p-q}d\mu\\
&=D\ab{P(I\et)-Q(I\et)}=\ell_{\infty}(P,Q).\qedhere
\end{align*}
}%
\end{proof}
With Proposition~\ref{label-casLinfty} at hand, Proposition~\ref{prop-lossvar} applies and the family $\sT(\ell,\sM)$ that satisfies our Assumption~\ref{Hypo-1} for the $\ell_{\infty}$-loss is given by 
{
\begin{equation}
t_{(P,Q)}=\text{sign}(P(I\et)-Q(I\et))\cro{\frac{P(I\et)+Q(I\et)}{2}-\1_{I\et}}
\label{eq-tLinfty}
\end{equation}
for all $P,Q\in\sM_{D}$.} The following corollary of Theorem~\ref{thm-main01} is proven in Section~\ref{sect-pfscor7}.
%
\begin{cor}\label{cor-Linfty}
Assume that the data $X_{1},\ldots,X_{n}$ are i.i.d.\ with density $p\et\in\sL_{\infty}(E,\mu)$. Let $n\ge 2$, $D\in\{2,\ldots,+\infty\}$. The $\ell_{\infty}$-estimator $\widehat P=\widehat p\cdot \mu$ based on $\cM_{D}$ and the family $\sT(\ell,\sM)$ defined above satisfies, for all $\xi>0$ with a probability at least $1-e^{-\xi}$, 
\begin{align}
\ell_{\infty}(P\et,\widehat P)&\le 5\inf_{\overline p\in\cM_{D}}\ell_{\infty}(P\et,\overline P)+2D\cro{
\sqrt{\frac{2\log(2D)}{n}}+\sqrt{\frac{2\xi}{n}}+\frac{\epsilon}{n}}.\label{eq-corLinfty}
\end{align}
\end{cor}
Inequality~\eref{eq-corLinfty} shows that the $\ell_{\infty}$-estimator on $\cM_{D}$ performs well for estimating densities of the form $p\et=\overline p_{D}+g$ with $\overline p_{D}\in\overline \cM_{D}$ and $g$ such that $\norm{g}_{\mu,\infty}$ is small compared to $D\sqrt{\log D/n}$.  

As is the case when $j\in(1,+\infty)$, the estimator $\widetilde P=\widetilde p\cdot\mu$ based on the classical {histogram} $\widetilde p$ defined by~\eref{def-histo} is an $\ell_{\infty}$-estimator of $P\et$ (whenever {$\widetilde P$} belongs to $\sM_{D}$). Indeed, for all $Q\in\sM_{D}$, 
\begin{align*}
\gT(\bsX,\widetilde P\on,Q\on)&=n\text{sign}(\widehat \nu_{n}(I\et)-Q(I\et))\cro{\frac{\widehat \nu_{n}(I\et)+Q(I\et)}{2}-\widehat \nu_{n}(I\et)}\\
&=-\frac{n}{2}\ab{\widehat \nu_{n}(I\et)-Q(I\et)}\le 0.
\end{align*}
%

\section{The case of the TV-loss\label{sect-TV}}
Throughout this section, $\sP$ is the set of all probability measures on $(E,\cE)$.
\subsection{Building suitable families $\sT(\ell,\sM)$\label{sect-TV0}}
It is well-known that the TV-distance $\norm{P-Q}$ defined by (\ref{eq-def-TV0}) between two probabilities $P,Q\in\sP$ can equivalently be written as 
\begin{equation}\label{eq-def-TV}
\norm{P-Q}=\sup_{f\in\cF}\cro{\E_{P}(f)-\E_{Q}(f)},
\end{equation}
where $\cF$ is the symmetric class of {all} measurable functions $f$ on {$E$} with values in $[-1/2,1/2]$. The supremum in~\eref{eq-def-TV} is reached for 
\begin{equation}\label{def-fPQ-TV}
f_{(P,Q)}=\frac{1}{2}\pa{\1_{p>q}-\1_{q>p}}
\end{equation}
where $p$ and $q$ denote versions of the respective densities of $P$ and $Q$ with respect to some common dominating measure $\mu$. We deduce from Proposition~\ref{prop-lossvar} the following corollary.
%
\begin{cor}\label{cor-TV}
Let $\sM=\{P=p\cdot \mu,\; p\in\cM\}$ be a {countable} subset of $\sP$ and $\ell$  the TV-loss. The family $\cF_{0}=\{f_{(P,Q)},\; (P,Q)\in\sM^{2}\}$ with $f_{(P,Q)}$ defined by~\eref{def-fPQ-TV} satisfies Assumption~\ref{Hypo-classF} with $b=1$. {The set $\sT(\ell,\sM)$ of all the functions} 
\begin{align}
t_{(P,Q)}&=\frac{1}{2}\cro{\1_{q>p}-Q(q>p)}-\frac{1}{2}\cro{\1_{p>q}-P(p>q)}\label{phi-TV0}
\end{align}
{with $(P,Q)\in\sM^{2}$} satisfies Assumption~\ref{Hypo-1} with $a_{0}=3/2$ and $a_{1}=1/2$.
\end{cor}
%

\subsection{Risk bounds based on VC-dimensions}\label{sect-ex1}
In this section, we pretend {(although this may not be true)} that our observations $X_{1},\ldots,X_{n}$ are i.i.d.\ with a distribution $ P\et$ belonging to a statistical model $\sM\subset \sP$ associated to a density model $\cM$. Given a density $\overline p\in\cM$, we consider the following assumption. 
%
\begin{ass}\label{Hcor-estimD}
The classes of subsets of $E$ given by $\{\{\overline p<q\},  q\in\cM\setminus\{\overline p\}\}$ and $\{\{\overline p>q\},  q\in\cM\setminus\{\overline p\}\}$ are both VC with dimension not larger than $V(\overline p)\ge 1$. 
\end{ass}
We refer the reader to Dudley~\citeyearpar{MR876079} for the definition of the VC-dimension of a class of sets. The family of sets of the form $\{p>q\}$ with $p,q\in\cM$ {is} known as the Yatracos class associated to $\cM$. Assumption~\ref{Hcor-estimD} is weaker than the usual assumption that the Yatracos class $\{\{p>q\},\; p,q\in\cM\}$ is VC (see Devroye and Lugosi~\citeyearpar{MR1843146} for example). In particular, we shall see how to take advantage of this weaker form in our Example~\ref{exa-densit-decr} for estimating a density under a shape constraint.
\begin{cor}\label{cor-estimD}
Let $\overline p\in \cM$ satisfy Assumption~\ref{Hcor-estimD}. For any TV-estimator $\widehat P\in\sM$ based on the family $\sT(\ell,\cM)$ given in Corollary~\ref{cor-TV}, {all} $\gP\et\in\sbP$  and all $\xi>0$, with a probability at least $1-e^{-\xi}$,
\begin{align}
\frac{1}{n}\sum_{i=1}^{n}\norm{P_{i}\et-\widehat P}\le &\frac{6}{n}\sum_{i=1}^{n}\norm{P_{i}\et-\overline P}+40\sqrt{\frac{5V(\overline p)}{n}}+2\sqrt{\frac{2\xi}{n}}+\frac{2\epsilon}{n}
\label{eqcor-estimD0}\\
&-\inf_{P\in\sM}\frac{1}{n}\sum_{i=1}^{n}\norm{P_{i}\et-P}\nonumber.
\end{align}
In particular, if Assumption~\ref{Hcor-estimD} is satisfied for all $\overline p\in\cM$ and  $\sup_{\overline{p}\in\cM}V(\overline p)=V<+\infty$,
\begin{align}
\frac{1}{n}\sum_{i=1}^{n}\norm{P_{i}\et-\widehat P}&\le 5\inf_{\overline P\in\overline \sM}\frac{1}{n}\sum_{i=1}^{n}\norm{P_{i}\et-\overline P}+ 40\sqrt{\frac{5V}{n}}+2\sqrt{\frac{2\xi}{n}}+\frac{2\epsilon}{n}.\label{eqcor-estimD}
\end{align}
\end{cor}
The proof of this corollary can be found in Section~\ref{sect-pfscor3}.

When the $X_{i}$ are truly i.i.d.\ with distribution $P\et$, (\ref{eqcor-estimD}) becomes
\begin{equation}\label{eqcor-estimD1}
\norm{P\et-\widehat P}\le 5\inf_{P\in\overline \sM}\norm{P\et-P}+40\sqrt{\frac{5V}{n}}+2\sqrt{\frac{2\xi}{n}}+\frac{2\epsilon}{n}.
\end{equation}
Whenever $P\et=p\et\cdot \mu$ is absolutely continuous with respect to $\mu$, the above result immediately translates into an upper bound on the $\L_{1}$-loss between the densities of $P\et$ and $\widehat P$ via the well-known formula 
\[
\norm{P-Q}=\frac{1}{2}\int_{E}\ab{\frac{dP}{d\mu}-\frac{dQ}{d\mu}}d\mu.
\]
Integrating~\eref{eqcor-estimD1} with respect to $\xi$, we deduce a risk bound {for the estimator $\widehat p$ of $p\et$} of the form 
\[
\E\cro{\norm{p\et-\widehat p}_{\mu,1}}\le C\left[\inf_{p\in \cM}\norm{p\et- p}_{\mu,1}+\sqrt{\frac{V}{n}}\right],
\]
for some positive number $C>0$ depending on $\epsilon$ only. Up to the numerical constant $C>0$, this bound is similar to that obtained for the minimum distance estimator in Devroye and Lugosi~\citeyearpar{MR1843146}.


\begin{exa}[Estimation of the mean of a Gaussian vector]\label{exa-gauss}\ \\
In order to illustrate the robustness property of TV-estimators, let us focus on the following problem. The observations are presumed to be {i.i.d., following a common Gaussian distribution with mean vector $m\et$ and identity covariance matrix:} $P_{m\et}=\cN(m\et,I_{d})$ in $\R^{d}$. But they are actually contaminated so that, for $1\le i\le n$, the true distribution of $X_{i}$ is $P_{i}\et=(1-\alpha_{i})P_{m\et}+\alpha_{i}R_{i}$ for some arbitrary probabilities $R_{i}$ and small numbers $\alpha_{i}\in [0,1]$. We choose for our model the family $\sM$ of Gaussian distributions $P_{m}$ with mean $m\in\Q^{d}$ and identity covariance matrix. Denoting by $p_{m}$ the {corresponding density,} we see that for all $m,\overline m\in\Q^{d}$ {with} $m\ne \overline m$, the sets $\{p_{\overline m}<p_{m}\}$ and $\{p_{\overline m}> p_{m}\}$ are half-spaces of $\R^{d}$. The VC-dimension of this class is not larger than $V=d+1$ (see Devroye and Lugosi~\citeyearpar{MR1843146}, Corollary 4.2 page 33). Assumption~\ref{Hcor-estimD} is therefore satisfied with $V(\overline p)=V=d+1$ for all $\overline p\in\cM$. Besides, the following lemma {which is} proven in Section~\ref{sect-pfslem1}
{allows }to relate the TV-distance between $P_{m}$ and $P_{m'}$ to the Euclidean one between the parameters $m$ and $m'$.
%
\begin{lem}\label{lem-TV}
For all $m,m'\in\R^{d}$, 
\begin{equation}\label{eq-CompTVEuclid0}
\norm{P_{m}-P_{m'}}=\P\cro{|Z|\le \ab{m-m'}/2}
\end{equation}
where $Z$ is a standard real-valued Gaussian random variable. Consequently, 
\begin{equation}\label{eq-CompTVEuclid}
0.78\min\ac{1,\frac{\ab{m-m'}}{\sqrt{2\pi}}}\le \norm{P_{m}-P_{m'}}\le  \min\ac{1,\frac{\ab{m-m'}}{\sqrt{2\pi}}}.
\end{equation}
\end{lem}
This means that when $m'$ is close enough to $m$ the quantity $\|P_{m}-P_{m'}\|$ is of order $\ab{m-m'}/\sqrt{2\pi}$ while it is of order 1 when $m'$ and $m$ are far apart. 
We deduce from~\eref{eqcor-estimD} that, whatever $\overline m\in\Q^{d}$ and $\xi>0$, with probability at least $1-e^{-\xi}$, the TV-estimator $\widehat P=P_{\widehat m}$ satisfies 
\[
\norm{P_{\overline m}-P_{\widehat m}}\ge0.78\min\ac{1,\frac{\ab{\overline m-\widehat m}}{\sqrt{2\pi}}} 
\]
and
\begin{align}
\norm{P_{\overline m}-P_{\widehat m}}&\le \frac{1}{n}\sum_{i=1}^{n}\pa{\norm{P_{i}\et-P_{\overline m}}+\norm{P_{i}\et-P_{\widehat m}}}\nonumber\\
&\le \frac{6}{n}\sum_{i=1}^{n}\norm{P_{i}\et-P_{\overline m}}+40\sqrt{\frac{5(d+1)}{n}}+2\sqrt{\frac{2\xi}{n}}+\frac{2\epsilon}{n}.
\label{eq-TVgauss}
\end{align}
%
Since the mapping $m\mapsto \norm{P_{i}\et-P_{m}}$ is continuous with respect to the Euclidean norm on $\R^{d}$ and $\overline{m}$ can be chosen arbitrarily close to $m\et$, \eref{eq-TVgauss} is actually satisfied {with $\overline{m}=m\et$}. Using  the inequality $\norm{P_{i}\et-P_{m\et}}=\alpha_{i}\norm{R_{i}-P_{m\et}}\le \alpha_{i}$ for all $i\in\{1,\ldots,n\}$, we derive that, for $\xi>0$, with probability at least $1-e^{-\xi}$, the TV-estimator $\widehat P=P_{\widehat m}$ satisfies 
{
\begin{align}
0.78\min&\ac{1,\frac{\ab{m\et-\widehat m}}{\sqrt{2\pi}}}\nonumber\\
&\le \frac{6}{n}\sum_{i=1}^{n}\norm{P_{i}\et-P_{m\et}}+40\sqrt{\frac{5(d+1)}{n}}+2\sqrt{\frac{2\xi}{n}}+\frac{2\epsilon}{n}
\label{eq-TVgauss2b}
\end{align}
and
\begin{equation}
0.78\min\ac{1,\frac{\ab{m\et-\widehat m}}{\sqrt{2\pi}}}\le \frac{6}{n}\sum_{i=1}^{n}\alpha_{i}+40\sqrt{\frac{5(d+1)}{n}}+2\sqrt{\frac{2\xi}{n}}+\frac{2\epsilon}{n}.
\label{eq-TVgauss2}
\end{equation}
}
When the average $n^{-1}\sum_{i=1}^{n}\alpha_{i}$ is small  compared to $\sqrt{(d+1)/n}$ the bound we get is almost as good as that we would get {if there were} no {contamination, which therefore} warrants the robustness property of the TV-estimator $\widehat m$ with respect to contamination. When $\alpha_{i}=\alpha$ for all $i$, \eref{eq-TVgauss2}  is similar to the bound obtained in Gao {\em et al.}~\citeyearpar{gao2018robust}[Theorem~3.1] for TV-Gan in this setting. 
\end{exa}

An interesting feature of Corollary~\ref{cor-estimD} and more precisely~\eref{eqcor-estimD0} lies in the fact that the upper bound involves the quantity $V(\overline p)$ which may depend on the choice of $\overline p$. This means that the best choice of $\overline p$ in view of minimizing the right-hand side of~\eref{eqcor-estimD0} might not be the density of the best approximation point of $P\et$ in $\sM$. From this point of view, \eref{eqcor-estimD0} contrasts with~\eref{eqcor-estimD} which requires that for all $\overline{p}\in\cM$ this quantity be bounded independently of $\overline p$. This subtle difference allows us to deal with statistical models for which the quantity $V(\overline p)$ may vary from one density $\overline p$ to another and be even infinite for some $\overline p$. Such {a} situation typically arises when one estimates a density under a shape constraint, as shown {by} the following example.

\begin{exa}[Estimating a density under a monotonicity constraint]\label{exa-densit-decr}\ \\
Let us consider the problem of estimating a density which is presumably belonging to the set $\overline \cM$ of all non-increasing densities on {some} unknown half-line, i.e.\ densities $p$ (with respect to the Lebesgue measure $\mu=\lambda$ {on $\R$}) which are non-increasing on an interval (that may depend on $p$) of the form $(a,+\infty)$ with $a=a(p)\in\R$ and vanish elsewhere. For $d\ge 1$, let $\overline \cM_{d}$ be the subset of $\overline \cM$ of those densities of the form $\overline p=\sum_{I\in \cI}a_{I}\1_{I}$ where $\cI$ is a set of at most $d$ disjoint intervals with positive lengths and $a_{I}>0$ for all $I\in\cI$. In other words, $\overline \cM_{d}$ is the set of all non-increasing piecewise constant densities the supports of which are the unions of at most $d$ (non-trivial) intervals. We shall denote by $\overline \sM_{d}=\{\overline{p}\cdot\lambda,\,\overline{p}\in\overline \cM_{d}\}$ the corresponding set of probabilities and by $\cM_{d}$ and $\cM$ respectively some countable and dense subsets of $\overline \cM_{d}$ and $\overline \cM$ for  the $\L_{1}(\lambda)$-distance. We shall assume with no loss of generality that $\cM_{d}\subset \cM$ for all $d\ge 1$. 

Given $q\in\cM$ and $\overline p\in\cM_{d}$, the sets $\{\overline p<q\}$ and $\{\overline p> q\}$ are unions of at most $d$ intervals {so that it follows from Lemma~1 in Baraud and Birg\'e~\citeyearpar{MR3565484} that} Assumption~\ref{Hcor-estimD} is satisfied with $V(\overline p)\le 2d$.  We may then apply Corollary~\ref{cor-estimD} with an arbitrary choice of $d\ge1$ and $\overline p\in\cM_{d}$ (with $\overline P=\overline p\cdot \lambda$). Since $\cM_{d}$ is dense in $\overline \cM_{d}$ for all $d\ge 1$, we get the following result. 
\begin{prop}\label{prop-histo}
Let $\epsilon\le 1$. For all $\xi>0$, with a probability at least $1-e^{-\xi}$, the TV-estimator $\widehat P=\widehat p\cdot \lambda$ provided by Corollary~\ref{cor-estimD} and based on $\cM$ satisfies  
\begin{align}
\frac{1}{n}\sum_{i=1}^{n}\norm{P_{i}\et-\widehat P} &\le \inf_{d\ge 1}\cro{\inf_{\overline P\in\overline \sM_{d}}\frac{5}{n}\sum_{i=1}^{n}\norm{P_{i}\et-\overline P}+41\sqrt{\frac{10d}{n}}}+2\sqrt{\frac{2\xi}{n}}.\label{eqcor-estimDHisto}
\end{align}
In particular, if the data are i.i.d.\ with density $p\et$, 
\begin{align}
\norm{p\et-\widehat p}_{\lambda,1}&\le 5\inf_{d\ge 1}\cro{\inf_{\overline p\in\overline \cM_{d}}\norm{p\et-\overline p}_{\lambda,1}+16.4\sqrt{\frac{10d}{n}}}+4\sqrt{\frac{2\xi}{n}}
\label{eqcor-estimDHisto1}
\end{align}
with probability at least $1-e^{-\xi}$, for all $\xi>0$. 
\end{prop}
A famous estimator of a {truly monotone density $p\et$} is the Grenander one,  see Grenander~\citeyearpar{MR599175} and Groeneboom~\citeyearpar{MR822052}. The Grenander estimator relies on the assumption that the left endpoint $a=a(p\et)$ of the support of the target density $p\et$ is exactly known. Since this estimator is defined as the Maximum Likelihood Estimator (MLE for short) over the set of all non-increasing densities on $[a,+\infty)$, it would not exist on the larger set $\overline \cM$. Our TV-estimator does not need to know the value of $a$. More references on the performance of the MLE for estimating a density under a shape constraint can be found in a 2018 special issue of {\em Statistical Science}. 

Results of the same flavour as that presented in our Proposition~\ref{prop-histo} can be established for many other families of densities on the real line that satisfy a shape constraint (convexity, concavity or log-concavity,...). We refer to Baraud and Birg\'e~\citeyearpar{MR3565484} for more details. 

\end{exa}

Let $\overline \cM(H,L)$ be the subset of $\overline \cM$ that consists of those densities $p$ such that $I=\{x\in\R,\; \,p(x)>0\}$ is an interval of length not larger than $L>0$ and the variation of $p$ on $I$, i.e.\ the quantity $\sup_{x\in I}p(x)-\inf_{x\in I}p(x)$, is not larger than $H\ge 0$. The following approximation result, which is due to Birg\'e~\citeyearpar{MR902242}[see Section~2 pages 1014-1015], {enables} us to derive uniform risk bounds over $\overline \cM(H,L)$.
%
\begin{prop}\label{approx-Birge}
Let $p\in \overline \cM(H,L)$ with $H\ge 0$ and $L>0$. For each $d\ge 1$, there is {a density} $\overline p_{d}\in\overline \cM_{d}$ such that
\begin{equation}\label{eq-birge}
\norm{p-\overline p_{d}}_{\lambda,1}\le \exp\cro{\frac{\log(HL+1)}{d}}-1.
\end{equation}
\end{prop}
A remarkable feature of this result lies in the fact that, for large enough values of $d$, the approximation bound is of order $\log(1+HL)/d$ and therefore only depends logarithmically on $HL$. From this point of view, it significantly improves the usual approximation bound $HL/d$ which can easily be obtained by approximating $p$ with {a piecewise constant function built} on a regular partition of the support of $p$ into $d$ pieces. 

Using Proposition~\ref{approx-Birge} together with~\eref{eqcor-estimDHisto1} and optimizing with respect to $d$ leads to the following risk bound.
\begin{prop}\label{prop-monotone}
Let $\epsilon\le 1$ and $\widehat p$ be the TV-estimator of Proposition~\ref{prop-histo}. There exists a universal constant $C>0$ such that, whatever $H\ge 0$, $L>0$, $p\et\in\overline \cM(H,L)$ and $\xi>0$,
\begin{equation}\label{eq-monotone}
\norm{p\et-\widehat p}_{\lambda,1}\le C\left[\cro{\frac{\log(1+HL)}{n}}^{1/3}+\cro{\frac{\log(1+HL)+1+\xi}{n}}^{1/2}\right]
\end{equation}
with probability at least $1-e^{-\xi}$.
\end{prop}
%
\subsection{Robust regression with unimodal errors}\label{sect-ex2}
In this section $E=\R$. Given a density $q$ on $\R$ (with respect to the Lebesgue measure $\mu=\lambda$), we denote by $P_{\theta}$ the distribution with density $q_{\theta}=q(\cdot -\theta)$ for $\theta\in\R$ and for $\gtheta=(\theta_{1},\ldots,\theta_{n})\in \R^{n}$, $\gP_{\gtheta}$ is the product probability $P_{\theta_{1}}\otimes \ldots\otimes P_{\theta_{n}}$, i.e.\ the distribution of a random vector of the form $\bsX'=\gtheta+\geps$ where the components $\eps_{1},\ldots,\eps_{n}$ of $\geps$ are i.i.d.\ with density $q$. The vector $\gtheta$ will be called the location parameter of the distribution $\gP_{\gtheta}$. We presume that the true distribution $\gP\et=P_{1}\et\otimes\ldots\otimes P_{n}\et$ of our observation $\bsX$ is close to a probability of the form $\gP_{\gtheta\et}$. In view of estimating the location parameter $\bst\et$, we assume that 
it belongs to some (countable) subset $\Theta$ of $\R^{n}$. Our model for the distribution $\gP\et$ is therefore $\sbM=\{\gP_{\gtheta},\; \gtheta\in \Theta\}\subset \overline \sM^{n}$ with $\overline \sM=\{P_{\theta},\; \theta\in\R\}$. 
%
\begin{ass}\label{Hcor-estimE}
The density $q$ is unimodal on $\R$ and $\Theta$ is a subset of a linear subspace of $\R^{n}$ with dimension $d\ge 1$. 
\end{ass}
Under this assumption, we prove in Section~\ref{sect-pfscor4} 
the following deviation bound.
\begin{cor}\label{cor-estimE}
Let $\epsilon\le 1/2$. If Assumption~\ref{Hcor-estimE} is satisfied, any TV-estimator $\gP_{\widehat \gtheta}=\bigotimes_{i=1}^{n}\gP_{\widehat \theta_{i}}$ based on the model $\sbM$ and the family $\sT(\ell,\sM)$ given by~\eref{phi-TV0} satisfies, for all $\xi>0$, with a probability at least $1-e^{-\xi}$,
\begin{align}
&\frac{1}{n}\sum_{i=1}^{n}\norm{P_{i}\et-P_{\widehat \theta_{i}}}\le 5\inf_{\gtheta\in \Theta}\frac{1}{n}\sum_{i=1}^{n}\norm{P_{i}\et-P_{\theta_{i}}}+277\sqrt{\frac{d+1}{n}}+2\sqrt{\frac{2\xi}{n}}.\label{eqcor-estimE}
\end{align}
\end{cor}
It is interesting to analyze further the approximation term that appears in the right-hand side of~\eref{eqcor-estimE}. Let us assume hereafter that the data are of the form $\bsX=\gtheta\et+\geps$ with $\gtheta\et\in \R^{n}$ and $\eps_{1},\ldots,\eps_{n}$ {are} i.i.d.\ with a density $p$ that may not be $q$. Then, for each $i\in\{1,\ldots,n\}$, the TV-loss between the true $i$-th marginal distribution $P_{i}\et=p_{\theta_{i}\et}\cdot \lambda$ {with $p_{\theta_{i}\et}(x)=p(x-\theta_{i}\et)$} and $P_{\theta_{i}}=q_{\theta_{i}}\cdot \lambda\in \overline \sM$ with $\gtheta\in \Theta$ can be decomposed as follows: 
\begin{align*}
\norm{P_{i}\et-P_{\theta_{i}}}&=\frac{1}{2}\int_{\R}\ab{p_{\theta_{i}\et}-q_{\theta_{i}}}d\lambda\le \frac{1}{2}\int_{\R}\ab{p_{\theta_{i}\et}-q_{\theta_{i}\et}}d\lambda+ \frac{1}{2}\int_{\R}\ab{q_{\theta_{i}\et}-q_{\theta_{i}}}d\lambda\\
&=\frac{1}{2}\cro{\norm{p-q}_{\lambda,1}+\int_{\R}\ab{q_{\theta_{i}\et}-q_{\theta_{i}}}d\lambda}.
\end{align*}
Since the translation $t\mapsto q_{t}$ is uniformly continuous from $\R$ to $\sL_{1}(\R,\lambda)$, {it admits a modulus of continuity $w_{q}$ which is a nondecreasing, continuous and concave function on $[0,+\infty)$ such that $w_{q}(0)=0$}, {from} which we deduce that 
\[
\norm{P_{i}\et-P_{\theta_{i}}}\le \frac{1}{2}\cro{\norm{p-q}_{\lambda,1}+w_{q}\pa{|\theta_{i}\et-\theta_{i}|}}
\]
for all $i\in\{1,\ldots,n\}$ and $\gtheta\in \Theta$. Averaging these inequalities with respect to $i$ {leads to}
\begin{align}
\inf_{\gtheta\in \Theta}\frac{1}{n}\sum_{i=1}^{n}\norm{P_{i}\et-P_{\theta_{i}}}&\le \frac{1}{2}\cro{\norm{p-q}_{\lambda,1}+\inf_{\gtheta\in \Theta}\frac{1}{n}\sum_{i=1}^{n}w_{q}\pa{|\theta_{i}\et-\theta_{i}|}}.
\label{eq-analybiais}
\end{align}
It follows from the properties of $w_{q}$ that
\[
\Delta_{q}(\gtheta,\gtheta')=\frac{1}{n}\sum_{i=1}^{n}w_{q}\pa{|\theta_{i}-\theta_{i}'|}\quad \text{for $\gtheta,\gtheta'\in \R^{n}$}
\]
defines a distance on $\R^{n}$. We deduce from~\eref{eq-analybiais} that the {approximation} term is small when both $\norm{p-q}_{\lambda,1}$ and $\inf_{\gtheta\in \Theta}\Delta_{q}(\gtheta\et,\gtheta)$ are small. The first quantity accounts for a misspecification of the error distribution when $p\ne q$ while the second quantity depends on how well  the parameter set $\Theta$ approximates $\gtheta\et$ with respect to the distance $\Delta_{q}$. In order to illustrate this result further, let us consider the following example.
\begin{exa}
Let $X_{1},\ldots,X_{n}$ be independent random variables satisfying
\begin{equation}\label{eq-reg-cauchy}
X_{i}=\theta_{i}\et+\eps_{i}\quad \text{for i=$1,\ldots,n$},
\end{equation}
where $\gtheta\et=(\theta_{1}\et,\ldots,\theta_{n}\et)$ belongs to $[-B/2,B/2]^{n}$ for some $B>0$ and $\eps_{1},\ldots,\eps_{n}$ are i.i.d.\ with {Cauchy density $p=q:x\mapsto [\pi(1+x^{2})]^{-1}$}. Our purpose is to estimate $\gtheta\et$ on the basis of a model $\Theta\subset [-B/2,B/2]^{n}\cap\cV$ where $\cV$ is a linear space of dimension $d\ge 1$. This framework can be viewed as a regression where the errors are Cauchy distributed and the $\theta_{i}$ correspond to the values of a regression function at fixed design points. The reader can check that 
\begin{equation}\label{eq-lem-cauchy1}
\norm{P_{\theta}-P_{\theta'}}=\frac{2}{\pi}\arctan\frac{|\theta-\theta'|}{2}\quad \text{for all $\theta,\theta'\in\R$.}
\end{equation}
In particular, using the facts that $(B^{-1}\arctan B)u\le \arctan u\le u$ for all $u\in [0,B]$ and setting $\ab{\gu}_{1,n}=n^{-1}\sum_{i=1}^{n}|u_{i}|$ with $\gu=(u_{1},\ldots,u_{n})\in\R^{n}$, we deduce that
\begin{align*}
\frac{\arctan B}{B\pi}\ab{\gtheta\et-\gtheta}_{1,n}\le \frac{1}{n}\sum_{i=1}^{n}\norm{P_{\theta_{i}\et}-P_{\theta_{i}}}\le \frac{1}{\pi }\ab{\gtheta\et-\gtheta}_{1,n}
\end{align*}
for all $\gtheta\in \Theta$. Since Assumption~\ref{Hcor-estimE} is satisfied, we may apply Corollary~\ref{cor-estimE} and obtain that for all $\xi>0$, with a probability at least $1-e^{-\xi}$, 
\begin{align*}
\frac{\arctan B}{B\pi}\ab{\gtheta\et-\widehat \gtheta}_{1,n}&\le \frac{5}{\pi }\inf_{\gtheta\in \Theta} \ab{\gtheta\et-\gtheta}_{1,n}+277\sqrt{\frac{d+1}{n}}+ 2\sqrt{\frac{2\xi}{n}}.
\end{align*}
\end{exa}

\subsection{Faster rates under Assumption~\ref{Hypo-2}.}\label{sect-TV-Fast}
Unlike the results established by Devroye and Lugosi~\citeyearpar{MR1843146}[Chapter 7] and Gao {\em et al.}~\citeyearpar{gao2018robust} for estimating a density with respect to the TV-loss, we shall prove that TV-estimators may converge at a rate which can be faster than $1/\sqrt{n}$ provided that the model $\sM$ {satisfies Assumption~\ref{Hypo-2}}. To check whether {it} is fulfilled on $\cM$, one may use the following result {which is} proven in Section~\ref{sect-pfs2}.
%
\begin{prop}\label{prop-VT2}
If there exists a constant $a_{2}'\ge 0$ such that
%
\begin{equation}\label{cond-3bis}
P(p\le q)\wedge Q(p> q)\le a_{2}'\|P-Q\|
\end{equation}
for all probabilities $P,Q$ in $\sM$ then, for all probabilities $S\in\sP$, 
%
\begin{equation}\label{cond-3bis00}
S(p>q)\wedge S(p\le q)\le a_{2}\cro{\| S-P\|+\|S-Q\|}
\end{equation}
with $a_{2}=1+a_{2}'$. Besides the family $\sT(\ell,\sM)$ defined in Corollary~\ref{cor-TV}
satisfies Assumption~\ref{Hypo-2}. 
\end{prop}
Let us now comment on Condition~\eref{cond-3bis}. The testing affinity between two probabilities $P$ and $Q$ (see Le Cam~\citeyearpar{MR0334381, MR856411}) is defined as 
\begin{align}
\pi(P,Q)&=1-\norm{P-Q}=\int_{E}\pa{p\wedge q}d\mu=P(p\le q)+Q(p>q)\nonumber\\
&=P(p\le q)\vee Q(p>q)+P(p\le q)\wedge Q(p>q).\label{eq-afftest}
\end{align}
It corresponds to the sum of the errors of first and second kinds of the (optimal) test function $\1_{p\le q}$ when testing $P$ versus $Q$ on the basis of a single observation. In many situations, when $P$ and $Q$ are close with respect to the TV-distance, both errors are close to $1/2$. This is not the case when \eref{cond-3bis} holds: we deduce from~\eref{eq-afftest} that 
\[
1-(1+a_{2}')\norm{P-Q}\le P(p\le q)\vee Q(p>q).
\]
This inequality together with \eref{eq-afftest} show that when $P$ and $Q$ are close, one of the testing errors is close to 0 while the other is close to 1. To illustrate this phenomenon, let us present two examples in the translation model, i.e.\ when $\cM=\{p_{\theta}=p(\cdot-\theta),\; \theta\in\Q\}$ for some density $p$ with respect to the Lebesgue measure $\mu=\lambda$ on $\R$. As usual, we denote by $P_{\theta}$ the probability associated to the density $p_{\theta}$.

\begin{exa}\label{exa-1}
The density $p=\1_{[-1/2,1/2]}$ is that of the uniform distribution on $[-1/2,1/2]$. It is easy to see that for all $\theta,\theta'\in\R$, $P_{\theta'}(p_{\theta}>p_{\theta'})=0$. Hence 
\[
P_{\theta}(p_{\theta}\le p_{\theta'})\wedge P_{\theta'}(p_{\theta}> p_{\theta'})=P_{\theta'}(p_{\theta}> p_{\theta'})=0\le a_{2}'\norm{P_{\theta}-P_{\theta'}}
\]
and Condition~\eref{cond-3bis} is therefore satisfied with $a_{2}'=0$.
\end{exa}
%
\begin{exa}\label{exa-2}
We take for $p$ the unbounded density $x\mapsto \alpha x^{\alpha-1}\1_{(0,1]}$ for some $\alpha\in (0,1)$. Note that for $\theta>\theta'$, $P_{\theta}(p_{\theta}\le p_{\theta'})=P_{\theta}(p_{\theta}< p_{\theta'})=0$, hence 
\begin{align*}
P_{\theta}(p_{\theta}\le p_{\theta'})\wedge P_{\theta'}(p_{\theta}> p_{\theta'})&\le P_{\theta}(p_{\theta}\le p_{\theta'})=0
\end{align*}
and for $\theta<\theta'$, 
\begin{align*}
P_{\theta}(p_{\theta}\le p_{\theta'})\wedge P_{\theta'}(p_{\theta}> p_{\theta'})&\le P_{\theta'}(p_{\theta}> p_{\theta'})=0.
\end{align*}
Condition~\eref{cond-3bis} is therefore satisfied with $a_{2}'=0$.
\end{exa}
Let us now go back to the framework of Section~\ref{sect-ex1} assuming moreover that the observations $X_{1},\ldots,X_{n}$ are (truly) i.i.d.\ with distribution $P\et$ and that the family $\cM$ of densities associated to our statistical model $\sM$ satisfies Assumption~\ref{Hcor-estimD}.

\begin{cor}\label{cor-TV-fast}
Let $\epsilon \le 35$. Assume that $X_{1},\ldots X_{n}$ are i.i.d.\ with distribution $P\et$ and that Condition~(\ref{cond-3bis}) and {Assumption~\ref{Hcor-estimD} are both} satisfied with $V(\overline p)\le V$ for all $\overline p\in\cM$.
Then  any TV-estimator $\widehat P\in\sM$ based on the family $\sT(\ell,\sM)$ provided by Corollary~\ref{cor-TV} satisfies, for all $\xi>0$, with a probability at least $1-e^{-\xi}$, 
\begin{align}
\norm{P\et-\widehat P}&\le 14\inf_{P\in\overline \sM}\norm{P\et-P}+\frac{144a_{2}}{n}\cro{
ca_{2}^{2}V\log\pa{\frac{2en}{V\wedge n}}+1+\xi},
\label{eq-cor-TV-fast}
\end{align}
where $c$ is a positive numerical constant ($c=4.5\times 10^{5}$ suits).
\end{cor}
The proof {of} this Corollary can be found in Section~\ref{sect-pfscor5}.

To illustrate this result, let us go back to our Example~\ref{exa-2}. We {have seen} that (\ref{cond-3bis}) holds with $a_{2}'=0$ {so that we may take} $a_{2}=1$. For all $\theta,\theta'\in\R$, the sets $\{x\in\R,\; p_{\theta}(x)> p_{\theta'}(x)\}$ are intervals and such a class of subsets of $\R$ cannot shatter more than 2 points. {Consequently, Assumption}~\ref{Hcor-estimD} is satisfied with $V=2$ and it follows from Corollary~\ref{cor-TV-fast} that, whatever the true distribution $P\et$ of our observations, with a probability at least $1-e^{-\xi}$, 
\begin{equation}
\norm{P\et-\widehat P}\le C\left[\inf_{P\in\overline\sM}\norm{P\et-P}+\frac{\log n+1+\xi}{n}\right]\label{appl-ex-2},
\end{equation}
for some numerical constant $C>0$. For this particular translation model, the TV-distance between two probabilities $P_{\theta}$ and $P_{\theta'}$ in $\overline\sM$ can be related to the Euclidean distance between their parameters by arguing as follows. First of all, it is not difficult to check that the testing affinity between $P_{\theta}$ and $P_{\theta'}$ (with $\theta<\theta'$) writes as 
%
\begin{align*}
\pi(P_{\theta},P_{\theta'})&=\int_{\R}\pa{p_{\theta}\wedge p_{\theta'}}d\lambda=\int_{\theta}^{\theta'}0d\lambda+\int_{\theta'}^{1+\theta}p_{\theta}d\lambda+\int_{1+\theta}^{1+\theta'}0d\lambda\\
&=\cro{(x-\theta)^{\alpha}}_{\theta'}^{1+\theta}=1-\ab{\theta'-\theta}^{\alpha}\quad\text{when }
\theta'\le \theta+1
\end{align*}
and  $\pi(P_{\theta},P_{\theta'})=0$ for $\theta'>\theta+1$. Consequently, for all $\theta,\theta'\in\R$
%
\begin{align*}
\norm{P_{\theta}-P_{\theta'}}=1-\pi(P_{\theta},P_{\theta'})=\ab{\theta-\theta'}^{\alpha}\wedge 1,
\end{align*}
which means, using the triangle inequality, that if $P\et$ is close to some distribution $P_{\overline \theta}\in\overline\sM$, by (\ref{appl-ex-2}), the estimator $\widehat P=P_{\widehat \theta}$ of $P\et$ satisfies, with a probability at least $1-e^{-\xi}$,  
\[
\cro{\ab{\overline \theta-\widehat \theta}^{\alpha}\wedge 1}=\norm{P_{\overline \theta}-P_{\widehat \theta}}\le C\left[2\norm{P\et-P_{\overline \theta}}+\frac{\log n+1+\xi}{n}\right].
\]
In particular, if $P\et$ belongs to $\overline{\sM}$, i.e.\ $P\et=P_{\overline \theta}$ for some $\overline{\theta}\in\R$, and if $n$ is large enough, the estimator $\widehat \theta$ estimates $\overline \theta$ with an accuracy of order $(\log n/n)^{1/\alpha}$. This rate is much faster than $1/\sqrt{n}$ whatever $\alpha\in (0,1)$ and is optimal up to the logarithmic factor. 

It is not difficult to check that the above calculations extend to the case $\alpha=1$, i.e.\ when the statistical model is the translation of the uniform density $p=\1_{[-1/2, 1/2]}$ as in Example~\ref{exa-1}. The TV-estimator then converges at rate (at least) $\log n/n$. In particular, it does not coincide with the empirical median which converges at rate $1/\sqrt{n}$  in this case. Note that this result is not contradictory {to Proposition~\ref{prop-med} (to be presented in Section~\ref{sect-RO} below)} since the density $p$ is not a decreasing function of $|x|$. This proves, in passing, that {the} assumption that $f$ is decreasing {in Assumption~\ref{Ass-transl}} is necessary. 

\section{Hellinger and KL-losses\label{sect-HKL}}

\subsection{Building suitable families $\sT(\ell,\sM)$\label{sect-ExaTest3}}
The Hellinger and KL-losses cannot be defined by variational formulas like (\ref{eq-lossG}) and (\ref{eq-Hypo-classF}) but, as we shall see, satisfy the following alternative expressions for $P$ and $Q$ in $\overline \sP$:
%
\begin{equation}\label{eq-lossGb}
\ell(P,Q)=\sup_{f\in\cF}\cro{\int_{E}fdP-\Lambda(Q,f)}=\int_{E}f_{(P,Q)}dP-\Lambda(Q,f_{(P,Q)}),
\end{equation}
for some {suitable} class of functions $\cF$ and a fixed function $\Lambda$ on $\overline \sP\times\cF$. Observe that (\ref{eq-lossG}) and (\ref{eq-Hypo-classF}) are actually a special case of (\ref{eq-lossGb}) when $\Lambda(Q,f)=\int_{E}fdQ$.

A common feature of losses of the forms \eref{eq-lossGb} and \eref{eq-lossG} lies in the fact that, for the $P,Q$ that belong to some subset $\sQ$ of $\overline \sP$, we know where the supremum is reached, i.e.\ we have identified a function $f_{(P,Q)}$ such that $\ell(P,Q)=\int_{E}f_{(P,Q)}dP-\Lambda(Q,f_{(P,Q)})$ if $(P,Q)\in\sQ^{2}$. Let us assume that $\sQ$ contains $\sM$ as well as all probabilities $R=(P+Q)/2$ with $P,Q\in\sM$. Then a candidate function $t_{(P,Q)}$ to {satisfy Assumption~\ref{Hypo-1}} is
\begin{equation}\label{T-formeG}
t_{(P,Q)}=C\cro{\pa{f_{(R,P)}-\Lambda(P,f_{(R,P)})}-\pa{f_{(R,Q)}-\Lambda(Q,f_{(R,Q)})}},
\end{equation}
where $C$ denotes a positive normalizing constant that is chosen for $t_{(P,Q)}$ to {fulfill Assumption~\ref{Hypo-1}-\ref{cond-4}}. This {definition} of $t_{(P,Q)}$ is motivated by the equalities
\[
t_{(P,Q)}=-t_{(Q,P)}\quad \text{and}\quad \E_{R}\cro{t_{(P,Q)}(X)}=C\cro{\ell(R,P)-\ell(R,Q)},
\]
the second one meaning that the sign of $\E_{R}\cro{t_{(P,Q)}(X)}$ is the same as that of $\ell(R,P)-\ell(R,Q)$. When $\Lambda(Q,f)=\int_{E}fdQ$ and $\cF$ is symmetric
%
\[
\ell(R,P)=\sup_{f\in\cF}\cro{\int_{E}fdR-\int_{E}fdP}=\frac{1}{2}\sup_{f\in\cF}\cro{\int_{E}fdQ-\int_{E}fdP}=\frac{1}{2}\ell(P,Q)
\]
and we may therefore choose $f_{(R,P)}=f_{(Q,P)}=-f_{(P,Q)}=-f_{(R,Q)}$ which, together with~\eref{T-formeG}, gives
\begin{align*}
t_{(P,Q)}&=C\left[\pa{f_{(R,P)}-\int_{E}f_{(R,P)}dP}-\pa{f_{(R,Q)}-\int_{E}f_{(R,Q)}dQ}\right]\\
&=2C\cro{\int_{E}f_{(P,Q)}\frac{dP+dQ}{2}-f_{(P,Q)}}.
\end{align*}
Up to the normalizing constant, we recover {the definition \eref{eq-def-phiPQ} of $t_{(P,Q)}$}.

\subsection{The Hellinger distance}
An alternative way of defining the Hellinger distance given by (\ref{eq-Hell1}) is provided by the following proposition (with the conventions $0/0=1$ and $a/0=+\infty$ for all $a>0$) {which is} proven in Section~\ref{sect-pfs3}.
%
\begin{prop}\label{prop-hel}
Let $\cG$ be the class of all measurable functions $g$ on $(E,\cE)$ with values in $[0,+\infty]$. For all probabilities $P,Q$ on $(E,\cE)$,
%
\begin{equation}\label{eq-hel}
h^{2}(P,Q)=\frac{1}{2}\sup_{g\in \cG}\cro{\E_{P}\pa{1-g}+\E_{Q}\pa{1-1/g}}.
\end{equation}
If $\mu$ is a measure that dominates $P$ and $Q$ and $P=p\cdot\mu$, $Q=q\cdot\mu$, the supremum is reached for $g=g_{(P,Q)}=\sqrt{q/p}$. In particular, the Hellinger affinity between $P$ and $Q$ satisfies
\begin{equation}\label{eq-rho}
\rho(P,Q)=\frac{1}{2}\inf_{g\in \cG}\cro{\E_{P}\pa{g}+\E_{Q}\pa{1/g}}.
\end{equation}
\end{prop}
{Setting $f=1-g$ in (\ref{eq-hel}), we see that} \eref{eq-lossGb} is satisfied for the class $\cF$ of functions with values in $[-\infty,1]$, $\Lambda(Q,f)=\int_{E}[f/(1-f)]dQ$ and $f_{(P,Q)}=1-\sqrt{q/p}$ with $\sQ$ the set of all probabilities on $E$. {Then }\eref{T-formeG} {leads to}
\begin{align*}
t_{(P,Q)}=\;&C\left[\pa{f_{(R,P)}-\Lambda(P,f_{(R,P)})}-\pa{f_{(R,Q)}-\Lambda(Q,f_{(R,Q)})}\right]\\
=\;&C\left[\frac{\sqrt{q}-\sqrt{p}}{\sqrt{r}}+\int_{E}\sqrt{r}(\sqrt{q}-\sqrt{p})d\mu\right]
\quad\text{with } r=\frac{p+q}{2}.
\end{align*}
{The resulting test corresponds to the one} proposed in Baraud~\citeyearpar{MR2834722}. In particular, we obtain the following result which is proven in Section~\ref{sect-pfs4}.
\begin{prop}\label{prop-hel2}
Let $\sP$ be the set of {all} probabilities on $(E,\cE)$ dominated by $\mu$, $\sM$ a countable subset of $\sP$ and {$\ell$ be the loss function} defined by $\ell(P,Q)=h^{2}(P,Q)$ for all $P,Q\in\sP$. The family $\sT(\ell,\sM)$ of functions $t_{(P,Q)}$ defined for $P,Q\in\sM$ by
%
\begin{equation}\label{phi-hel}
t_{(P,Q)}=\frac{1}{2\sqrt{2}}\cro{\rho(R,Q)-\rho(R,P)+\frac{\sqrt{q}-\sqrt{p}}{\sqrt{r}}}\quad \text{with } R=\frac{P+Q}{2},
\end{equation}
satisfies Assumption~\ref{Hypo-2} with 
$a_{0}=(\sqrt{2}+1)/2$,  $a_{1}=(\sqrt{2}-1)/2$, $a_{2}=3/2$.
\end{prop}
Any $\ell$-estimator based on this family $\sT(\ell,\sM)$ is a $\rho$-estimator. It is possible to design other families $\sT(\ell,\sM)$ that satisfy Assumption~\ref{Hypo-2} on the larger set of all probabilities on $(E,\cE)$, {which} are not necessarily dominated by $\mu$, but this requires more technicalities. We prefer to avoid them here and rather refer the {interested} reader to Baraud and Birg\'e~\citeyearpar{BarBir2018}.

\subsection{The Kullback-Leibler divergence}
We mention the Kullback-Leibler divergence as an example of loss function that fits our assumptions. Nevertheless, we would probably not recommend it in general as a loss function. As seen in the introduction, an estimator $\widehat \theta$ of a parameter $\theta$ may perform well in the sense that the associated probabilities $P_{\widehat \theta}$ and $P_{\theta}$ would be difficult to distinguish (say from a sample of size $10^{6}$) while $K(P_{\widehat \theta},P_{\theta})=+\infty$. 

The KL-divergence given by (\ref{eq-KLa}) can alternatively be defined via the following variational formula:
%
\begin{equation}\label{K-variat}
K(P,Q)=\sup_{f\in\cF}\cro{\E_{P}\cro{f}-\log \E_{Q}\pa{e^{f}}},
\end{equation}
which corresponds to (\ref{eq-lossGb}) with $\Lambda(Q,f)=\log\E_{Q}(e^{f})$ and $\cF$ the class of all measurable functions $f$ such that $|f|$ is bounded on $E$. Let $\mu$ be some reference positive measure on $E$ and $\sQ$ the set of all probabilities $P$ on $E$ which are absolutely continuous with respect to $\mu$ and such that $|\log(dP/d\mu)|\in\L_{\infty}(E,\mu)$. For $P=p\cdot\mu$ and $Q=q\cdot\mu$ in $\sQ$, equality holds in (\ref{K-variat}) for $f=f_{(P,Q)}=\log(p/q)$. Since, for $R=(P+Q)/2$, $\Lambda(P,f_{(R,P)})=\Lambda(Q,f_{(R,Q)})
=0$, {we deduce from \eref{T-formeG} that} $t_{(P,Q)}$ is proportional to
\[
{\cro{f_{(R,P)}-\Lambda(P,f_{(R,P)})}-\cro{f_{(R,Q)}-\Lambda(Q,f_{(R,Q)})}= \log\frac{r}{p}-\log\frac{r}{q}={\log\frac{q}{p}}}
\]
and therefore corresponds to the well-known likelihood ratio test. The following result {is} proven in Section~\ref{sect-pfs5}.
%
\begin{prop}\label{prop-KL}
Let $\sP$ be the set of all probabilities $S$ on $(E,\cE)$ which are dominated by $\mu$ and whose densities $s$ satisfy $\E_{S}\cro{\ab{\log s}}<+\infty$. Assume that $\sM=\{p\cdot\mu, p\in\cM\}$ is a countable subset of $\sQ$ and that {for all $p,q\in\cM$ and some constant $a>0$, }
\begin{equation}\label{eq-maj-min}
e^{-a}\le \frac{p}{q}(x)\le e^{a}\quad {\text{for all $x\in E$.}}
\end{equation}
The family $\sT(\ell,\sM)$ of functions $t_{(P,Q)}$ given by 
%
\begin{equation}\label{T-KL}
t_{(P,Q)}=\frac{1}{2a}\log\pa{\frac{q}{p}}\quad \text{for all $P,Q\in\sM$}
\end{equation}
satisfies Assumption~\ref{Hypo-2} with $a_{0}=a_{1}=1/(2a)$ and {$a_{2}=2a/[\tanh(a/2)]$} for the KL-loss $\ell(P,Q)=K(P,Q)$, $P,Q\in\sP$.
\end{prop}

Under~\eref{eq-maj-min}, the squared Hellinger distance and the Kullback-Leibler divergence turn out to be equivalent on $\sM$. {It is well-known that, whatever $P$ and $Q$, $2h^{2}(P,Q)\le K(P,Q)$. If, moreover $P$ and $Q$ belong to $\sM$ and \eref{eq-maj-min} holds, it follows from Lemma 7.23 in Massart~\citeyearpar{MR2319879} that}
\[
2h^{2}(P,Q)\le K(P,Q)\le 2(2+a)h^{2}(P,Q)\quad \text{for all $P,Q\in\sM$}.
\]
If the data are i.i.d.\ with distribution $P\et$, {the left-hand side inequality $2h^{2}(P\et,P)\le K(P\et,P)$ still holds for all $P\in\sM$ but \eref{eq-maj-min} does not imply anything about $K(P\et,P)$ which cannot therefore be compared to $h^{2}(P\et,P)$}. This means that the result of Theorem~\ref{thm-main02} for the Kullback-Leibler divergence cannot be deduced from {the one} established for the squared Hellinger distance.

\section{TV-estimators versus $\rho$-estimators}\label{sect-TVversusH}
As explained in Section~\ref{sect-ex1}, a nice feature of TV-estimators lies in their robustness properties. As described in details in Baraud {{\em et al.}}
\citeyearpar{MR3595933} and Baraud and Birg\'e~\citeyearpar{BarBir2018}, $\rho$-estimators also possess robustness properties except from the fact that these properties are expressed in terms of the Hellinger distance and not the TV one. Since these two distances are not equivalent in general, it is worth analyzing further the main differences between these {two types of} estimators. 


\subsection{Robustness and optimality}\label{sect-RO}
Let us go back to our Example~\ref{exa-gauss} in the simple situation where  the data are i.i.d.\ with distribution $P\et$ and $d=1$. Provided that $\norm{P\et-P_{m\et}}$ and $1/n$ are both small enough, the right-hand side of~\eref{eq-TVgauss2b} is smaller than 0.78 and we deduce that, with probability at least $1-e^{-\xi}$, the TV-estimator $\widehat m$ of $m\et$ satisfies 
%
\begin{equation}\label{eq-lestim-trans}
\ab{m\et-\widehat m}\le C\cro{\norm{P\et-P_{m\et}}+\sqrt{\frac{1+\xi}{n}}}
\end{equation}
for some universal constant $C>0$. 

Alternatively, in this statistical setting, we may use a $\rho$-estimator $\widetilde m$ for estimating $m\et$. By combining Corollary~3 of Baraud and Birg\'e~\citeyearpar{BarBir2018} with Proposition~42 of Baraud  {{\em et al.}}~\citeyearpar{MR3595933}, we obtain that the $\rho$-estimator 
$P_{\widetilde m}$ of $P\et$ satisfies 
\begin{equation}\label{eq-rho-est}
h^{2}(P_{m\et},P_{\widetilde m})\le C'\left[h^{2}(P\et,P_{m\et})+\frac{\log n+\xi}{n}\right]
\end{equation}
where $C'$ denotes some positive universal constant. {Since} 
\[
(1-e^{-1})\cro{\frac{(m-m')^{2}}{8}\wedge 1}\le h^{2}(P_{m},P_{m'})=1-e^{-\frac{(m-m')^{2}}{8}}
,
\]
we derive, as we did for \eref{eq-lestim-trans}}, that when $h^{2}(P\et,P_{m\et})$ and $1/n$ are small enough, with a probability at least $1-e^{-\xi}$, 
\begin{equation}\label{eq-rhoestim-trans}
\ab{m\et-\widetilde m}\le C''\left[h(P\et,P_{m\et})+\sqrt{\frac{\log n+\xi}{n}}\right]
\end{equation}
for some universal $C''>0$. 


If we forget about the logarithmic factor and the universal constants $C,C''$, the main difference between inequalities~\eref{eq-lestim-trans} and~\eref{eq-rhoestim-trans} lies in the expression of the approximation terms $\norm{P\et-P_{m\et}}$ and $h(P\et,P_{m\et})$. Since, for all probabilities $P,Q$, $\norm{P-Q}\le \sqrt{2}h(P,Q)$, the accuracy of $\widehat m$ cannot be much worse than that of $\widetilde m$ but it can indeed be much better: when $P\et=(1-\alpha)P_{m\et}+\alpha R$ for some small value of  $\alpha\in (0,1)$ and a probability $R$ on $\R$ which is singular with respect {to $P_{m\et}$}, we obtain that 
\[
\norm{P\et-P_{m\et}}=\alpha\norm{P_{m\et}-R}=\alpha
\]
while 
\begin{align*}
h(P\et,P_{m\et})=\sqrt{1-\sqrt{1-\alpha}}\sim\sqrt{\alpha/2}{\quad\text{when $\alpha$ is small}}.
\end{align*}
For small values of $\alpha$, $h(P\et,P_{m\et})$ is therefore much larger than $\norm{P\et-P_{m\et}}$. While~\eref{eq-lestim-trans} warrants that the performance of the TV-estimator {$\widehat{m}$} remains stable as long as $\alpha$ is small compared to $1/\sqrt{n}$, the bound we get on the accuracy on the $\rho$-estimator {$\widetilde m$} deteriorates as soon as $\alpha$ becomes large compared to $(\log n)/n$.  From this point of view, the estimator $\widetilde m$ appears less robust than $\widehat m$. This disappointing result (for $\rho$-estimators) is actually not restricted to this Gaussian model and can actually be generalized to many other situations for which the TV-distance and the Hellinger one are equivalent on the model $\sM$.

{This} apparent superiority of TV-estimators over $\rho$-estimators must nevertheless be put into perspective in the light of the following example. Assume that the data $X_{1},\ldots,X_{n}$ are truly i.i.d.\ from a translation model $\cM=\{p_{\theta}=p(\cdot-\theta),\; \theta\in \Q\}$ where the density $p$ satisfies the following condition. 
\begin{ass}\label{Ass-transl}
There exists a positive decreasing function $f$ on $(0,+\infty)$ such that $p(x)=f(|x|)$ for all $x\in\R\setminus\{0\}$.  
\end{ass}
In this {case, it is not difficult to compute a TV-estimator for the location parameter $\theta$}. Putting aside the fact that our statistical model is parametrized by $\Q$ and not $\R$ in order to make it countable,  the empirical median turns out to be a TV-estimator.  More precisely, let $X_{(1)}<X_{(2)}<\ldots<X_{(n)}$ be the order statistics associated to the $n$-sample $X_{1},\ldots,X_{n}$ with $n\ge 2$ and define the empirical median as $X_{(\lceil n/2\rceil)}$ where 
\[
\lceil x \rceil=\min\{k\in\N,\; k\ge x\}\quad \text{for all $x>0$,}
\]
that is
\begin{equation}\label{eq-mediane}
\sum_{i=1}^{n}\1_{X_{i}< X_{(\lceil n/2\rceil)}}<\frac{n}{2}\le \sum_{i=1}^{n}\1_{X_{i}\le X_{(\lceil n/2\rceil)}}.
\end{equation}
The proof of the following result is {provided in} Section~\ref{pf-propTV-mediane}.
\begin{prop}\label{prop-med}
Let Assumption~\ref{Ass-transl} be satisfied. Any element $\widehat \theta\in \Q$ that satisfies $X_{(\lceil n/2\rceil)}<\widehat \theta<X_{(\lceil n/2\rceil+1)}$ is a TV-estimator of $\theta$ for the choice $\epsilon=1/2$. 
\end{prop}
It is {nevertheless easy} to find an example of a translation model satisfying Assumption~\ref{Ass-transl} for which the empirical median is sub-optimal. The choice 
\[
p:x\mapsto \frac{\alpha}{2(1+\alpha)}\cro{\frac{1}{|x|^{1-\alpha}}\wedge \frac{1}{x^{2}}}\1_{|x|>0}\quad \text{with $\alpha\in (0,1)$}
\]
actually suits. For this density, one can check that the empirical median converges at rate $n^{-1/(2\alpha)}$ (with respect to the Euclidean loss) while the minimax rate is actually of order $n^{-1/\alpha}$. In contrast to the empirical median, the $\rho$-estimator converges to the location parameter at the optimal rate $n^{-1/\alpha}$ up to a possible logarithmic factor. 
 
{As a matter of conlusion}, TV-estimators are robust but not necessarily optimal.

\subsection{Logarithmic factors}
The above discussion {did put} aside the logarithmic factor that appears in the right-hand side of~\eref{eq-rho-est} compared to~\eref{eq-lestim-trans}. In fact, the results obtained for the Hellinger loss in Baraud and Birg\'e~\citeyearpar{BarBir2018, MR3565484} often involve such logarithmic factors. These factors turn out to be sometimes necessary when one uses the Hellinger loss. For example, let  $\cM$ be a countable and dense subset (with respect {to} the Hellinger distance) of the set $\overline \cM$ of all {probability densities with respect to the Lebesgue measure on $\R$ which are piecewise constant with respect to some partition of $\R$ into at most $d\ge 1$ intervals. This} means that the elements of $\overline \cM$ are of the form 
\[
\sum_{i=1}^{d}a_{i}\1_{(b_{i},b_{i+1}]}\quad \text{with }-\infty<b_{0}<\ldots<b_{d+1}<+\infty
\]
and $a_{1},\ldots,a_{d}\in\R_{+}$ satisfying $\sum_{i=1}^{d+1}a_{i}(b_{i+1}-b_{i})=1$. It is proven in Baraud and Birg\'e~\citeyearpar{MR3565484} that, if the data $X_{1},\ldots,X_{n}$ are i.i.d.\ with a density $p\et\in\overline \cM$, the $\rho$-estimator $\widetilde p$ on $\cM$ satisfies, for some universal constant $C>0$,
\begin{equation}\label{eq-histo-log}
\sup_{p\et\in\overline \cM_{d}}\E\cro{h^{2}(p\et,\widetilde p)}\le C{\frac{d}{n}
{\max\left\{\log^{3/2}\left(\frac{n}{d}\right),1\right\}}.}
\end{equation}
It {has also been} shown in Birg\'e and Massart~\citeyearpar{MR1653272}[Proposition~2] that the minimax rate is at least $(d/n)\max\{\log(n/d),1\}$ (up to some universal constant) when $d\ge 9$ . The logarithmic factor appearing in the right-hand side of~\eref{eq-histo-log} is therefore necessary (with a possibly smaller power though). A look at the proof of Proposition~2 in  Birg\'e and Massart shows that this logarithmic factor is due to some combinatoric arguments based on the fact that $\cM$ contains histograms {built} on possibly irregular partitions of $[0,1]$.

Surprisingly, this logarithmic factor disappears for the TV-loss. It is easy to see that for $\overline p,q\in\cM$, the sets {$\{\overline{p}> q\}$} are the union of at most $d+1$ intervals and Assumption~\ref{Hcor-estimD} is therefore satisfied with $V(\overline p)=2(d+1)$ for all $\overline p\in\cM$. {Proposition~\ref{prop-histo} and more precisely~\eref{eqcor-estimDHisto1}, implies} that the TV-estimator of $\widehat p$ satisfies, for some numerical constant $C'>0$,
\begin{equation}\label{eq-histo-log0}
\sup_{p\et\in\overline \cM_{d}}\E\cro{\norm{p\et-\widehat p}_{\lambda,1}^{2}}\le C'\frac{d}{n}.
\end{equation}
One can prove that this bound is optimal in the sense that the minimax rate with respect to {the squared TV loss over $\cM$ is not smaller than} $cd/n$ for some numerical constant $c>0$ when $d\ge 2$. This means that the minimax rates with respect to the Hellinger and TV-losses may differ from at least a logarithmic factor. 

\section{Application to robust testing}\label{sect-RobustTest}

\subsection{The two-points model and robust tests\label{sect-RB1}}
As already mentioned, our estimation procedure is based on a suitable test between two distinct elements of our statistical model. The aim of this section is to analyse the properties of these tests, that {is to} evaluate their errors of first and second kinds, not only when the true probability is equal to one of the two {distributions to be tested} but more generally when it is close enough to one of them with respect to {the} loss $\ell$. We shall therefore {analyze} the robustness properties of these tests. 

Given two distinct elements $\gP,\gQ$ in $\sbM$, we define the test $\Phi_{(\gP,\gQ)}$ between $\gP$ and $\gQ$ as 
\begin{equation}\label{def-Phi}
\Phi_{(\gP,\gQ)}(\bsX)=
\begin{cases}
1 \quad\text{if}\quad\gT(\bsX,\gP,\gQ)>0\\
0\quad\text{if}\quad\gT(\bsX,\gP,\gQ)<0.
\end{cases}
\end{equation}
This means that we decide that $\gP$ is closer to $\gP\et$ when $\Phi_{(\gP,\gQ)}(\bsX)=0$ and that $\gQ$  is closer to $\gP\et$ when $\Phi_{(\gP,\gQ)}(\bsX)=1$, the choice between $\gP$ and $\gQ$ being unimportant, as well as the value of $\Phi_{(\gP,\gQ)}$, when $\gT(\bsX,\gP,\gQ)=0$. The following result is proven in Section~\ref{sect-pfs12}.
%
\begin{prop}\label{prop-TestRob}
Let Assumption~\ref{Hypo-1} hold and $\gP\et\in\sbP$ be such that $\gamma=a_{0}\gell(\gP\et,\gP)/[a_{1}\gell(\gP\et,\gQ)]<1$. Then
%
\begin{equation}
\P\cro{\Phi_{(\gP,\gQ)}(\bsX)=1}\le  \exp\cro{-\frac{2\gell^{2}(\gP\et,\gQ)}{n}\cro{a_{1}(1-\gamma)}^{2}}.
\label{prop-TestRob1a}
\end{equation}
If, moreover, Assumption~\ref{Hypo-2}-\ref{cond-3} is satisfied 
\begin{align}
{\lefteqn{\P\cro{\Phi_{(\gP,\gQ)}(\bsX)=1}}\hspace{15mm}}\nonumber\\
&\le\exp\left[-\frac{\gell(\gP\et,\gQ)}{2}
\frac{a_{1}(1-\gamma)^{2}}{[(1-\gamma)/3]+[(1+\gamma (a_{1}/a_{0}))(a_{2}/a_{1})]}\right].
\label{prop-TestRob2a}
\end{align}
\end{prop}
Inequalities~\eref{prop-TestRob1a} and~\eref{prop-TestRob2a} both say that if $\gP\et$ is close enough to $\gP$ and far enough from $\gQ$ with respect to the loss $\gell$, the test $\Phi_{(\gP,\gQ)}$ decides $\gP$ with probability close to 1. In view of the symmetry of the assumptions with respect to $\gP$ and $\gQ$, it suffices to exchange their roles to bound $\P\cro{\Phi_{(\gP,\gQ)}(\bsX)=0}$ now assuming that $\gamma=a_{0}\gell(\gP\et,\gQ)/[a_{1}\gell(\gP\et,\gP)]<1$.

Recalling from Section~\ref{sect-2.4} that $a_{1}\le a_{0}$, note that {one cannot say anything about the performance of the test if 
\[
a_{1}/a_{0}\le\gell(\gP\et,\gP)/\gell(\gP\et,\gQ)\le a_{0}/a_{1}.
\]
However, this} is a situation where $\gell(\gP\et,\gP)$ and $\gell(\gP\et,\gQ)$ are of the same order {(in most cases that we considered $a_{0}=3a_{1}$)} which means that choosing $\gP$ or $\gQ$ is actually unimportant. 

In order to comment {on} these results further, let us consider the density framework with $\gP=\gP\et=(P\et)\on$ and $\gQ=Q\on$ for some probability $Q$ on $(E,\cE)$. Looking at~\eref{prop-TestRob1a}, {we see} that the test accepts the hypothesis $P\et=P$ with probability close to one as soon as $\ell(P\et,Q)=\gell(\gP\et,\gQ)/n$ is large enough compared to $(1/\sqrt{n})\vee \ell(P\et,P)$. The situation is even better {when Assumption~\ref{Hypo-2}-\ref{cond-3} holds since \eref{prop-TestRob2a} shows that it is} enough that $\ell(P\et,Q)$ be {large compared} to $(1/n)\vee \ell(P\et,P)$. It is well-known, mainly from the work of Le Cam~\citeyearpar{MR0334381}, that it is impossible to distinguish between two probabilities $P$ and $Q$ from an $n$-sample when the Hellinger distance $h(P,Q)$ is small enough compared to $1/\sqrt{n}$. As a consequence, the test $\Phi_{(\gP,\gQ)}$ is optimal under Assumption~\ref{Hypo-1} when the loss $\ell$ is of the order of the Hellinger distance and optimal under Assumption~\ref{Hypo-2} when it is of {the order of the squared} Hellinger distance.  

As we have seen earlier, most loss functions of interest are actually powers of some distance on $\sP$. For illustration, let us focus on the case of $\ell=h^{2}$ for which Assumption~\ref{Hypo-2} holds, in which case (\ref{prop-TestRob2a}) becomes, according to 
Proposition~\ref{prop-hel2},
\begin{equation}
\P\cro{\Phi_{(\gP,\gQ)}(\bsX)=1}\le\exp\left[-\frac{3\left(\sqrt{2}-1\right)(1-\gamma)^{2}
nh^{2}(P\et,Q)}{4\left[9\sqrt{2}+10+\gamma(9\sqrt{2}-10)\right]}\right],
\label{eq-TestRob3}
\end{equation}
provided that
\begin{equation}
\gamma=\left(3+2\sqrt{2}\right)\frac{h^{2}(P\et,P)}{h^{2}(P\et,Q)}<1.
\label{eq-TestRob4}
\end{equation}

An interesting feature of this result lies in the fact that the test $\Phi_{(\gP,\gQ)}$ is powerful even in the situation where both $h(P\et,P)$ and $h(P\et,Q)$ are larger than $h(P,Q)/2$ provided that (\ref{eq-TestRob4}) is satisfied. In contrast, a test between the two disjointed Hellinger balls $\{R\in\sP,\; h(P,R)\le r\}$ and $\{R\in\sP,\; h(Q,R)\le r\}$, as proposed in Birg\'e~\citeyearpar{MR764150}[Section~5] and Birg\'e~\citeyearpar{Robusttests}, would require the condition $r<h(P,Q)/2$ and could not cope with the situation described above.  In order to provide a concrete example of such a situation, let $R$ and $P\et$ be two singular probabilities, $\alpha=0.1$, $P=\cos^{2}(2\alpha)P\et+\sin^{2}(2\alpha)R$ and $Q=\cos^{2}(6\alpha)P\et+\sin^{2}(6\alpha)R$. Then $h(P,Q)=\sqrt{2}\sin(2\alpha)\approx 0.281$, $h(P\et,P)=\sqrt{2}\sin \alpha\approx 0.141$, $h(P\et,Q)=\sqrt{2}\sin(3\alpha)\approx 0.418$, consequently both $h(P\et,P)$ and $h(P\et,Q)$ are larger than $h(P,Q)/2$. Since $\gamma<0.666$, our test between $P$ and $Q$ is powerful as soon as $n$ is sufficiently large. This example confirms that our procedure differs from the tests between balls that were proposed by Birg\'e~\citeyearpar{MR764150,Robusttests} and Huber~\citeyearpar{MR0185747} for the Hellinger and total variation distances respectively. 

More generally, if $\ell=\dis^{j}$ for some distance $\dis$ and $j\ge1$, the test will perform nicely if $\dis(P\et,Q)/\dis(P\et,P)$ is large enough, even if $\dis(P\et,Q)$ is much larger than $\dis(P,Q)/2$.

\subsection{The case of a loss satisfying a variational formula}
In this section, we shall more specifically consider the case of a loss $\ell$ of the form~\eref{eq-lossG} and assume that we have at disposal {an} $n$-sample $X_{1},\ldots,X_{n}$ with common distribution $P\et$. Since $\ell$ behaves like a distance, it is interesting to study the properties of our test for the problem of testing two {disjointed} $\ell$-balls, i.e. between $\{S\in\sP,\; \ell(S,P)\le r\}$ and $\{S\in\sP,\; \ell(S,Q)\le r\}$ for $P,Q\in \overline\sP$ and $r<\ell(P,Q)/2$. The following result is proven in Section~\ref{sect-pfs21}.
\begin{prop}\label{prop-test-casVar}
Assume that the loss $\ell$ satisfies~\eref{eq-lossG} and Assumption~\ref{Hypo-classF} with  $\sM=\{P,Q\}$ for $P,Q\in\overline \sP$. Let $\Phi_{(\gP,\gQ)}$ be defined by~\eref{def-Phi} with $\gT(\bsX,\gP,\gQ)=\sum_{i=1}^{n}t_{(P,Q)}(X_{i})$ and $t_{(P,Q)}$ given by~\eref{eq-def-phiPQ}. For all $P\et\in\sP$ such that $\ell(P\et,P)\le \kappa\ell(P,Q)$ with $\kappa\in [0,1/2)$, 
\begin{equation}\label{err-test-gene}
\P\cro{\Phi_{(\gP,\gQ)}=1}\le \exp\cro{-\frac{(1-2\kappa)^{2}}{2b^{2}}n\ell^{2}(P,Q)}.
\end{equation}
\end{prop}
{As a matter of} illustration, let us consider the case where $\overline \sP=\sP$ is the set of all probabilities on $(E,\cE)$ and $\ell$ is the TV-loss. Then $b=1$ and given two distinct probabilities $P,Q$ in $\sP$, we derive from~\eref{phi-TV0} that the test statistic $\gT(\bsX,\gP,\gQ)$ writes as
\[
\frac{n}{2}\cro{\frac{1}{n}\sum_{i=1}^{n}\1_{q>p}(X_{i})-Q(q>p)}-\frac{n}{2}\cro{\frac{1}{n}\sum_{i=1}^{n}\1_{p>q}(X_{i})-P(p>q)}.
\]
%
%
One may compare the test $\Phi_{(\gP,\gQ)}$ with that proposed by Devroye and Lugosi~\citeyearpar{MR1843146} [Chapter 6] which is based on the test statistic
\[
T'(\bsX,\gP,\gQ)= \ab{\frac{1}{n}\sum_{i=1}^{n}\1_{q>p}(X_{i})-Q(q>p)}- \ab{\frac{1}{n}\sum_{i=1}^{n}\1_{q>p}(X_{i})-P(q>p)}
\]
and rejects $P$ if and only if  $T'(\bsX,\gP,\gQ)>0$. Unlike ours, the test proposed by Devroye and Lugosi is not symmetric with respect to $P$ and $Q$. For example, when $p(X_{i})=q(X_{i})$ for all $i$, the test always chooses $Q$, since $T'(\bsX,\gP,\gQ)=Q(q>p)-P(q>p)>0$, while ours decides  $P$ or $Q$ on the basis of the sign of $Q(q>p)-P(p>q)=Q(q\ge p)-P(p\ge q)$. 

\subsection{Case of the $\L_{j}$-loss}
We assume here that $\gP=P\on$ and $\gQ=Q\on$ where $P,Q$ are not necessarily probabilities but possibly signed measures with densities $p$ and $q$ with respect to some dominating measure $\mu$. We consider the $\L_{j}$-loss for $j\in (1,+\infty)$ and assume that $p$ and $q$ belong to $\sL_{j}(E,\mu)\cap\sL_{1}(E,\mu)$. Clearly, \eref{eq-LinftyLj} is satisfied {for the model $\sM=\{P,Q\}$} as soon as $R=\norm{p-q}_{\infty}/\norm{p-q}_{\mu,j}<+\infty$ (assuming $P\ne Q$) and it follows from Corollary~\ref{cor-Lj1} that 
{
\begin{align*}
\frac{\gT(\bsX,\gP,\gQ)}{n}&=\frac{1}{2}\cro{\frac{1}{n}\sum_{i=1}^{n}\pa{\frac{\sigma |p-q|^{j-1}}{\norm{p-q}_{\infty}^{j-1}}}(X_{i})-\int_{E}\frac{\sigma |p-q|^{j-1}}{\norm{p-q}_{\infty}^{j-1}}\frac{p+q}{2}\,d\mu}
\end{align*}
}
where $\sigma(x)=\1_{q>p}(x)-\1_{p>q}(x)$ for all $x\in E$. Note that for $j=2$, 
{
\begin{align*}
4\norm{p-q}_{\infty}\frac{\gT(\bsX,\gP,\gQ)}{n}={\cro{\frac{2}{n}\sum_{i=1}^{n}q(X_{i})-\norm{q}_{\mu,2}^{2}}-\cro{\frac{2}{n}\sum_{i=1}^{n}p(X_{i})-\norm{p}_{\mu,2}^{2}}}
\end{align*}
}
and the test $\Phi_{(\gP,\gQ)}$ between $P$ and $Q$ is {the one} associated to the classical $\L_{2}$-contrast function. 

We deduce from Proposition~\ref{prop-TestRob} and Corollary~\ref{cor-Lj1} the following result. 
\begin{prop}\label{prop-test-Lj}
Let $j\in (1,+\infty)$, $P=p\cdot \mu,\ Q=q\cdot \mu$ be two distinct and possibly signed measures on $(E,\cE)$ with $p,q\in\sL_{j}(E,\mu)\cap \sL_{1}(E,\mu)$. Assume that $X_{1},\ldots,X_{n}$ are independent with respective densities 
$p_{1}\et,\ldots,p_{n}\et\in  \sL_{j}(E,\mu)$. If 
\[
\gamma=\frac{3\sum_{i=1}^{n}\norm{p_{i}\et-p}_{\mu,j}}{\sum_{i=1}^{n}\norm{p_{i}\et-q}_{\mu,j}}<1\quad \text{and}\quad R=\frac{\norm{p-q}_{\infty}}{\norm{p-q}_{\mu,j}}<+\infty,
\]
the test $\Phi_{(\gP,\gQ)}$ defined by~\eref{def-Phi} satisfies 
\[
\P\cro{\Phi_{(\gP,\gQ)}(\bsX)=1}\le  \exp\cro{-\frac{(1-\gamma)^{2}n}{8R^{2(j-1)}}\pa{\frac{1}{n}\sum_{i=1}^{n}\norm{p_{i}\et-q}_{\mu,j}}^{2}}.
\]
In particular, if $X_{1},\ldots,X_{n}$ are i.i.d.\ with density $p\et\in\sL_{j}(E,\mu)$, 
\[
\P\cro{\Phi_{(\gP,\gQ)}(\bsX)=1}\le  \exp\cro{-\frac{(1-\gamma)^{2}n}{8R^{2(j-1)}}\norm{p\et-q}_{\mu,j}^{2}}
\]
provided that $\gamma=3\norm{p\et-p}_{\mu,j}/\norm{p\et-q}_{\mu,j}<1$.
\end{prop}

\section{Proofs of Theorems~\ref{thm-main01} and~\ref{thm-main02}}\label{sect-pfsth}
Let $\overline \gP$ be an arbitrary point in $\sbM$, $\kappa\in [0,1)$ and $\zeta>0$ to be chosen later on. For $\gP\in\sbM$ and $\gx=(x_{1},\ldots,x_{n})\in \gE$, let us set 
\begin{align*}
\gZ(\gx,\gP)&=\sup_{\gQ\in\sbM}\cro{\st\gT(\gx,\gP,\gQ)-(1-\kappa)\E\cro{\gT(\bsX,\gP, \gQ)}}-\zeta\\
&=\sup_{\gQ\in\sbM}\cro{\st(1-\kappa)\E\cro{\gT(\bsX,\gQ,\gP)}-\gT(\gx,\gQ,\gP)}-\zeta.
\end{align*}
It follows from  \eref{eq-Estat1} that 
\begin{align}
\sup_{\gQ\in\sbM}\E\cro{\gT(\bsX,\gP, \gQ)}&\le a_{0}\gell(\gP\et,\gP)-a_{1}\gell(\gP\et,\sbM)\label{eq-biais1}
\end{align}
and for all $\gQ\in\sbM$, 
\begin{align*}
(1-\kappa)a_{1}\gell(\gP\et,\gQ)\le&\; (1-\kappa)a_{0}\gell(\gP\et,\overline \gP)+(1-\kappa)\E\cro{\gT(\bsX,\gQ,\overline \gP)}\\
= &\;(1-\kappa)a_{0}\gell(\gP\et,\overline \gP)+(1-\kappa)\E\cro{\gT(\bsX,\gQ,\overline \gP)}\\
&\;-\gT(\bsX,\gQ,\overline \gP)+\gT(\bsX,\gQ,\overline \gP)\\
\le  &\;(1-\kappa)a_{0}\gell(\gP\et,\overline \gP)+\gZ(\bsX,\overline \gP)+\gT(\bsX,\gQ,\overline \gP)+\zeta\\
\le &\;(1-\kappa)a_{0}\gell(\gP\et,\overline \gP)+\gZ(\bsX,\overline \gP)+\gT(\bsX,\gQ)+\zeta.
\end{align*}
This last inequality applies in particular to $\gQ=\widehat \gP\in \sbE(\bsX)$ and, since
\[
\gT(\bsX,\widehat \gP)\le \inf_{\gP'\in\sbM}{\gT(\bsX,\gP')}+\epsilon \le \gT(\bsX,\overline \gP)+\epsilon
\]
we deduce that 
\begin{equation}
(1-\kappa)a_{1}\gell(\gP\et,\widehat \gP)\le (1-\kappa)a_{0}\gell(\gP\et,\overline \gP)+\gZ(\bsX,\overline \gP) +\gT(\bsX,\overline \gP)+\zeta+\epsilon.
\label{eq-fond00}
\end{equation}
%
{We derive from~\eref{eq-biais1} that} 
\begin{align*}
\gT(\bsX,\overline \gP)&=\sup_{\gQ\in\sbM}\gT(\bsX,\overline \gP,\gQ)\\
&\le \sup_{\gQ\in\sbM}\cro{\st\gT(\bsX,\overline \gP,\gQ)-(1-\kappa)\E\cro{\gT(\bsX,\overline \gP,\gQ}}-\zeta\\
&\quad +(1-\kappa)\sup_{\gQ\in\sbM}\E\cro{\gT(\bsX,\overline \gP,\gQ}+\zeta\\
&\le \gZ(\bsX,\overline \gP)+(1-\kappa)\cro{a_{0}\gell(\gP\et,\overline\gP)-a_{1}\gell(\gP\et,\sbM)}+\zeta,
\end{align*}
which, {together} with~\eref{eq-fond00}, leads to 
\begin{align}
(1-\kappa)a_{1}\gell(\gP\et,\widehat \gP)&\le (1-\kappa)\cro{2a_{0}\gell(\gP\et,\overline\gP)-a_{1}\gell(\gP\et,\sbM)}\label{eq-fond}\\
&\quad + 2\zeta +\epsilon+2\gZ(\bsX,\overline \gP).\nonumber 
\end{align}
{The following lemma, to be proven in Section~\ref{ProofLem-xi1}, provides a control of
$\gZ(\bsX,\overline\gP)$ involving $\gw(\overline \gP)$ as defined in (\ref{def-wpbar}).}
\begin{lem}\label{lem-xi1}
Under the assumptions of Theorem~\ref{thm-main01}, for the choices $\kappa=0$ and $\zeta=\gw(\overline \gP)+\sqrt{n\xi/2}$, $\gZ(\bsX,\overline\gP)\le 0$ with a probability at least $1-e^{-\xi}$.
\end{lem}
{To complete the proof of Theorem~\ref{thm-main01} we argue as follows. Choosing $\zeta$ and $\kappa$ as in Lemma~\ref{lem-xi1} implies that $\gZ(\bsX,\overline \gP)\le0$ with a probability at least $1-e^{-\xi}$, which, together with~\eref{eq-fond}, leads to
\begin{align*}
a_{1}\gell(\gP\et,\widehat \gP)&\le 2a_{0}\gell(\gP\et,\overline\gP)-a_{1}\gell(\gP\et,\sbM)+2\zeta+\epsilon\\
&\le 2a_{0}\gell(\gP\et,\overline\gP)-a_{1}\gell(\gP\et,\sbM)+2\gw(\overline \gP)+\sqrt{2n\xi}+\epsilon
\end{align*}
and (\ref{eq-thm01}) follows from a division by $a_{1}>0$}. To derive (\ref{eq-BornRisk}), we use the equality $\E[Y]=\int_{0}^{+\infty}\P[Y>t]\,dt$ which holds for any nonnegative random variable $Y$, then an integration with respect to $\xi$ and conclude since $\overline\gP$ is arbitrary in $\sbM$.

To prove Theorem~\ref{thm-main02} we fix $\kappa=1/2$ in the definition of $\gZ(\gx,\gP)$, in which case (\ref{eq-fond}) becomes
\begin{equation}
a_{1}\gell(\gP\et,\widehat \gP)\le2a_{0}\gell(\gP\et,\overline\gP)-a_{1}\gell(\gP\et,\sbM)
+4\zeta+2\epsilon+4\gZ(\bsX,\overline \gP),
\label{eq-fond1}
\end{equation}
and, given a positive number $y_{0}$ to be chosen later on, we set, for all $j\in \N$, 
%
\begin{equation}\label{def-mmj}
r_{j}=jy_{0},\qquad\sbM_{j}=\ac{\gQ\in\sbM,\; r_{j}\le \gell(\gP\et,\gQ)<r_{j+1}}
\end{equation}
and, for $\gx=(x_{1},\ldots,x_{n})\in \gE$,
\begin{equation}
\gZ_{j}(\gx,\overline \gP)=\sup_{\gQ\in\sbM_{j}}\ab{\gT(\gx,\overline \gP,\gQ)-\E\cro{\gT(\bsX,\overline \gP,\gQ)}}.
\label{def-zmj}
\end{equation}
%
We then deduce from \eref{eq-Estat1} that, for $\gQ\in \sbM_{j}$,
\[
\E\cro{\gT(\bsX,\overline \gP,\gQ)}\le a_{0}\gell(\gP\et,\overline \gP)-a_{1}\gell(\gP\et,\gQ)\le a_{0}\gell(\gP\et,\overline \gP)-a_{1}y_{0}j,
\]
from which we derive, since $\sbM=\bigcup_{j\ge 0}\sbM_{j}$, that
\begin{align}
\gZ(\bsX,\overline \gP)&= \sup_{j\in\N}\sup_{\gQ\in\sbM_{j}}\cro{\st\gT(\bsX,\overline \gP,\gQ)-(1/2)\E\cro{\gT(\bsX,\overline{\gP}, \gQ)}}-\zeta\nonumber\\
&\le \sup_{j\in\N}\cro{\gZ_{j}(\bsX,\overline \gP)+(1/2)\sup_{\gQ\in\sbM_{j}}\E\cro{\gT(\bsX,\overline \gP,\gQ)}}-\zeta\nonumber\\
&\le(a_{0}/2)\gell(\gP\et,\overline \gP)+\sup_{j\in\N}\Xi_{j},
\label{eq-pf4}
\end{align}
with
\begin{equation}\label{def-xij}
\Xi_{j}=\gZ_{j}(\bsX,\overline \gP)-(a_{1}/2)y_{0}j-\zeta.
\end{equation}
{In order} to control the random variables $\Xi_{j}$ for $j\in\N$, we use following lemma to be proven in Section~\ref{ProofLem-xi2}.  
\begin{lem}\label{lem-xi2}
Under the assumptions of Theorem~\ref{thm-main02}, let
%
\begin{equation}\label{def-zeta2}
y_{0}=D(\overline\gP)\quad\text{and}\quad\zeta=\frac{a_{1}}{4}\gell(\gP\et,\overline\gP)+\frac{a_{1}y_{0}}{2}+2\pa{1+\frac{4a_{2}}{a_{1}}}\xi.
\end{equation}
Then,
\[
\P\cro{\Xi_{j}>0}\le 2^{-(j+1)}e^{-\xi}\quad\text{for all $j\ge 0$.}
\]
\end{lem}
{Under the assumptions of Lemma~\ref{lem-xi2}, we derive} that, with a probability at least $1-e^{-\xi}$, $\sup_{j\in\N}\Xi_{j}\le 0$, in which case it follows {from~\eref{eq-fond1},~\eref{eq-pf4} and \eref{def-zeta2} that 
\begin{align*}
a_{1}\gell(\gP\et,\widehat \gP)&\le2a_{0}\gell(\gP\et,\overline\gP)-a_{1}\gell(\gP\et,\sbM)+4\zeta+2\epsilon+2a_{0}\gell(\gP\et,\overline \gP)\\
&\le\cro{4a_{0}+a_{1}}\gell(\gP\et,\overline\gP)-a_{1}\gell(\gP\et,\sbM)+2a_{1}D(\overline\gP)\\&\quad+8\cro{1+(4a_{2}/a_{1})}\xi+2\epsilon.
\end{align*}
and (\ref{thm2-b1}) follows, which concludes the proof of Theorem~\ref{thm-main02}.

\subsection{Proof of Lemma~\ref{lem-xi1}}\label{ProofLem-xi1}
If $\kappa=0$, for all $\gx=(\etc{x})\in\gE$,
\begin{equation}
\gZ(\gx,\overline\gP)=\sup_{\gQ\in\sbM}\ab{\sum_{i=1}^{n}\pa{t_{(\overline P_{i},Q_{i})}(x_{i})-\E\cro{t_{(\overline P_{i},Q_{i})}(X_{i})}}}-\zeta
\label{eq-simexp}
\end{equation}
and it follows from \eref{def-wpbar} that
\begin{equation}
 \E\cro{\gZ(\bsX,\overline \gP)}=\E\cro{\sup_{\gQ\in\sbM}\ab{\overline \gZ(\bsX,\overline \gP,\gQ)}}-\zeta \le\gw(\overline \gP)-\zeta{=-\sqrt{n\xi/2}}.
\label{eq-espZ}
\end{equation}
Under Assumption~\ref{Hypo-1}-\ref{cond-4}, for all
$i\in\{1,\ldots,n\}$, $\gQ\in\sbM$ and $x,x'\in E$ the quantity $\left|t_{(\overline P_{i},Q_{i})}(x)-t_{(\overline P_{i},Q_{i})}(x')\right|$ is bounded by $1$ {so that, for all $\gx\in\gE$} and $x_{i}'\in E$
\[
\ab{\gZ((x_{1},\ldots,x_{i},\ldots,x_{n}),\overline \gP)-\gZ((x_{1},\ldots,x_{i}',\ldots,x_{n}),\overline \gP)}\le 1.
\]
The random variables $X_{1},\ldots,X_{n}$ being independent,  Theorem~5.1 of Massart~\citeyearpar{MR2319879} applies to {the function} $\gx\to \gZ(\gx,\overline \gP)$, {showing that, with a probability at least $1-e^{-\xi}$,}
{
\begin{equation}
\gZ(\bsX,\overline \gP)\le \E\cro{\gZ(\bsX,\overline \gP)}+\sqrt{n\xi/2}\le0
\label{eq-zbarmj}
\end{equation}
}
{by~\eref{eq-espZ}}, which concludes our proof.
\subsection{Proof of Lemma~\ref{lem-xi2}}\label{ProofLem-xi2}
Let us recall that {$\gw(\overline\gP,y)\le c_{1}y$ for $y>D(\overline \gP)=y_{0}$ by (\ref{def-d}) with $c_{1}$ given by (\ref{def-c1p})}. Since the mapping $y\mapsto \gw(\overline\gP,y)$ defined by~\eref{eq-wpbary} is nondecreasing, 
\[
\gw(\overline\gP,y_{0})\le \gw(\overline\gP,y)\le c_{1}y\quad \text{for all }y>y_{0},
\]
so that the above inequality still holds for $y=y_{0}$. Since, for all $j\in\N$, $\sbM_{j} \subset \sbB(\gP\et,r_{j+1})$ with $r_{j+1}=(j+1)y_{0}$, we derive from \eref{eq-wpbary} that, whatever $j\in\N$, 
\begin{align}
\E\left[\gZ_{j}(\bsX,\overline \gP)\right]=\E\cro{\sup_{\gQ\in\sbM_{j}}\ab{\overline \gZ(\bsX,\overline \gP,\gQ)}}\le \gw(\overline\gP,r_{j+1})\le c_{1}r_{j+1}.
\label{eq-propym102}
\end{align}
Let us {now} recall the following version of Talagrand's inequality that can be found in Baraud, Birg\'e and Sart~\citeyearpar{MR3595933}.

\begin{prop}\label{talagrand}
Let $T$ be some finite set, $U_{1},\ldots,U_{n}$ be independent centered random vectors with values 
in $\R^{|T|}$ and $Z=\sup_{t\in T}\ab{\sum_{i=1}^{n}U_{i,t}}$. If, for some positive numbers $b$ and $v$, 
\[
\max_{i=1,\ldots,n}\ab{U_{i,t}}\le b\qquad\mbox{and}\qquad
\sum_{i=1}^{n}\E\left[U^2_{i,t}\right]\le v^{2}\ \quad\mbox{for all }t\in T,
\] 
then, for all positive {numbers} $c$ and $z$,
\begin{equation}
\P\left[Z\le(1+c){\mathbb E}(Z)+(8b)^{-1}cv^2+2\left(1+8c^{-1}\right)bz\right]\ge1-e^{-z}.
\label{massart3}
\end{equation}
\end{prop}
The above result extends to countable sets $T$ (by monotone convergence) and we may therefore take $T=\sbM_{j}$, $U_{i,\gQ}=t_{(\overline P_{i},Q_{i})}(X_{i})-\E\cro{t_{(\overline P_{i},Q_{i})}(X_{i})}$ for all $i\in\{1,\ldots,n\}$, so that {$Z=\gZ_{j}(\bsX,\overline \gP)$}, and $b=1$ {by} Assumption~\ref{Hypo-1}-\ref{cond-4}. {Furthermore, Assumption~\ref{Hypo-2}-\ref{cond-3} and the definition of $\sbM_{j}$ imply that}
\begin{align*}
\sup_{\gQ\in\sbM_{j}}\sum_{i=1}^{n}\Var\cro{t_{(\overline P_{i},Q_{i})}(X_{i})}&\le a_{2}\sup_{\gQ\in\sbM_{j}}\cro{\gell(\gP\et,\overline \gP)+\gell(\gP\et,\gQ)}\\
&\le a_{2}\cro{\gell(\gP\et,\overline \gP)+r_{j+1}}.
\end{align*}
{We may therefore apply Proposition~\ref{talagrand} with $v^{2}=a_{2}\cro{\gell(\gP\et,\overline \gP)+r_{j+1}}$ and} $z=z_{j}=(j+1)l+\xi$, $l>0$. Then using \eref{eq-propym102} together with the fact that $y_{0}\ge c_{1}^{-1}$ by~\eref{def-d}, we derive that, with a probability at least $1-2^{-(j+1)}e^{-\xi}$, 
{
\begin{align*}
\gZ_{j}(\bsX,\overline \gP)
&\le (1+c)\E\cro{\gZ_{j}(\bsX,\overline \gP)}+(cv^{2}/8)+2(1+8c^{-1})\cro{(j+1)l+\xi}\\
&\le (1+c)c_{1}r_{j+1}+\frac{ca_{2}}{8}\cro{\gell(\gP\et,\overline \gP)+r_{j+1}}
\\&\quad+2(1+8c^{-1})\frac{lr_{j+1}}{y_{0}}+2(1+8c^{-1})\xi\\
&\le r_{j+1}\cro{(1+c)c_{1}+\frac{ca_{2}}{8}+2(1+8c^{-1})c_{1}l}\\&\quad+\frac{ca_{2}}{8}\gell(\gP\et,\overline \gP) +2(1+8c^{-1})\xi \\
&= Ar_{j+1}+\frac{ca_{2}}{8}\gell(\gP\et,\overline \gP) +2(1+8c^{-1})\xi
\end{align*}
}%
with
\begin{align*}
A=(1+c)c_{1}+\frac{ca_{2}}{8}+2(1+8c^{-1})c_{1}l=c_{1}\pa{1+2l+c+16lc^{-1}}+\frac{a_{2}c}{8}.
\end{align*}
Setting $c=2a_{1}/a_{2}$ and $l=\log2$, we deduce from the definition~\eref{def-c1p} of $c_{1}$ that
\[
\frac{4c_{1}}{a_{1}}=\pa{1+2l+\frac{2a_{1}}{a_{2}}+\frac{8l a_{2}}{a_{1}}}^{-1}=\pa{1+2l+c+16lc^{-1}}^{-1},
\]
hence $A\le a_{1}/2$ and, by~\eref{def-zeta2},
\[
\zeta=(a_{2}c/8)\gell(\gP\et,\overline\gP)+(a_{1}y_{0}/2)+2\pa{1+8c^{-1}}\xi.
\]
It finally follows from (\ref{def-xij}) that, with a probability at least $1-2^{-(j+1)}e^{-\xi}$,
\begin{align*}
\Xi_{j}&\le Ar_{j+1}+(a_{2}c/8)\gell(\gP\et,\overline \gP) +2(1+8c^{-1})\xi -(a_{1}/2)y_{0}j-\zeta\\
&\le Ar_{j+1} -(a_{1}r_{j+1}/2)+(a_{2}c/8)\gell(\gP\et,\overline \gP)+(a_{1}y_{0}/2)+2(1+8c^{-1})\xi -\zeta\\
&\le(a_{2}c/8)\gell(\gP\et,\overline\gP)+(a_{1}y_{0}/2)+2\pa{1+8c^{-1}}\xi-\zeta\le 0.
\end{align*}

\section{Other proofs}\label{sect-opfs}
We shall repeatedly use the following result which is consequence of  Proposition~3.1 in Baraud~\citeyearpar{Bar2016}.
\begin{prop}\label{prop-baraud}
Let $X_{1},\ldots,X_{n}$ be independent random variables with values in $(E,\cE)$ and $\cC$ a VC-class of subsets of $E$ with VC-dimension not larger than $V\ge 1$ {which satisfies $\sum_{i=1}^{n}\P(X_{i}\in C)\le n\sigma^{2}$ for some $\sigma\in (0,1]$ and all $C\in\cC$}.
Then, 
\[
\E\cro{\sup_{C\in\cC}\ab{\sum_{i=1}^{n}\pa{\1_{C}(X_{i})-\P(X_{i}\in C)}}}\le 10\pa{\sigma\vee a}\sqrt{nV\cro{5+\log\pa{\frac{1}{\sigma\vee a}}}}
\]
where 
\[
a=\cro{32\sqrt{\frac{V\wedge n}{n}\log\pa{\frac{2en}{V\wedge n}}}}{\bigwedge} 1.
\]
\end{prop}

\subsection{Proof of Proposition~\ref{prop-lossvar}\label{sect-pfs1}}
The properties of $\ell$ are  straightforward and Assumptions~\ref{Hypo-1}-\ref{cond-1}
and~\ref{cond-4} are direct consequences of Assumptions~\ref{Hypo-classF}-\ref{Hypo-classF1} and~\ref{Hypo-classF2} respectively. Let us now establish~\eref{eq-loss-prty} for some pair $(P,Q)\in\sM^{2}$.  Using~\eref{eq-lossG} and the triangle inequality, we obtain that for all $S\in\overline \sP$
\begin{align}
b\E_{S}\cro{t_{(P,Q)}(X)}&=\int_{E}f_{(P,Q)}\frac{dP+dQ}{2}-\int_{E}f_{(P,Q)}dS\nonumber\\
&=\int_{E}f_{(P,Q)}\frac{dQ-dP}{2}+\int_{E}f_{(P,Q)}dP-\int_{E}f_{(P,Q)}dS\nonumber\\
&\le \ell(S,P)-\frac{1}{2}\ell(P,Q),\label{eq-loss-prtyb}
\end{align}
and the conclusion follows from the triangle inequality. 
\subsection{Proof of Proposition~\ref{lem-Wass}}\label{sect-pfs6}
Let $(P,Q)\in\sP^{2}$ and ${\rm sgn}={\rm sgn}(P,Q,\cdot)={\1_{F_{Q}(\cdot)>F_{P}(\cdot)}-\1_{F_{P}(\cdot)>F_{Q}(\cdot)}}$ be the function corresponding to the sign of $F_{Q}-F_{P}$ on the set $\{F_{Q}\ne F_{P}\}$ and which vanishes elsewhere. We write $f=f_{(P,Q)}$ for short. For all real numbers $0\le x<x'\le 1$, $|f(x)-f(x')|=|\int_{x}^{x'}{\rm sgn}(t)dt|\le x'-x\le 1$. Hence, $f$ belongs to $\cF$ and satisfies {Assumption~\ref{Hypo-classF}-\ref{Hypo-classF2} with $b=1$. Assumption~\ref{Hypo-classF}-\ref{Hypo-classF1} is clearly true and Assumption~\ref{Hypo-classF}-\ref{Hypo-classF3}} derives from the following consequence of Fubini's theorem:
{
\begin{align*}
\lefteqn{\E_{P}\cro{f(X)}-\E_{Q}\cro{f(X)}}\hspace{15mm}\\
=&\int_{0}^{1}\cro{\int_{0}^{1}{\rm sgn}(t)\1_{0\le t< x}dt}dP(x)-\int_{0}^{1}\cro{\int_{0}^{1}{\rm sgn}(t)\1_{0\le t< x}dt}dQ(x)\\
=&\int_{0}^{1}{\rm sgn}(t)(1-F_{P}(t))dt-\int_{0}^{1}{\rm sgn}(t)(1-F_{Q}(t))dt\\
=&\int_{0}^{1}{\rm sgn}(t)\cro{F_{Q}(t)-F_{P}(t)}dt= \int_{0}^{1}\ab{F_{P}(t)-F_{Q}(t)}dt=W(P,Q).
\end{align*}
}%
For the last equality, we refer to Shorack and Wellner~\citeyearpar{MR838963}[Page 64]. 

\subsection{Proof of Corollary~\ref{cor-Wasser}}\label{sect-pfscor2}
Let $f$ be a function on $[0,1]$ that satisfy the following property: there exists a function $f'$ on $[0,1]$ such that $\norm{f'}_{\infty}\le 1$ and 
\[
f(x)=\int_{0}^{x}f'(u)\,du=\int_{0}^{1}f'(u)\1_{x\ge u}\,du\quad \text{for all $x\in [0,1]$.}
\]
Using Fubini's theorem, we obtain that
\begin{align}
\ab{\sum_{i=1}^{n}f(X_{i})-\E\cro{f(X_{i})}}&=\ab{\sum_{i=1}^{n}\cro{\int_{0}^{1}f'(u)\pa{\1_{X_{i}\ge u}-\P\cro{X_{i}\ge u}}du}}\nonumber\\
&=\ab{\int_{0}^{1}f'(u)\sum_{i=1}^{n}\pa{\1_{X_{i}\ge u}-\P\cro{X_{i}\ge u}}du}\nonumber\\
&\le \int_{0}^{1}\ab{\sum_{i=1}^{n}\pa{\1_{X_{i}\ge u}-\P\cro{X_{i}\ge u}}}du.\label{pf-Weq00}
\end{align}
It follows from Proposition~\ref{lem-Wass} that the functions $f_{(P,Q)}$ defined by~\eref{def-f-wass} satisfy this property for all probabilities $P,Q\in\sP$. Hence, by definition~\eref{def-wpbar}
for all $\overline P\in\sM$ and $\overline \gP=\overline P\on$ 
\begin{align*}
\gw(\overline \gP)&\le \E\cro{\sup_{Q\in\sM}\ab{\sum_{i=1}^{n}f_{(\overline P,Q)}(X_{i})-\E\cro{f_{(\overline P,Q)}(X_{i})}}}\\
&\le \int_{0}^{1}\E\cro{\ab{\sum_{i=1}^{n}\pa{\1_{X_{i}\ge u}-\P\cro{X_{i}\ge u}}}}du\le \int_{0}^{1}\sqrt{\sum_{i=1}^{n}\Var(\1_{X_{i}\ge u})}du.
\end{align*}
Hence, $\gw(\overline \gP)\le \sqrt{n}/2$ and by applying  Theorem~\ref{thm-main01} with the values of $a_{0}=3/2$ and $a_{1}=1/2$ provided by Corollary~\ref{cor-wasser} and by using the fact that $\overline P$ is arbitrary in $\sM$, we obtain \eref{eq-cor-Wasser}.
\subsection{Proof of Corollary~\ref{cor-L21}}\label{sect-pfscor6}
As a subset of $\sL_{2}(E,\mu)$, $V$ is also separable and admits an (at most countable) Hilbert basis $(\varphi_{I})_{I\in \cI}$. {It follows from \eref{eq-LinfL20} that}
{
\begin{equation}\label{eq-LinfL2b}
R^{2}\ge\sup_{t\in V,\ \norm{t}_{2}=1}\norm{t}_{\infty}^{2}=\sup_{x\in E}\sup_{\sum_{I}c_{I}^{2}=1}\ab{\sum_{I\in \cI}c_{I}\varphi_{I}(x)}^{2}=\norm{\sum_{I\in \cI}\varphi_{I}^{2}}_{\infty} 
\end{equation}
}%
and {also that} the equality $t=\sum_{I\in \cI}\<t,\varphi_{I}\>\varphi_{I}$ holds both pointwise and in $\L_{2}(E,\mu)$ for all $t\in V$. Given $\overline P=\overline p\cdot\mu$ and $Q=q\cdot\mu$ in $\sM$ with $\overline p,q\in \cM\subset V$, $\overline p\ne q$, we may therefore write for all $x\in E$, $\overline p(x)-q(x)=\sum_{I\in \cI}c_{I}\varphi_{I}(x)$ with $\sum_{I\in\cI} c_{I}^{2}=\|\overline p-q\|_{\mu,2}^{2}>0$. Since  $f_{(\overline P,Q)}=(\overline p-q)/\norm{\overline p-q}_{\mu,2}$ when $P\ne Q$, it follows from Cauchy-Schwarz inequality that
{
\begin{align*}
\lefteqn{\sup_{q\in\cM\setminus\{\overline p\}}\ab{\sum_{i=1}^{n}\pa{\frac{(\overline p-q)(X_{i})}{\norm{\overline p-q}_{\mu,2}}-\E\cro{\frac{(\overline p-q)(X_{i})}{\norm{\overline p-q}_{\mu,2}}}}}}\hspace{30mm}\\
&\le \sup_{(c_{I})_{I\in\cI},\,\sum_{I\in\cI}c_{I}^{2}=1}\sum_{I\in\cI}|c_{I}|\ab{\sum_{i=1}^{n}\pa{\varphi_{I}(X_{i})-\E\cro{\varphi_{I}(X_{i})}}}\\
&\le \sqrt{\sum_{I\in\cI}\ab{\sum_{i=1}^{n}\pa{\varphi_{I}(X_{i})-\E\cro{\varphi_{I}(X_{i})}}}^{2}}.
\end{align*}
}%
We deduce from the definition~\eref{def-wpbar} of $\gw(\overline \gP)$ with $\overline \gP=\overline P\on$ and $t_{(\overline P,Q)}$ given by~\eref{eq-def-phiPQ} together with Jensen's inequality and~\eref{eq-LinfL2b} that
\begin{align*}
\gw(\overline \gP)&\le \frac{1}{2R}\E\cro{\sup_{q\in\cM\setminus\{\overline p\}}\ab{\sum_{i=1}^{n}\pa{\frac{(\overline p-q)(X_{i})}{\norm{\overline p-q}_{\mu,2}}-\E\cro{\frac{(\overline p-q)(X_{i})}{\norm{\overline p-q}_{\mu,2}}}}}}\\
&\le  \frac{1}{2R}\sqrt{\sum_{I\in \cI}\sum_{i=1}^{n} \Var[\varphi_{I}(X_{i})]}\le \frac{1}{2R}\sqrt{\sum_{I\in \cI}\sum_{i=1}^{n} \int_{E}\varphi_{I}^{2}p_{i}\et d\mu}\\
&\le \frac{1}{2R}\sqrt{\sum_{i=1}^{n} \norm{\sum_{I\in\cI}\varphi_{I}^{2}}_{\infty}}\le \frac{\sqrt{n}}{2}.
\end{align*}
Then, we conclude in the same way as for the proof of Corollary~\ref{cor-Wasser}. 
\subsection{Proof of Proposition~\ref{prop-approx-L2}}\label{sect-pfs-prop13}
Let $I=\{0,1\}^{d}\setminus\{(0,\ldots,0)\}$ and consider a multivariate tensor product wavelet basis
\[
\{\bs{\Phi}_{\gk},\bs{\Psi}_{j,\gk}^{\gi},\; \gk\in\Z^{d}, j\ge 0, \gi\in I\}
\]
of $\L_{2}(\R^{d})$ based on the father and mother wavelets $\phi$ and $\psi$ defined on $\R$, with compact support,  regularity $r>\alpha$ and $\L_{2}$-norms equal to 1. This means that, for all $\gx=(x_{1},\ldots,x_{d})\in\R^{d}$, $\gk=(k_{1},\ldots,k_{d})\in\Z^{d}$, $j\ge 0$ and $\gi=(i_{1},\ldots,i_{d})\in I$,
\[
\bs{\Phi}_{\gk}(\gx)=\prod_{l=1}^{d}\phi(x_{l}-k_{l})\quad \text{and}\quad {\bs{\Psi}_{j,\gk}^{\gi}(\gx)=2^{jd/2}\prod_{l=1}^{d}\psi^{(i_{l})}\pa{2^{j}x_{l}-k_{l}},}
\]
with $\varphi^{(1)}=\psi$ and $\varphi^{(0)}=\phi$. If a function $f\in\sL_{2}(\R^{d})$ can be written as  
\begin{equation}\label{eq-decomp-f}
f=\sum_{\gk\in\Z^{d}}\cro{\scal{f}{\bs{\Phi}_{\gk}}\bs{\Phi}_{\gk}+\sum_{j\ge 0}\sum_{\gi\in I}\scal{f}{\bs{\Psi}_{j,\gk}^{\gi}}\bs{\Psi}_{j,\gk}^{\gi}}\quad \text{a.e.}
\end{equation}
and if it belongs to the Besov space $B^{\alpha}_{s,\infty}(\R^{d})$, then the quantitiy
\begin{equation}\label{eq-semi-norm}
\ab{f}_{\alpha,s,\infty}'=\sup_{j\ge 0}2^{j(\alpha+d/2-d/s)}\pa{\sum_{\gk\in\Z^{d},\gi\in I}\ab{\scal{f}{\bs{\Psi}_{j,\gk}^{\gi}}}^{s}}^{1/s}
\end{equation}
is finite and equivalent to the Besov semi-norm $\ab{f}_{\alpha,s,\infty}$ associated to $B^{\alpha}_{s,\infty}(\R^{d})$ (up to constants that depend on $\alpha,s,d,\phi,\psi$). Therefore, replacing $\ab{f}_{\alpha,s,\infty}'$ by $\ab{f}_{\alpha,s,\infty}$ will only change the values of the constants in what follows. 
We refer the reader to Section~4.3 of the book by Nickl and Gin\'e~\citeyearpar{MR3588285} for more details on Besov spaces on $\R^{d}$ and their connections with multivariate tensor product wavelet bases with regularity $r$. 
Since the father and mother wavelets $\varphi,\psi$ have compact support on $\R$, the functions $\bs{\Phi}_{\gk}$ and 
$ \bs{\Psi}_{j,\gk}^{\gi}$ also have compact support on $\R^{d}$ for all $\gk\in\Z^{d}$, $j\ge 0$ and $\gi\in I$. In fact, there exists a number $K_{0}>0$, depending on $d$, $\varphi$ and $\psi$ only, such that for all $\gx=(x_{1},\ldots,x_{d})\in \R^{d}$, $j\ge 0$ the sets
\[
\Lambda(\gx)=\ac{\gk\in\Z^{d}, \ab{\bs{\Phi}_{\gk}(\gx)}>0}\;\text{ and }\;\Lambda_{j}(\gx)=\ac{\gk\in\Z^{d}, \sum_{\gi\in I}\ab{\bs{\Psi}_{j,\gk}^{\gi}(\gx)}>0}
\]
have cardinalities not larger than $K_{0}$. In particular, for $J\ge 0$, the functions $t$ of the form 
%
\begin{equation}\label{eq-def-V-wav}
t(\gx)=\sum_{\gk\in\Z^{d}}\cro{\beta_{\gk,\bs{0}}\bs{\Phi}_{\gk}(\gx)+\sum_{j=0}^{J}\sum_{\gi\in I}\beta_{j,\gk,\gi}\bs{\Psi}_{j,\gk}^{\gi}(\gx)}\quad\text{for all }\gx\in \R^{d}
\end{equation}
{with
%
\begin{equation}\label{eq-def-V-wav'}
\sum_{\gk\in\Z^{d}}\cro{\beta_{\gk,\bs{0}}^{2}+\sum_{j=0}^{J}\sum_{\gi\in I}\beta_{j,\gk,\gi}^{2}}<+\infty
\end{equation}
are well-defined since the series in \eref{eq-def-V-wav} only involves a finite number of non-zero terms and \eref{eq-def-V-wav} implies that they belong to $\L_{2}(\R^{d})$}. We define $V_{J}$ as the linear space of these functions $t$ given by {(\ref{eq-def-V-wav}) and (\ref{eq-def-V-wav'})} and, for all $j\ge 0$, the linear space $U_{j}$ as the space of functions $u$ of the form $u=\sum_{\gk\in\Z^{d}}\sum_{\gi\in I}\beta_{j,\gk,\gi}\bs{\Psi}_{j,\gk}^{\gi}$ with $ \sum_{\gk\in\Z^{d}}\sum_{\gi\in I}\beta_{j,\gk,\gi}^{2}<+\infty$. 
Since the functions $\bs{\Phi}_{\gk}$ and $\bs{\Psi}_{j,\gk}^{\gi}$ form an orthonormal system in $\L_{2}(\R^{d})$ for $\gk\in\Z^{d}$, $j\ge 0$ and $\gi\in I$, the linear spaces $(V_{J},\norm{\cdot}_{\lambda,2})$ and $(U_{j},\norm{\cdot}_{\lambda,2})$ with $j\ge 0$ are Hilbert spaces. Moreover, for all $\gx\in \R^{d}$, 
%
\begin{align*}
\lefteqn{\sum_{\gk\in\Z^{d}}\cro{\bs{\Phi}_{\gk}^{2}(\gx)+\sum_{j=0}^{J}\sum_{\gi\in I}\pa{\bs{\Psi}_{j,\gk}^{\gi}}^{2}(\gx)}}\hspace{30mm}\\
&=\sum_{\gk\in \Lambda(\gx)}\bs{\Phi}_{\gk}^{2}(\gx)+\sum_{j=0}^{J}\cro{\sum_{\gk\in \Lambda_{j}(\gx)}\sum_{\gi\in I}\pa{\bs{\Psi}_{j,\gk}^{\gi}}^{2}(\gx)}\\
&\le K_{0}\cro{\norm{\phi}_{\infty}^{2d}+ 2^{d}\max_{\gi\in I}\norm{\bs{\Psi}_{0,\bs{0}}^{\gi}}_{\infty}^{2}\sum_{j=0}^{J}2^{jd}}\le K_{1}^{2}2^{Jd},
\end{align*}
where $K_{1}$ only depends on $d,\phi$ and $\psi$. It follows from (\ref{eq-def-V-wav}) and Cauchy-Schwarz inequality that, for all $\gx\in \R^{d}$ and $t\in V_{J}$,
%
\begin{align*}
\ab{t(\gx)}^{2}&=\ab{\sum_{\gk\in\Z^{d}}\cro{\beta_{\gk,\bs{0}}\bs{\Phi}_{\gk}(\gx)+\sum_{j=0}^{J}\sum_{\gi\in I}\beta_{j,\gk,\gi}\bs{\Psi}_{j,\gk}^{\gi}(\gx)}}^{2}\\
&\le \cro{\sum_{\gk\in\Z^{d}}\pa{\beta_{\gk,\bs{0}}^{2}+\sum_{j=0}^{J}\sum_{\gi\in I}\beta_{j,\gk,\gi}^{2}}}\cro{\sum_{\gk\in\Z^{d}}\pa{\bs{\Phi}_{\gk}^{2}(\gx)+\sum_{j=0}^{J}\sum_{\gi\in I}\pa{\bs{\Psi}_{j,\gk}^{\gi}}^{2}(\gx)}}\\
&\le \norm{t}_{\lambda,2}^{2}\times  K_{1}^{2}2^{Jd}
\end{align*}
which implies that $V_{J}$ satisfies Assumption~\ref{hypo-Linfty-L2} with $R=K_{1}2^{Jd/2}$. 

For all $\gx\in\R^{d}$ and $t\in U_{j}$ with $j\ge 0$
\begin{align}
\ab{t(\gx)}^{s}&=\ab{\sum_{\gk\in\Z^{d}}\sum_{\gi\in I}\scal{t}{\bs{\Psi}_{j,\gk}^{\gi}}\bs{\Psi}_{j,\gk}^{\gi}(\gx)}^{s}= \ab{\sum_{\gk\in\Lambda_{j}(\gx)}\sum_{\gi\in I}\scal{t}{\bs{\Psi}_{j,\gk}^{\gi}}\bs{\Psi}_{j,\gk}^{\gi}(\gx)}^{s}\nonumber\\
&\le \pa{\ab{\Lambda_{j}(\gx)}\ab{I}}^{s-1}\sum_{\gk\in\Lambda_{j}(\gx)}\sum_{\gi\in I}\ab{\scal{t}{\bs{\Psi}_{j,\gk}^{\gi}}}^{s}\ab{\bs{\Psi}_{j,\gk}^{\gi}(\gx)}^{s}\nonumber\\
&\le (K_{0}2^{d})^{s-1}\sum_{\gk\in\Z^{d}}\sum_{\gi\in I}\ab{\scal{t}{\bs{\Psi}_{j,\gk}^{\gi}}}^{s}\ab{\bs{\Psi}_{j,\gk}^{\gi}(\gx)}^{s}.
\label{eq-tj-wav}
\end{align}
Since, for all $\gi\in I$ and $\gk\in\Z^{d}$, $\|\bs{\Psi}_{j,\gk}^{\gi}\|_{\lambda,s}=2^{jd(1/2-1/s)}\|\bs{\Psi}_{0,\bs{0}}^{\gi}\|_{\lambda,s}$, integrating \eref{eq-tj-wav} with respect to $\gx\in\R^{d}$ leads to the bound,
%
\begin{equation}\label{eq-tj-Vj-p}
\norm{t}_{\lambda,s}\le K_{2}2^{jd(1/2-1/s)}\cro{\sum_{\gk\in\Z^{d}}\sum_{\gi\in I}\ab{\scal{t}{\bs{\Psi}_{j,\gk}^{\gi}}}^{s}}^{1/s}\quad\text{for all }t\in U_{j},
\end{equation}
where $K_{2}$ depends on $d,\phi,\psi$ and $s$.

Let us now consider a function $f$ in $B^{\alpha}_{s,\infty}\cap\sL_{1}(\R^{d})\cap \sL_{2}(\R^{d})$. 
It follows from~\eref{eq-decomp-f} that $f$ can be expanded in the wavelet basis as $\overline f_{J}+\sum_{j>J}f_{j}$ a.e. with $\overline f_{J}\in V_{J}$ and 
\[
f_{j}=\sum_{\gk\in\Z^{d}}\sum_{\gi\in I}\scal{f}{\bs{\Psi}_{j,\gk}^{\gi}}\bs{\Psi}_{j,\gk}^{\gi}\in U_{j}\quad \text{for all $j>J$}.
\]
Since $f$  belongs to $\sL_{1}(\R^{d})$, for all $j\ge 0$
%
\begin{align*}
\sum_{\gk\in\Z^{d}}\sum_{\gi\in I}\ab{\scal{f}{\bs{\Psi}_{j,\gk}^{\gi}}}&\le \int_{\R^{d}}\ab{f(\gx)}\cro{\sum_{\gk\in\Z^{d}}\sum_{\gi\in I}\ab{\bs{\Psi}_{j,\gk}^{\gi}(\gx)}}d\gx\\
&\le K_{0}2^{jd/2}\max_{\gi\in I}\norm{\bs{\Psi}_{0,\bs{0}}^{\gi}}_{\infty}\norm{f}_{\lambda,1}
\end{align*}
and similarly, 
%
\begin{align*}
\sum_{\gk\in\Z^{d}}\ab{\scal{f}{\bs{\Phi}_{\gk}}}&\le K_{0}\norm{\phi}_{\infty}^{d}\norm{f}_{\lambda,1}.
\end{align*}
As a consequence, $\overline f_{J}$ and $f_{j}$ for $j>J$ belong to $\sL_{1}(\R^{d})$ and 
%
\begin{align}
\norm{f_{j}}_{\lambda,1}&=\int_{\R^{d}}\ab{f_{j}(\gx)}d\gx\le \sum_{\gk\in\Z^{d}}\sum_{\gi\in I}\ab{\scal{f}{\bs{\Psi}_{j,\gk}^{\gi}}}\int_{\R^{d}}\ab{\bs{\Psi}_{j,\gk}^{\gi}(\gx)}d\gx\nonumber \\
&=2^{-jd/2}\max_{\gi\in I}\norm{\bs{\Psi}_{0,\bs{0}}^{\gi}}_{\lambda,1}\sum_{\gk\in\Z^{d}}\sum_{\gi\in I}\ab{\scal{f}{\bs{\Psi}_{j,\gk}^{\gi}}}\le K_{3}\norm{f}_{\lambda,1}\label{eq-fj-l1}
\end{align}
where $K_{3}$ depends on $d,\phi$ and $\psi$ only. Besides, since $f$  belongs to $B^{\alpha}_{s,\infty}$ we deduce from~\eref{eq-semi-norm} and~\eref{eq-tj-Vj-p} that  
%
\begin{equation}\label{eq-fj-lp}
\norm{f_{j}}_{\lambda,s}\le K_{4}\ab{f}_{\alpha,s,\infty}2^{-j\alpha}\quad \text{for all $j>J$,}
\end{equation}
where $K_{4}$ depends on $d,\phi,\psi,s$ and $\alpha$. Combining~\eref{eq-fj-l1} and~\eref{eq-fj-lp} and using the fact that $s\ge 2$, we derive that for all $j>J$ and $z_{j}>0$
%
\begin{align*}
\norm{f_{j}}_{\lambda,2}^{2}&=\int_{\R^{d}}f_{j}^{2}(\gx)\1_{|f_{j}|\le z_{j}}d\gx+\int_{\R^{d}}f_{j}^{2}(\gx)\1_{|f_{j}|> z_{j}}d\gx\\
&\le z_{j}\norm{f_{j}}_{\lambda,1}+\frac{\norm{f_{j}}_{\lambda,s}^{s}}{z_{j}^{s-2}}\le z_{j}K_{3}\norm{f}_{\lambda,1}+z_{j}^{-(s-2)}K_{4}^{s}\ab{f}_{\alpha,s,\infty}^{s}2^{-js\alpha}.
\end{align*}
Setting
\[
z_{j}=\left[\frac{K_{4}^{s}\ab{f}_{\alpha,s,\infty}^{s}}{K_{3}\norm{f}_{\lambda,1}}\right]^{1/(s-1)}2^{-js\alpha/(s-1)}\quad\text{when}\quad\norm{f}_{\lambda,1}\ne 0
\]
and letting $z_{j}$ tend to $+\infty$ otherwise, we derive that for all $j>J$
%
\begin{align}
\norm{f_{j}}_{\lambda,2}^{2}
&\le K_{5}\ab{f}_{\alpha,s,\infty}^{s/(s-1)}\norm{f}_{\lambda,1}^{(s-2)/(s-1)}2^{-js\alpha/(s-1)}\label{eq-fj-l2},
\end{align}
where $K_{5}$ only depends on $d,\phi,\psi,\alpha$ and $s$ (with the convention $0^{0}=0$). Since the spaces $U_{j}$ are mutually orthogonal, it follows from~\eref{eq-fj-l2} that 
%
\begin{align*}
\norm{f-\overline f_{J}}_{\lambda,2}^{2}&=\sum_{j>J}\norm{f_{j}}_{\lambda,2}^{2}\le K_{5} \ab{f}_{\alpha,s,\infty}^{s/(s-1)}\norm{f}_{\lambda,1}^{(s-2)/(s-1)}\sum_{j>J}2^{-js\alpha/(s-1)}\\
&\le K_{6}\ab{f}_{\alpha,s,\infty}^{s/(s-1)}\norm{f}_{\lambda,1}^{(s-2)/(s-1)}2^{-Js\alpha/(s-1)}
\end{align*}
where $K_{6}$ depends on $d,\phi,\psi,s$ and $\alpha$, which concludes the proof. 

\subsection{Proof of Corollary~\ref{cor-Lj-histo}\label{Pr-coro6}}
Let $V$ be the linear space spanned by the $D$ indicator functions $\1_{I}$ for $I\in\cI$. Since for all $t=\sum_{I\in\cI}t_{I}\1_{I}\in V$,
\[
\norm{t}_{\mu,j}^{j}=\sum_{I\in\cI}|t_{I}|^{j}D^{-1}\ge D^{-1}\max_{I\in\cI}\ab{t_{I}}^{j}=D^{-1}\norm{t}_{\infty}^{j},
\]
inequality~\eref{eq-LinftyLj} is satisfied with $R=D^{1/j}$. Moreover, given $\overline p,q\in\cM$ with $\overline{p}\ne q$, $(\overline p-q)/\norm{\overline p-q}_{\mu,j}$ writes as $\sum_{I\in\cI} b_{I}\1_{I}$ with 
\begin{equation}\label{cond-b}
1=\norm{\sum_{I\in\cI} b_{I}\1_{I}}_{\mu,j}=\ab{b}_{j}D^{-1/j} \quad \text{with}\quad \ab{b}_{j}=\pa{\sum_{I\in\cI}|b_{I}|^{j}}^{1/j}=R.
\end{equation}
Hence, 
\begin{align*}
\pa{\frac{\overline p-q}{\norm{\overline p-q}_{\mu,j}}}_{+}^{j-1}&=\sum_{I\in \cI} (b_{I})_{+}^{j-1}\1_{I}\quad \text{and}\quad \pa{\frac{\overline p-q}{\norm{\overline p-q}_{\mu,j}}}_{-}^{j-1}=\sum_{I\in \cI} (b_{I})_{-}^{j-1}\1_{I},
\end{align*}
so that, by the definition~\eref{def-fPQ} of $f_{(\overline P,Q)}$
\begin{align*}
f_{(\overline P,Q)}-\E\cro{ f_{(\overline P,Q)}}=\sum_{I\in\cI}\cro{(b_{I})_{+}^{j-1}-(b_{I})_{-}^{j-1}}\pa{\1_{I}-P\et(I)}
\end{align*}
and
\begin{align*}
\lefteqn{\frac{1}{2R^{j-1}}\ab{\sum_{i=1}^{n}\pa{f_{(\overline P,Q)}(X_{i})-\E\cro{ f_{(\overline P,Q)}}}}}\hspace{30mm}\\
&=\frac{1}{2R^{j-1}}\ab{\sum_{I\in\cI}\pa{(b_{I})_{+}^{j-1}-(b_{I})_{-}^{j-1}}\sum_{i=1}^{n}\cro{\1_{I}(X_{i})-P\et(I)}}\\
&\le \frac{1}{2R^{j-1}}\sum_{I\in\cI}|b_{I}|^{j-1}\ab{\sum_{i=1}^{n}\cro{\1_{I}(X_{i})-P\et(I)}}.
\end{align*}
Using~\eref{cond-b} and H\"older inequality with the conjugate exponents $j/(j-1)$ and $j$ we derive that 
\begin{align*}
\lefteqn{\frac{1}{2R^{j-1}}\ab{\sum_{i=1}^{n}\pa{f_{(\overline P,Q)}(X_{i})-\E\cro{ f_{(\overline P,Q)}}}}}\hspace{40mm}\\
&\le \frac{1}{2}\pa{\frac{\ab{b}_{j}}{R}}^{j-1}\cro{\sum_{I\in\cI}\ab{\sum_{i=1}^{n}\cro{\1_{I}(X_{i})-P\et(I)}}^{j}}^{1/j}\\
&=\frac{1}{2}\cro{\sum_{I\in\cI}\ab{\sum_{i=1}^{n}\cro{\1_{I}(X_{i})-P\et(I)}}^{j}}^{1/j}.
\end{align*}
Using Jensen's inequality and~\eref{def-wpbar}, we get
\begin{align}
\gw(\overline \gP)&\le \frac{1}{2} \cro{\sum_{I\in\cI}\E\ab{\sum_{i=1}^{n}\cro{\1_{I}(X_{i})-P\et(I)}}^{j}}^{1/j}.\label{eq-bL2w}
\end{align}
When $j>2$, we may use Theorem~15.10 [Page 442] in Boucheron {\em et al}~\citeyearpar{MR3185193} with $Z=\sum_{i=1}^{n}\1_{I}(X_{i})$ and the fact that $\1_{I}(X_{1}),\ldots,\1_{I}(X_{n})$ are independent nonnegative random variables bounded by 1. We obtain that 
\begin{align}
c_{j}^{-1}\E\ab{\sum_{i=1}^{n}\cro{\1_{I}(X_{i})-P\et(I)}}^{j}&\le 1+\cro{nP\et(I)}^{j/2}\label{eq-bpI}
\end{align}
with 
\begin{equation}\label{def-Cj}
c_{j}=\cro{2^{j-1}(2\kappa j)^{j/2}+(K j)^{j/2}}\vee \cro{2^{j-1}(\kappa j)^{j}},\; \quad \kappa=\frac{\sqrt{e}}{2(\sqrt{e}-1)}
\end{equation}
and $K=(e-\sqrt{e})^{-1}$. Using the inequality below that holds for all $j'\ge 1$
\[
\norm{\overline p_{D}}_{\mu,j'/2}^{j'/2}=\frac{1}{D}\sum_{i\in\cI}\pa{D\int_{I}p\et d\mu}^{j'/2}=D^{j'/2-1}\sum_{i\in\cI}[P\et(I)]^{j'/2}
\]
and the fact that $u\mapsto u^{1/j}$ is sub-additive, by summing~\eref{eq-bpI} over $I\in \cI$, we deduce from~\eref{eq-bL2w} that
\begin{align*}
\gw(\overline \gP)&\le \frac{c_{j}^{1/j}}{2}\cro{D+n^{j/2}\sum_{I\in \cI}[P\et(I)]^{j/2}}^{1/j}\\&= \frac{c_{j}^{1/j}}{2}\cro{D+n^{j/2}D^{1-j/2}\norm{\overline p_{D}}_{\mu,j/2}^{j/2}}^{1/j}\\
&\le \frac{c_{j}^{1/j}}{2}\cro{D^{1/j}+D^{1/j-1/2}\sqrt{n\norm{\overline p_{D}}_{\mu,j/2}}}.
\end{align*}
Since $\overline p_{D}$ is a density, $\mu$ a probability and $j>2$, $1=\norm{\overline p_{D}}_{\mu,1}\le \norm{\overline p_{D}}_{\mu,j/2}$ and consequently for $D\le n$
\[
\gw(\overline \gP)\le c_{j}^{1/j}D^{1/j-1/2}\sqrt{n\norm{\overline p_{D}}_{\mu,j/2}}.
\]
%
%
%
When $j\in (1,2]$, we use Jensen's inequality and get
\begin{align*}
\E\ab{\sum_{i=1}^{n}\cro{\1_{I}(X_{i})-P\et(I)}}^{2(j/2)}&\le \cro{nP\et(I)(1-P\et(I))}^{j/2}\le\cro{nP\et(I)}^{j/2}\
\end{align*}
and, arguing as before, we obtain that
\begin{align*}
\gw(\overline \gP)&\le \frac{\sqrt{n}}{2}\cro{\sum_{I\in\cI}[P\et(I)]^{j/2}}^{1/j}=\frac{D^{1/j-1/2}}{2}\sqrt{n\norm{\overline p_{D}}_{\mu,j/2}}.
\end{align*}
Putting these bounds together we conclude that for all $j>1$, 
\[
\gw(\overline \gP)\le \frac{c_{j}'D^{1/j-1/2}}{2}\sqrt{n\norm{\overline p_{D}}_{\mu,j/2}}\quad\text{with}\quad c'_{j}=\left\{
\begin{array}{ll}2c_{j}^{1/j}&\;\text{when }j>2\\1&\;\text{for }j\in (1,2].
\end{array}\right.
\]
Applying Theorem~\ref{thm-main01} with the constants provided by Corollary~\ref{cor-Lj1}, $R=D^{1/j}$ and using that $\overline p$ is arbitrary in $\cM_{D}$, we obtain that  for all $\xi>0$ and $\overline p\in \cM_{D}$, with a probability at least $1-e^{-\xi}$, 
\begin{align*}
\lefteqn{\norm{p\et-\widehat p}_{\mu,j}}\hspace{10mm}\\&\le 5\inf_{\overline p\in\overline \cM_{D}}\norm{p\et-\overline p}_{\mu,j}+ 4c_{j}'\sqrt{\frac{D}{n}\norm{\overline p_{D}}_{\mu,j/2}}+\frac{4D^{1-1/j}}{\sqrt{n}}\cro{\sqrt{2\xi}+\frac{\epsilon}{\sqrt{n}}}.
\end{align*}
Finally, it follows from~\eref{def-Cj} that, for $j>2$,
\[
4c_{j}'=8c_{j}^{1/j}\le 8\ac{\cro{2^{1-1/j}(2\kappa j)^{1/2}+(K j)^{1/2}}\vee \cro{2^{1-1/j}\kappa j}}=C_{j}.
\]

\subsection{Proof of Corollary~\ref{cor-Linfty}}\label{sect-pfscor7}
It follows from (\ref{eq-tLinfty}) that, for $P,Q\in\sM_{D}$, 
\begin{align*}
\ab{\sum_{i=1}^{n}\pa{t_{(P,Q)}(X_{i})-\E\cro{t_{(P,Q)}}}}&=\ab{\sum_{i=1}^{n}{\cro{\1_{I\et}(X_{i})-P\et(I\et)}}}\\
&\le \max_{I\in\cI}\ab{\sum_{i=1}^{n}{\cro{\1_{I\et}(X_{i})-P\et(I\et)}}}.
\end{align*}
{Hence, \eref{def-wpbar} implies that
\begin{align*}
\gw(\overline \gP)\le \E\cro{\max_{I\in\cI}\ab{\sum_{i=1}^{n}\cro{\1_{I}(X_{i})-P\et(I)}}}\quad\text{for all }\overline P\in\sM_{D}.
\end{align*}
}%
Since the random variables $U_{\eps,I}=\eps\sum_{i=1}^{n}\pa{\1_{I}(X_{i})-P\et(I)}$ with {$\eps\in\{-1,1\}$} and $I\in\cI$ satisfy 
\begin{align*}
\E\cro{e^{\lambda U_{\eps,I}}}\le e^{\lambda^{2}n/8}\quad \text{for all $\lambda>0$,}
\end{align*}
{we derive from Section~6.1.1 in Massart~\citeyearpar{MR2319879} that}
\begin{align*}
\gw(\overline \gP)=\E\cro{\sup_{{\eps\in\{-1,1\},\,} I\in\cI}U_{\eps,I}}\le \sqrt{\frac{n}{2}\log(2D)}.
\end{align*}
We conclude by applying Theorem~\ref{thm-main01} with  $a_{0}=3/(2D)$ and $a_{1}=1/(2D)$.


\subsection{Proof of Corollary~\ref{cor-estimD}}\label{sect-pfscor3}
For all $\gP=P\on,\gQ=Q\on\in\sbM$,
\begin{align*}
\lefteqn{\ab{\overline \gZ\pa{\bsX,\gP,\gQ}}}\hspace{10mm}\\ &\le \frac{1}{2}\left[\ab{\sum_{i=1}^{n}\cro{\1_{q>p}(X_{i})-P_{i}\et(q>p)}}+\ab{\sum_{i=1}^{n}\cro{\1_{p>q}(X_{i})-P_{i}\et(p>q)}}\right].
\end{align*}
Since the classes $\ac{\{q<\overline p\},\; q\in\cM\setminus\{\overline p\}}$ and $\ac{\{q>\overline p\},\; q\in\cM\setminus\{\overline p\}}$ are both VC with dimension not larger than $V=V(\overline p)$, {it follows from Proposition~\ref{prop-baraud} (with $\sigma=1$) that}
\begin{align}
\E\cro{\sup_{q\in\cM\setminus\{\overline p\}}\ab{ \sum_{i=1}^{n}\pa{\1_{\overline p>q}(X_{i})-P_{i}\et(\overline p>q)}}}\le 10\sqrt{5nV},\label{eq-Besup00}
\end{align}
and 
\begin{align}
\E\cro{\sup_{q\in\cM\setminus\{\overline p\}}\ab{ \sum_{i=1}^{n}\pa{\1_{\overline p<q}(X_{i})-P_{i}\et(\overline p<q)}}}
&\le 10\sqrt{5nV}.\label{eq-Besup001}
\end{align}
Consequently, $\gw(\overline \gP)\le 10\sqrt{5nV}$ and the result follows {from} Theorem~\ref{thm-main01} with the constants $a_{0}$ and $a_{1}$ given in Corollary~\ref{cor-TV}.

\subsection{Proof of Lemma~\ref{lem-TV}}\label{sect-pfslem1}
{Let us recall that $P_{m}=\cN(m,I_{d})$ and $p_{m}$ is the corresponding density with respect to the Lebesgue measure so that $p_{m}(x)=p_{0}(x-m)$}. Since the Lebesgue measure is translation invariant, $\norm{P_{m}-P_{m'}}=\norm{P_{m-m'}-P_{0}}$ for all $m,m'\in\R^{d}$ and it suffices to prove the {lemma} for $m'=0$. Let $m\in\R^{d}$. Since the results clearly hold for $m=0$, {let us now consider the case of $m\ne 0$. It follows from (\ref{eq-def-TV}) and (\ref{def-fPQ-TV}) that
\begin{align*}
\norm{P_{m}-P_{0}}&=\frac{1}{2}\int\left[1_{p_{m}>p_{0}}-1_{p_{m}<p_{0}}\right][p_{m}-p_{0}]\,d\mu\\&=\frac{1}{2}\int\left[1_{p_{m}>p_{0}}-1_{p_{m}\le p_{0}}\right][p_{m}-p_{0}]\,d\mu\\&
=\int\left[1_{p_{m}>p_{0}}-\frac{1}{2}\right][p_{m}-p_{0}]\,d\mu=\int1_{p_{m}>p_{0}}[p_{m}-p_{0}]\,d\mu\\&=\int1_{p_{0}(x-m)>p_{0}(x)}[p_{0}(x-m)-p_{0}(x)]\,dx.
\end{align*}
Since $p_{0}(x-m)>p_{0}(x)$ is equivalent to $|x-m|^{2}<|x|^{2}$, we get, denoting by $\bsZ$ a standard normal vector in $\R^{d}$,
\begin{align*}
\norm{P_{m}-P_{0}}&=\int_{|x-m|^{2}<|x|^{2}}p_{0}(x-m)\,dx-\int_{|x-m|^{2}<|x|^{2}}p_{0}(x)\,dx
\\&=\int_{|x|^{2}<|x+m|^{2}}p_{0}(x)\,dx-\int_{|x-m|^{2}<|x|^{2}}p_{0}(x)\,dx\\&=\P\cro{\ab{\bsZ}^{2}<\ab{\bsZ+m}^{2}}-\P\cro{\ab{\bsZ-m}^{2}<\ab{\bsZ}^{2}}\\&=\P\cro{-\ab{m}/2<\scal{\bsZ}{m/\ab{m}}<\ab{m}/2}.
\end{align*}
Since $Z=\scal{\bsZ}{m/\ab{m}}$ is a standard normal variable on $\R$, (\ref{eq-CompTVEuclid0}) follows. To derive (\ref{eq-CompTVEuclid}),} we argue as follows. Clearly, $\norm{P_{m}-P_{0}}\le 1$ and {$p_{0}$ is bounded by $1/\sqrt{2\pi}$. Consequently
\[
\norm{P_{m}-P_{0}}=2\P\cro{0\le Z\le \frac{\ab{m}}{2}}=2\int_{0}^{\ab{m}/2}p_{0}(x)dx\le \frac{|m|}{\sqrt{2\pi}}\bigwedge1,
\]
which leads to the right-hand side of~\eref{eq-CompTVEuclid}. As to the left-hand side, we observe that} the mapping $z\mapsto z^{-1}\int_{0}^{z}p_{0}(x)dx$ being decreasing on $(0,+\infty)$, the minimum of  $m\to  \P\cro{|Z|\le \ab{m}/2}/\min\{1,\ab{m}/\sqrt{2\pi}\}$ is reached for $\ab{m}=\sqrt{2\pi}$ and is not smaller than 0.78. 

\subsection{Proof of Corollary~\ref{cor-estimE}}\label{sect-pfscor4}
Throughout this section we shall identify a vector $\gtheta\in\R^{n}$ with the function on $\cX=\{1,\ldots,n\}\times E$ defined by $(k,x)\mapsto \theta_{k}$ (which is therefore constant with respect to the second argument) and for {convenience we shall denote by the same symbol $\gtheta$ the vector of $\R^{n}$ and the corresponding function on $\cX$}. We consider the class $\cF$ of functions on  $\cX$ which are of the form $q_{\gtheta}:(k,x)\mapsto q(x-\gtheta(k,x))=q(x-\theta_{k})$. The linear space $\Theta$ (viewed as a space of functions on $\cX$) is VC-subgraph with dimension not larger than $d+1$, so is the class of functions of the form $(k,x)\mapsto x-\gtheta(x,k)$ by applying Proposition~42-(i) of Baraud~{\em et al.}~\citeyearpar{MR3595933} with $g:(k,x)\mapsto x$. Since $q$ is unimodal it follows from Proposition~42-(vi) of Baraud~{\em et al.}~\citeyearpar{MR3595933} that $\cF$ is VC-subgraph with dimension not larger than $9.41(d+1)$. Let us fix $\overline \gtheta\in \Theta$. Using Proposition~42-(i) again, we obtain that the class $\ac{q_{\gtheta}-q_{\overline \gtheta},\ \gtheta\in \Theta}$
is VC-subgraph with dimension not larger than $9.41(d+1)$ and the VC-dimensions of the classes (of subsets of $\cX$)
\begin{align*}
\cC_{1}=\ac{\{q_{\gtheta}-q_{\overline \gtheta}>0\},\; \gtheta\in \Theta}\quad\text{and}\quad \cC_{2}=\ac{\{q_{\gtheta}-q_{\overline \gtheta}< 0\},\; \gtheta\in \Theta}
\end{align*}
as well. Applying Proposition~\ref{prop-baraud} (with $\sigma=1$ and $V$ in place of $d$) we obtain that whatever the independent random variables $Y_{1},\ldots,Y_{n}$ with values in $\cX$ and distributions $\widetilde P_{1},\ldots,\widetilde P_{n}$ respectively, 
\begin{align}
\E\cro{\sup_{C\in\cC_{j}}\ab{ \sum_{i=1}^{n}\pa{\1_{C}(Y_{i})-\widetilde P_{i}(C)}}}&\le 10\sqrt{5nV}<69\sqrt{n(d+1)}\label{eq-Besup01}
\end{align}
for all $j\in\{1,2\}$. Let us consider the random variables $Y_{i}=(i,X_{i})$ with distributions  $\widetilde P_{i}\et=\delta_{i}\otimes P_{i}\et$ for all $i\in\{1,\ldots,n\}$, where $\delta_{i}$ denotes the Dirac probability at $i$. For all $\gtheta,\overline \gtheta \in \Theta$
\begin{align*}
\lefteqn{\ab{\sum_{i=1}^{n}\cro{\1_{q_{\theta_{i}}>q_{\overline \theta_{i}}}(X_{i})-\E\cro{\1_{q_{\theta_{i}}>q_{\overline \theta_{i}}}(X_{i})}}}}\quad\\
&= \ab{\sum_{i=1}^{n}\cro{\1_{q_{\gtheta}>q_{\overline \gtheta}}(i,X_{i})-\E\cro{\1_{q_{\gtheta}>q_{\overline \gtheta}}(i,X_{i})}}}\le \sup_{C\in\cC_{1}}\ab{\sum_{i=1}^{n}\pa{\1_{C}(Y_{i})-\widetilde P_{i}\et(C)}}
\end{align*}
and similarly, 
\[
\ab{\sum_{i=1}^{n}\cro{\1_{q_{\theta_{i}}<q_{\overline \theta_{i}}}(X_{i})-\E\cro{\1_{q_{\theta_{i}}<q_{\overline \theta_{i}}}(X_{i})}}}\le \sup_{C\in\cC_{2}}\ab{\sum_{i=1}^{n}\pa{\1_{C}(Y_{i})-\widetilde P_{i}\et(C)}}.
\]
It follows from~\eref{phi-TV0},~\eref{def-barZ},~\eref{eq-Besup01} and~\eref{def-wpbar} that for all $\overline \gtheta\in \Theta$, 
\[
v(\gP_{\overline \gtheta})=\frac{\gw(\gP_{\overline \gtheta})}{\sqrt{n}}\le 69\sqrt{d+1}.
\]
Finally, since by Corollary~\ref{cor-TV} the family $\sT(\ell,\sM)$ satisfies Assumption~\ref{Hypo-1} with $a_{0}=3/2$ and $a_{1}=1/2$, Theorem~\ref{thm-main01} applies and {leads to} \eref{eqcor-estimE}. 

\subsection{Proof of Proposition~\ref{prop-VT2}}\label{sect-pfs2}
Let $S\in\sP$. Let us first prove \eref{cond-3bis00}. Using the definition~\eref{eq-def-TV} of the TV-distance, {we derive that} $S(p> q)\le  \|S-Q\|+Q(p> q)$ and $S(p\le q)\le  \|S-P\|+P(p\le q)$.
{Therefore  \eref{cond-3bis} and the triangle inequality lead} to
\begin{align*}
S(p> q)\wedge S(p\le q)&\le \norm{S-P}\vee \norm{S-Q}+P(p\le q)\wedge Q(p> q)\\
&\le \norm{S-P}+\norm{S-Q}+a_{2}'\norm{P-Q}\\
&\le (1+a_{2}')\cro{\norm{S-P}+\norm{S-Q}}
\end{align*}
and to~\eref{cond-3bis00}. To prove that Assumption~\ref{Hypo-2} is satisfied, it suffices to show that 
\begin{equation}\label{eq-varSphi}
\Var_{S}\cro{t_{(P,Q)}}\le \frac{1}{2}\cro{S(p>q)S(p\le q)+S(q>p)S(q\le p)}
\end{equation}
and to use~\eref{cond-3bis00} with the pairs $(P,Q)$ and $(Q,P)$ successively. It follows from the definition~\eref{phi-TV0} of $t_{(P,Q)}$ that
{\begin{align*}
4\Var_{S}\cro{t_{(P,Q)}}&=4\Var_{S}\cro{f_{(P,Q)}}=\Var_{S}\cro{\1_{q>p}-\1_{p>q}}\\
&\le 2\cro{\Var_{S}\pa{\1_{q>p}}+\Var_{S}\pa{\1_{p>q}}}\\&
= 2\cro{S(q>p)S(q\le p)+S(p>q)S(p\le q)}
\end{align*}
}%
which leads to~\eref{eq-varSphi}.
\subsection{Proof of Corollary~\ref{cor-TV-fast}}\label{sect-pfscor5}
Let us fix $y\ge 0$ and $\overline P=\overline p\cdot \mu\in\sM$. We denote by $\sM(y)$ the subset of $\sM$ gathering those probabilities $Q$ that satisfy $\norm{P\et-Q}\le y/n$, or equivalently for which $\gQ=Q\on$ belongs to the set $\sbB(\gP\et,y)$ defined by~\eref{def-sball} (here $\gell(\gP\et,\gQ)=n\norm{P\et-Q}$ since the data are assumed to be i.i.d.\ with distribution $P\et$). We shall argue as in the proof of Corollary~\ref{cor-estimD} and set 
\[
\cC_{1}=\ac{\{q>\overline p\},\; q\in\cM}\quad \text{and}\quad \overline \cC_{1}=\ac{\{q\le \overline p\},\; q\in\cM}.
\]
Since $\overline \cC_{1}$ gathers the complementary sets of $\cC_{1}$, both classes share the same VC-dimension,  which is not larger than $V$ under the assumption of Corollary~\ref{cor-TV-fast}. Note that for all $Q\in\sM(y)$
{\begin{align*}
\ab{\sum_{i=1}^{n}\cro{\1_{q>\overline p}(X_{i})-P\et(q>\overline p)}}= \ab{\sum_{i=1}^{n}\cro{\1_{q\le \overline p}(X_{i})-P\et(q\le \overline p)}}.
\end{align*}
}%
and if $P\et(q>\overline p)\le 1/2$, we deduce from~\eref{cond-3bis00} that 
\[
{P\et}(q>\overline p)=P\et(q>\overline p)\wedge P\et(q\le \overline p)\le \cro{a_{2}\pa{\norm{P\et-\overline P}+\frac{y}{n}}}{\bigwedge} 1=\sigma^{2}.
\]
Otherwise, $P\et(q\le \overline p)\le 1/2$ and we obtain similarly that 
\[
P\et(q\le \overline p)=P\et(q\le \overline p)\wedge P\et(q> \overline p)\le \cro{a_{2}\pa{\norm{P\et-\overline P}+\frac{y}{n}}}{\bigwedge} 1=\sigma^{2}.
\]
Arguing similarly with the classes 
\[
\cC_{2}=\ac{\{q<\overline p\},\; q\in\cM}\quad \text{and}\quad \overline \cC_{2}=\ac{\{q\ge \overline p\},\; q\in\cM}, 
\]
and applying Proposition~\ref{prop-baraud}, we deduce that 
{\begin{align*}
\lefteqn{\E\cro{\sup_{Q\in\sM(y)}\ab{\overline \gZ(\bsX,\overline \gP,\gQ)}}}\hspace{30mm}\\&\le  \frac{1}{2}\sum_{j\in\{1,2\}}\E\left[\sup_{C\in\cC_{j},\; P\et(C)\le \sigma^{2}}\ab{\sum_{i=1}^{n}\pa{\1_{C}(X_{i})-P\et(C)}}\right]\\
&+\frac{1}{2}\sum_{j\in\{1,2\}}\E\left[\sup_{C\in\overline \cC_{j},\; P\et(C)\le \sigma^{2}}\ab{\sum_{i=1}^{n}\pa{\1_{C}(X_{i})-P\et(C)}}\right]\\
&\le 20\pa{\sigma\vee a}\sqrt{nV\cro{5+\log\pa{\frac{1}{\sigma\vee a}}}}.
\end{align*}
}%
Let us assume in the remaining part of this proof that for some $\lambda>1$ to be chosen later on, 
\begin{equation}\label{choix-y}
\frac{y}{n}\ge \norm{P\et-\overline P}+\frac{\lambda}{a_{2}}\overline a^{2}\quad \text{with}\quad \overline a=32\sqrt{\frac{V}{n}\log\pa{\frac{2en}{V\wedge n}}}.
\end{equation}
\[
a=\cro{32\sqrt{\frac{V\wedge n}{n}\log\pa{\frac{2en}{V\wedge n}}}}{\bigwedge} 1.
\]

Then,  $\sigma\ge \sqrt{a_{2}y/n}\:{\bigwedge1\ge\left(\sqrt{\lambda} \overline a\right)\wedge1\ge a}$ and consequently
\[
\E\cro{\sup_{Q\in\sM(y)}\ab{\overline \gZ(\bsX,\overline \gP,\gQ)}}\le B_{n}(y)=20\sigma\sqrt{nV\cro{5+\log\pa{\frac{1}{\sigma}}}}.
\]
Besides, using the inequalities 
\begin{align*}
\sigma&=\cro{a_{2}\pa{\norm{P\et-\overline P}+\frac{y}{n}}}^{1/2}{\bigwedge} 1\le \sqrt{2a_{2}\frac{y}{n}}{\quad\text{since}\quad\norm{P\et-\overline P}\le\frac{y}{n},}\\
\sigma&\ge{\left(\sqrt{\lambda} \overline a\right)\wedge1}\ge \cro{32\sqrt{\frac{\lambda(V\wedge n)\log(2e)}{n}}}{\bigwedge} 1\ge \sqrt{\frac{V\wedge n}{n}}
\end{align*}
and $\overline a\le \sqrt{a_{2}y/(\lambda n)}$ together with~\eref{choix-y}, we obtain that 
\begin{align*}
B_{n}(y)&\le 20\sqrt{2a_{2}V\cro{5+\frac{1}{2}\log\pa{\frac{n}{V\wedge n}}}}\sqrt{y}\\
&=20\sqrt{a_{2}V\cro{\frac{10}{\log(2e)}\log(2e)+\log\pa{\frac{n}{V\wedge n}}}}\sqrt{y}\\
&\le \frac{20}{32}\sqrt{\frac{10a_{2}}{\log(2e)}}\times \overline a\times \sqrt{ny}\le {\frac{5}{8}}\sqrt{\frac{10 a_{2}^{2}}{\lambda \log(2e)}}\times y<ya_{2}\sqrt{\frac{2.31}{\lambda}}.
\end{align*}
{Setting $\lambda=2.31(a_{2}/c_{1})^{2}$ with $c_{1}$ given by \eref{def-c1p}, we derive that 
%
\begin{equation}\label{maj-w1}
\E\cro{\sup_{Q\in\sM(y)}\ab{\overline \gZ(\bsX,\overline \gP,\gQ)}}\le c_{1}y.
\end{equation}
%
Inequality~\eref{maj-w1} implies that that mapping $y\mapsto \gw(\overline \gP,y)$ defined by~\eref{eq-wpbary} is not larger than $c_{1}y$ provided that $y$ satisfies \eref{choix-y}, hence by definition~\eref{def-d}, 
\begin{align*}
\frac{D(\overline \gP)}{n}\le \pa{\norm{P\et-\overline P}+\frac{\lambda\overline a^{2}}{a_{2}}}\bigvee \frac{1}{nc_{1}}.
\end{align*}
Under the assumptions of Corollary~\ref{cor-TV-fast}, the assumptions of our Theorem~\ref{thm-main02} are satisfied with $a_{0}=3/2$ and $a_{1}=1/2$ and we may therefore apply it. We obtain that for all $\xi>0$, with a probability at least $1-e^{-\xi}$, 
\begin{align}
\norm{P\et-\widehat P}\le\,&13\norm{P\et-\overline P}-\inf_{P\in\sM}\norm{P\et-P}+\frac{2D(\overline\gP)}{n}+\frac{16(1+8a_{2})}{n}\xi+\frac{4\epsilon}{n}\label{eq-VC007}\\
\le\,&15\norm{P\et-\overline P}-\inf_{P\in\sM}\norm{P\et-P}+2\left[\frac{\lambda \overline a^{2}}{a_{2}}\bigvee\frac{1}{nc_{1}}\right]\nonumber\\&+\frac{16(1+8a_{2})}{n}\xi+\frac{4\epsilon}{n}.\nonumber
\end{align}
Let us now observe that
\[
\frac{\lambda \overline a^{2}}{a_{2}}=2.31\frac{a_{2}}{c_{1}^{2}}\times2^{10}\frac{V}{n}\log\pa{\frac{2en}{V\wedge n}}\ge2.31\times2^{10}\log(2e)\frac{1}{nc_{1}}>\frac{1}{nc_{1}},
\]
since $a_{2}\wedge c_{1}^{-1}\wedge V\ge1$ by Proposition~\ref{prop-VT2} and \eref{def-c1p}, which also imply that $16\pa{1+8a_{2}}\le 144 a_{2}$ and
\begin{align*}
c_{1}^{-1}=4\cro{2(1+\log 4)+\frac{2}{a_{2}}+32a_{2}\log 2}<116a_{2}.
\end{align*}
Hence, 
\begin{align*}
2\left[\frac{\lambda \overline a^{2}}{a_{2}}\bigvee\frac{1}{nc_{1}}\right]
&=2.31\times2^{11}\frac{a_{2}}{c_{1}^{2}}\frac{V}{n}\log\pa{\frac{2en}{V\wedge n}}<144ca_{2}^{3}\frac{V}{n}\log\pa{\frac{2en}{V\wedge n}}
\end{align*}
with $c=4.5\times 10^{5}$. We finally deduce~\eref{eq-cor-TV-fast} from~\eref{eq-VC007} and the facts that $\overline P$ is arbitrary in $\sM$ and $\epsilon\le 35$}.
\subsection{Proof of Proposition~\ref{prop-hel}}\label{sect-pfs3}
For all $x\in E$ and $g\in\cG$
\[
\sqrt{p(x)q(x)}\le \frac{1}{2}\cro{g(x)p(x)+(1/g(x))q(x)}
\]
with the conventions $(+\infty)\times 0=0$ and $(+\infty)\times a=(+\infty)$ for all $a>0$. Note that equality holds for $g=g_{(P,Q)}=\sqrt{q/p}$ with our conventions. 
Integrating with respect to $\mu$ gives
\[
\int_{E}\sqrt{pq}d\mu=1-h^{2}(P,Q)\le \frac{1}{2}\cro{\E_{P}(g)+\E_{Q}(1/g)}\in [0,+\infty].
\]
{Consequently} for all $g\in\cG$
\[
h^{2}(P,Q)\ge \frac{1}{2}\cro{\E_{P}(1-g)+\E_{Q}(1-1/g)}\in [-\infty,1]
\]
with equality for $g=g_{(P,Q)}$, which leads to the result.

\subsection{Proof of Proposition~\ref{prop-hel2}}\label{sect-pfs4}
Let us set $\overline \phi_{(P,Q)}=t_{(P,Q)}\sqrt{2}$ and denote by $r=(p+q)/2$ the density of $R$ with respect to $\mu$. {Since} $(p\vee q)/r\le 2$, for all $x,x'\in E$ 
\[
\overline \phi_{(P,Q)}(x)-\overline \phi_{(P,Q)}(x')\le \frac{1}{2}\cro{\sqrt{\frac{q}{r}}(x)+\sqrt{\frac{p}{r}}(x')}\le \sqrt{2},
\]
{hence $t_{(P,Q)}=\overline \phi_{(P,Q)}/\sqrt{2}$} takes its values in $[-1,1]$.
For $T=t\cdot \mu\in\{P,Q\}$, we set 
\[
\rho_{r}(S,t)=\frac{1}{2}\cro{\rho(R,T)+\E_{S}\pa{\sqrt{\frac{t}{r}}(X)}},
\]
so that 
%
\begin{align*}
\E_{S}\cro{\overline \phi_{(P,Q)}(X)}&=\rho_{r}(S,q)-\rho_{r}(S,p)\\
&=\rho_{r}(S,q)-\rho(S,Q)+\rho(S,Q)-\rho(S,P)+\rho(S,P)-\rho_{r}(S,p).
\end{align*}
By Proposition~1 of Baraud~\citeyearpar{MR2834722} (which requires that $S\ll \mu$)
{\[
0\le \rho_{r}(S,t)-\rho(S,T)\le [h^{2}(S,P)+h^{2}(S,Q)]/\sqrt{2}\quad\text{for all }T\in\{P,Q\}
\]
and, since $\rho(S,Q)-\rho(S,P)=h^{2}(S,P)-h^{2}(S,Q)$, we deduce that 
\begin{align*}
\E_{S}\cro{\overline \phi_{(P,Q)}(X)}&\le \frac{1}{\sqrt{2}}\cro{h^{2}(S,P)+h^{2}(S,Q)}+h^{2}(S,P)-h^{2}(S,Q)\\
&\le \pa{1+\frac{1}{\sqrt{2}}}h^{2}(S,P)-\pa{1-\frac{1}{\sqrt{2}}}h^{2}(S,Q).
\end{align*}
Hence \ref{cond-2} is satisfied with $a_{0}=(\sqrt{2}+1)/2$ and $a_{1}=(\sqrt{2}-1)/2$. {Since} 
\[
4\Var_{S}\cro{\overline \phi_{(P,Q)}(X)}=\Var_{S}\cro{\frac{\sqrt{p}-\sqrt{q}}{\sqrt{r}}(X)}\le \E_{S}\cro{\frac{\pa{\sqrt{p(X)}-\sqrt{q(X)}}^{2}}{r(X)}},
\]
condition \ref{cond-3} with $a_{2}=3/2$ follows from the proof of Proposition~3 of  Baraud~\citeyearpar{MR2834722}.

\subsection{Proof of Proposition~\ref{prop-KL}}\label{sect-pfs5}
It is clear from the definition~\eref{T-KL} that $t_{(P,Q)}=-t_{(Q,P)}$ and that under~\eref{eq-maj-min} $t_{(P,Q)}(x)-t_{(P,Q)}(x')\le 1$ for all $x,x'\in E$. Using the definition of the Kullback-Liebler divergence and {the assumptions that $\int_{E}s\ab{\log s}d\mu<+\infty$, $|\log(dP/d\mu)|\in\L_{\infty}(E,\mu)$ and $|\log(dQ/d\mu)|\in\L_{\infty}(E,\mu)$}, we obtain that
\begin{align*}
\E_{S}\cro{t_{(P,Q)}}&=\frac{1}{2a}\E_{S}\cro{\log\pa{\frac{q}{p}}}=\frac{1}{2a}\E_{S}\cro{\log\pa{\frac{s}{p}}-\log\pa{\frac{s}{q}}}\\
&=\frac{1}{2a}\cro{K(S,P)-K(S,Q)}.
\end{align*}
Assumption~\ref{Hypo-1} is therefore satisfied with $a_{0}=a_{1}=1/(2a)$. The proof of Assumption~\ref{Hypo-2} relies on the following lemma.

\begin{lem}\label{lem-prop12}
Let $a>0$. For all $u,v\in\R$ such that $|u-v|\le a$
\[
(u-v)^{2}\le \frac{2a}{\tanh(a/2)}\cro{e^{u}-1-u+e^{v}-1-v}.
\]
\end{lem}

For a point $x\in E$ such that $s(x)>0$, let $u=\log(p(x)/s(x))$ and $v=\log(q(x)/s(x))$. Since 
\[
|u-v|=\ab{\log\pa{\frac{p(x)}{q(x)}}}\le a
\]
we may apply Lemma~\ref{lem-prop12} and get 
\begin{align*}
\log^{2}\pa{\frac{p(x)}{q(x)}}&\le \frac{2a}{\tanh(a/2)}\cro{\frac{p(x)}{s(x)}-1-\log\pa{\frac{p(x)}{s(x)}}+\frac{q(x)}{s(x)}-1-\log\pa{\frac{q(x)}{s(x)}}}.
\end{align*}
Integrating this inequality with respect $S$ gives
\begin{align*}
\Var_{S}\cro{t_{(P,Q)}(X)}&\le \E_{S}\cro{t_{(P,Q)}^{2}(X)}\le \frac{2a}{\tanh(a/2)}\cro{K(S,P)+K(S,Q)}
\end{align*}
{which} proves that Assumption~\ref{Hypo-2}-\ref{cond-3} is satisfied with $a_{2}=(2a)/\tanh(a/2)$. 

Let us now turn to the proof of Lemma~\ref{lem-prop12}. 
The mapping 
\[
z\mapsto \frac{d }{dz}\cro{2\log\cosh(z/2)}=\tanh(z/2)=\frac{e^{z}-1}{e^{z}+1}
\]
is concave on $[0,a]$, hence above its chord, which leads to the inequalities
\[
\frac{d }{dz}\cro{2\log\cosh(z/2)}\ge \frac{\tanh(a/2)}{a}z\quad \text{for all $z\in [0,a]$}
\]
and, by integration, 
\begin{equation}\label{eq-prop12-1}
2\log\cosh(t/2)=\int_{0}^{t}\tanh(z/2)\,dz\ge  \frac{\tanh(a/2)}{2a}t^{2} \quad \text{for all $t\in [0,a]$.}
\end{equation}
{The above inequality is actually} also true for all $t\in [-a,a]$ since the mapping $t\mapsto 2\log\cosh(t/2)$ is even. 

For $u,v\in\R$ such that $|u-v|\le a$, let us set $t=v-u\in [-a,a]$ so that
\begin{align*}
e^{u}-1-u+e^{v}-1-v=e^{u}(1+e^{t})-2(1+u)-t=f_{t}(u).
\end{align*}
For a fixed value of $t$, the mapping $f_{t}$  is differentiable on $\R$, 
tends to $+\infty$ when $u$ goes to $\pm \infty$ and satisfies $f_{t}'(u)=e^{u}(1+e^{t})-2$ for all $u\in\R$. The derivative only vanishes at the point  
\[
u_{t}\et=-\log\pa{\frac{e^{t}+1}{2}}
\]
which is therefore the location of the unique minimum of $f_{t}$ on $\R$. We deduce from~\eref{eq-prop12-1} that for all $u\in \R$
\begin{align*}
f_{t}(u)\ge f_{t}(u_{t}\et)=2\log\pa{\frac{e^{t}+1}{2}}-t={2\log\left(\cosh\left(\frac{t}{2}\right)\right)}\ge \frac{\tanh(a/2)}{2a}t^{2}
\end{align*}
{which proves} the lemma. 


\subsection{Proof of Proposition~\ref{prop-med}}\label{pf-propTV-mediane}
Let $\theta,\theta'\in \Q$, $\theta\ne \theta'$. Since $f$ is decreasing on $(0,+\infty)$, for all $x\in \R\setminus\{\theta,\theta'\}$, 
\begin{align}
p_{\theta'}(x)>p_{\theta}(x)&\iff f(|x-\theta'|)>f(|x-\theta|)\iff |x-\theta'|<|x-\theta|\nonumber\\
&\iff 
\begin{cases}
x>(\theta+\theta')/2 \quad \text{if}\quad \theta'>\theta,\\
x<(\theta+\theta')/2 \quad \text{if}\quad \theta'<\theta
\end{cases}
\label{eq-ensem}
\end{align}
and 
\begin{align}
p_{\theta'}(x)=p_{\theta}(x)&\iff f(|x-\theta'|)=f(|x-\theta|)\iff |x-\theta'|=|x-\theta|\nonumber\\
&\iff x=\frac{\theta+\theta'}{2}\in\Q.\label{eq-ensem1}
\end{align}
By the symmetry of $p$ and the change of variables $u=\theta+\theta'-x$, i.e.\ $x=\theta+\theta'-u$, we derive that
\begin{align*}
P_{\theta}\cro{p_{\theta}>p_{\theta'}}&=\int_{\R}\1_{p(x-\theta)>p(x-\theta')}p(x-\theta)dx=\int_{\R}\1_{p(\theta'-u)>p(\theta-u)}p(\theta'-u)du\\
&=\int_{\R}\1_{p(u-\theta')>p(u-\theta)}p(u-\theta')du=P_{\theta'}\cro{p_{\theta'}>p_{\theta}}.
\end{align*}
Using the expression~\eref{phi-TV0} of $t_{(P,Q)}$ and the fact that with probability 1 none of the $X_{i}$ belongs to $\Q$, we deduce that for all $i\in\{1,\ldots,n\}$ 
\begin{align*}
t_{(P_{\theta},P_{\theta'})}(X_{i})&=\frac{1}{2}\cro{\1_{p_{\theta'}>p_{\theta}}(X_{i})-\1_{p_{\theta}>p_{\theta'}}(X_{i})}=\frac{1}{2}-\1_{p_{\theta}>p_{\theta'}}(X_{i})\quad \text{a.s.}
\end{align*}
It then follows from~\eref{eq-ensem} that, a.s. 
\begin{align*}
\gT(\bsX,\gP_{\theta},\gP_{\theta'})&=\frac{n}{2}-\sum_{i=1}^{n}\1_{p_{\theta}>p_{\theta'}}(X_{i})= \frac{n}{2}-\begin{cases}
\sum_{i=1}^{n}\1_{X_{i}>(\theta+\theta')/2}\quad \text{if}\quad \theta>\theta',\\
\sum_{i=1}^{n}\1_{X_{i}<(\theta+\theta')/2}\quad \text{if}\quad \theta<\theta'.
\end{cases}
\end{align*}
Let us now take $\theta=\widehat \theta\in (X_{(\lceil n/2\rceil)},X_{(\lceil n/2\rceil+1)})$. It follows from~\eref{eq-mediane} that, if $\theta'>\widehat\theta$, 
\[
\sum_{i=1}^{n}\1_{X_{i}<(\widehat\theta+\theta')/2}\ge \sum_{i=1}^{n}\1_{X_{i}\le \widehat\theta}\ge \sum_{i=1}^{n}\1_{X_{i}\le X_{(\lceil n/2\rceil)}}\ge \frac{n}{2}
\]
and consequently, $\gT(\bsX,\gP_{\widehat \theta},\gP_{\theta'})\le 0$. 
If now $\theta'<\widehat\theta$ we may distinguish between two cases. Since $(\theta'+\widehat\theta)/2\in\Q$,
\[
\text{either} \quad \frac{\theta'+\widehat\theta}{2}<X_{(\lceil n/2\rceil)}<\widehat\theta \quad \text{or}\quad X_{(\lceil n/2\rceil)}<\frac{\theta'+\widehat\theta}{2}<\widehat\theta< X_{(\lceil n/2\rceil+1)}.
\]
In the first case, we derive from~\eref{eq-mediane} again that 
\begin{align*}
\sum_{i=1}^{n}\1_{X_{i}>(\theta'+\widehat \theta)/2}&=n-\sum_{i=1}^{n}\1_{X_{i}\le (\theta'+\widehat \theta)/2}\ge n-\sum_{i=1}^{n}\1_{X_{i}<X_{(\lceil n/2\rceil)}}>\frac{n}{2}
\end{align*}
hence, $\gT(\bsX,\gP_{\widehat \theta},\gP_{\theta'})<0$. In the second case, 
\begin{align*}
\sum_{i=1}^{n}\1_{X_{i}>(\theta'+\widehat \theta)/2}=n-\sum_{i=1}^{n}\1_{X_{i}\le (\theta'+\widehat \theta)/2}
=n-\left\lceil \frac{n}{2}\right\rceil\ge \frac{n-1}{2}
\end{align*}
which implies that $\gT(\bsX,\gP_{\widehat \theta},\gP_{\theta'})\le 1/2$. 

Putting all these bounds together, we finally obtain that 
\[
\gT(\bsX,\gP_{\widehat \theta})=\sup_{\theta'\in \Q}\gT(\bsX,\gP_{\widehat \theta},\gP_{\theta'})\le \frac{1}{2}\le \inf_{\theta\in\Q}\sup_{\theta'\in \Q}\gT(\bsX,\gP_{\theta},\gP_{\theta'})+\frac{1}{2}.
\]
Hence $\widehat \theta$ is a TV-estimator for the choice $\epsilon=1/2$.

\subsection{Proof of Proposition~\ref{prop-TestRob}}\label{sect-pfs12}
Let 
\begin{align*}
z=a_{1}\gell(\gP\et,\gQ)-a_{0}\gell(\gP\et,\gP)=a_{1}\gell(\gP\et,\gQ)(1-\gamma)>0.
\end{align*}
By~\eref{eq-Estat2} and Assumption~\ref{Hypo-1}-\ref{cond-1}, $z\le-\E\cro{\gT(\bsX,\gP,\gQ)}$ and we derive from (\ref{def-Phi}) that 
\begin{align*}
\P\cro{\Phi_{(\gP,\gQ)}(\bsX)=1}&\le\P\cro{\gT(\bsX,\gP,\gQ)\ge0}\\&\le
\P\cro{\gT(\bsX,\gP,\gQ)-\E\cro{\gT(\bsX,\gP,\gQ)}\ge z}.
\end{align*}
The variable
\[
\gT(\bsX,\gP,\gQ)-\E\cro{\gT(\bsX,\gP,\gQ)}=\sum_{i=1}^{n}\left(t_{(P_{i},Q_{i})}(X_{i})-\E\cro{t_{(P_{i},Q_{i})}(X_{i})}\right)
\]
is a sum of $n$ independent centred random variables and it follows from Assumption~\ref{Hypo-1}-\ref{cond-4} that $t_{(P_{i},Q_{i})}(X_{i})$ takes its values in an interval of length not larger than 1 for all $i\in\{1,\ldots,n\}$. 
We may apply Hoeffding's inequality, which gives $\P\cro{\Phi_{(\gP,\gQ)}(\bsX)=1}\le\exp\cro{-2z^{2}/n}$ {and} proves (\ref{prop-TestRob1a}). 

When Assumption~\ref{Hypo-2}-\ref{cond-3} is satisfied we proceed in the same way, replacing Hoeffding's inequality by Bernstein's (see inequality (2.16) in Massart~\citeyearpar{MR2319879}). {If we apply} this inequality to the independent random variables 
\[
t_{(P_{i},Q_{i})}(X_{i})-\E\cro{t_{(P_{i},Q_{i})}(X_{i})}\le 1=b\quad \text{for all $i\in\{1,\ldots,n\}$}
\]
{and take into account that the sum of their second moments is not} larger than 
\[
v=a_{2}\left[\gell(\gP\et,\gQ)+\gell(\gP\et,\gP)\right]=a_{2}\gell(\gP\et,\gQ)\pa{1+\frac{a_{1}\gamma}{a_{0}}},
\]
we derive that
{\begin{align*}
\P\cro{\Phi_{(\gP,\gQ)}(\bsX)=1}&\le\exp\left[-\frac{1}{2}\frac{z^{2}}{v+(bz/3)}\right]\\&\le
\exp\left[-\frac{\gell(\gP\et,\gQ)}{2}\frac{a_{1}(1-\gamma)^{2}}{[(1-\gamma)/3]+[(1+\gamma)a_{2}/a_{1}]}\right],
\end{align*}
}
which is \eref{prop-TestRob2a}.

\subsection{Proof of Proposition~\ref{prop-test-casVar}}\label{sect-pfs21}
Let us set 
\[
z=\frac{1}{b}\cro{\frac{1}{2}\ell(P,Q)-\ell(P\et,P)}\ge\frac{1/2-\kappa}{b}\ell(P,Q)>0.
\]
It follows from~\eref{eq-loss-prtyb} that $-\E\cro{t_{(P,Q)}(X)}\ge z$ and the rest of the proof is similar to that of~\eref{prop-TestRob1a} in Proposition~\ref{prop-TestRob}.

\paragraph{\bf Acknowledgement} The author would like to thank the two referees as well as Lucien Birg\'e for their many questions and comments which helped to improve this paper. 

\bibliographystyle{apalike}

\end{document}